\documentclass[11pt,a4paper,headings=optiontohead,dvipsnames,enabledeprecatedfontcommands]{scrartcl}
\ifx\pdftexversion\undefined\else
  \pdfoutput=1  % for arXiv, to make sure it uses pdflatex
\fi

\usepackage[english]{babel}

\usepackage{geometry}
\usepackage{fancyhdr}
\usepackage[singlespacing]{setspace}

\usepackage{xspace,xpatch,authblk}

\usepackage[toc, title]{appendix}

\usepackage{natbib}
\DeclareOldFontCommand{\tt}{\normalfont\ttfamily}{\mathtt}
\shortcites{cerezo2021,ebadi2021,ebadi2022,abbas2023,scip,wurtz2023}
\makeatletter
\xpretocmd{\NAT@citex}{\ifNAT@swa\else\fi}
\makeatother

\usepackage{amsmath,amssymb,amsthm,braket,mathtools}
\theoremstyle{plain}
\newtheorem{lemma}{Lemma}

\usepackage[group-separator={,},group-minimum-digits=4]{siunitx}

\usepackage{glossaries}

\usepackage[inline]{enumitem}

\PassOptionsToPackage{svgnames,dvipsnames}{xcolor}
\usepackage{graphicx}
\graphicspath{{./img/}{./comps/figures/}}
\usepackage{subcaption}
\usepackage[mode=buildnew]{standalone}
\usepackage{tikz}
\usetikzlibrary{calc}
\usetikzlibrary{shapes.geometric}

\usepackage{multirow,makecell,tabularx,booktabs}
\usepackage{pdflscape,afterpage}
\usepackage[para]{threeparttable}
\usepackage{hyperref}
\usepackage[capitalise, noabbrev]{cleveref}
\hypersetup{
  colorlinks,
  linkcolor={red!50!black},
  citecolor={blue!65!black},
  urlcolor={blue!80!black},
  pdftitle={Quantum computing for discrete optimization: a highlight of three technologies}
}

\usepackage{xr}
\usepackage{zref}
\usepackage{zref-clever}
\usepackage{zref-xr}
\zcsetup{cap,hyperref=false}

\newcounter{steps}
\newcommand{\defeq}{\coloneqq}

\newcommand{\NP}{\mathcal{NP}}

\newcommand{\ie}{i.e.\@\xspace}
\newcommand{\eg}{e.g.\@\xspace}
\newcommand{\Eg}{E.g.\@\xspace}
\newcommand{\subjecto}{\textrm{s.\,t.}} % \st already defined in soul package

\newtheorem{remark}{Remark}

\newcommand{\formatabbreviation}[1]{\emph{#1}}
\newglossaryentry{IBM}{name={\formatabbreviation{IBM}}, description={}}
\newglossaryentry{QuEra}{name={\formatabbreviation{QuEra}}, description={}}
\newglossaryentry{RymaxOne}{name={\formatabbreviation{RymaxOne}}, description={}}
\newglossaryentry{D-Wave}{name={\formatabbreviation{D-Wave}}, description={}}
\newglossaryentry{Google}{name={\formatabbreviation{Google}}, description={}}
\newglossaryentry{IQM}{name={\formatabbreviation{IQM}}, description={}}
\newglossaryentry{Rigetti}{name={\formatabbreviation{Rigetti}}, description={}}
\newglossaryentry{Xanadu}{name={\formatabbreviation{Xanadu}}, description={}}
\newglossaryentry{Pasqal}{name={\formatabbreviation{Pasqal}}, description={}}
\newglossaryentry{Quantinuum}{name={\formatabbreviation{Quantinuum}}, description={}}
\newglossaryentry{Honeywell}{name={\formatabbreviation{Honeywell QS}}, description={}}
\newglossaryentry{IonQ}{name={\formatabbreviation{IonQ}}, description={}}
\newglossaryentry{QuEra-OPT}{name={\mbox{NA-OPT}}, description={}}
\newglossaryentry{D-Wave-OPT}{name={\mbox{QA-OPT}}, description={}}
\newglossaryentry{IBM-OPT}{name={\mbox{QAOA-OPT}}, description={}}
\newglossaryentry{IBM-SIM-OPT}{name={\mbox{SIM-OPT}}, description={}}
\newglossaryentry{Gurobi}{name={\formatabbreviation{Gurobi}}, description={}}
\newglossaryentry{Aquila}{name={\formatabbreviation{Aquila}}, description={}}
\newglossaryentry{Advantage}{name={\formatabbreviation{Advantage}}, description={}}
\newglossaryentry{IBMCusco}{name={\formatabbreviation{ibm\_cusco}}, description={}}
\newglossaryentry{IBMNazca}{name={\formatabbreviation{ibm\_nazca}}, description={}}

\newacronym{DFJ}{DFJ}{Dantzig-Fulkerson-Johnson}
\newacronym{MTZ}{MTZ}{Miller-Tucker-Zemlin}
\newacronym{QAP}{QAP}{quadratic assignment problem}
\newacronym{LBOP}{LBOP}{linear binary optimization problem}
\newacronym{OR}{OR}{Operations Research}
\newacronym{QC}{QC}{quantum computer}
\newacronym{QPU}{QPU}{quantum processing unit}
\newacronym{CPU}{CPU}{central processing unit}
\newacronym{GPU}{GPU}{graphics processing unit}
\newacronym{NISQ}{NISQ}{noisy intermediate-scale quantum}
\newacronym{FTQC}{FTQC}{fault-tolerant quantum computing}
\newacronym{QEC}{QEC}{quantum error correction}
\newacronym{MIP}{MIP}{mixed integer program}
\newglossaryentry{QUBO}{
 name={QUBO},
 short={QUBO},
 description={quadratic unconstrained binary optimization problem},
 first={quadratic unconstrained binary optimization problem},
 plural={QUBOs},
 firstplural={quadratic unconstrained binary optimization problems (QUBOs)}
}
\newglossaryentry{PUBO}{
	name={PUBO},
	short={PUBO},
	description={polynomial unconstrained binary optimization problem},
	first={polynomial unconstrained binary optimization problem},
	plural={PUBOs},
	firstplural={polynomial unconstrained binary optimization problems (PUBOs)}
}
\newglossaryentry{TSP}{
 name={TSP},
 short={TSP}, 
 description={traveling salesperson problem},
 first={traveling salesperson problem (TSP)},
 plural={TSPs}
}
\newglossaryentry{MIS}{
	name={MIS},
	short={MIS},
	description={maximum independent set problem},
	first={maximum independent set problem (\glsentrytext{MIS})},
	plural={MIS},
	descriptionplural={maximum independent set problem},
	firstplural={maximum independent set problems (\glsentryplural{MIS})}
	}
\newglossaryentry{UD-MIS}{
	name={UD-MIS},
	short={UD-MIS},
	description={maximum independent set problem on a unit disk graph},
	first={maximum independent set problem on a unit disk graph (\glsentrytext{UD-MIS})},
	plural={UD-MIS},
	descriptionplural={maximum independent set problems on unit disk graphs},
	firstplural={maximum independent set problems on unit disk graphs (\glsentryplural{UD-MIS})}
	}
\newglossaryentry{MWIS}{
	name={MWIS},
	short={MWIS},
	description={maximum-weight independent set problem},
	first={maximum-weight independent set problem (\glsentrytext{MWIS})},
	plural={MWIS},
	descriptionplural={maximum-weight independent set problem},
	firstplural={maximum-weight independent set problem (\glsentryplural{MWIS})}
	}
\newglossaryentry{UD-MWIS}{
	name={UD-MWIS},
	short={UD-MWIS},
	description={maximum-weight independent set problem on a unit disk graph},
	first={maximum-weight independent set problem on a unit disk graph (\glsentrytext{UD-MWIS})},
	plural={MIS},
	descriptionplural={maximum-weight independent set problems on unit disk graphs},
	firstplural={maximum-weight independent set problems on unit disk graphs (\glsentryplural{UD-MWIS})}
	}
\newacronym{MaxCut}{MaxCut}{weighted maximum cut problem}
\newacronym{ILP}{ILP}{integer linear program}
\newacronym{AQC}{AQC}{adiabatic quantum computing}
\newacronym{GQC}{GQC}{gate-based quantum computing}
\newacronym{SPSA}{SPSA}{simultaneous perturbation stochastic approximation}
\newacronym{QAA}{QAA}{quantum adiabatic algorithm}
\newacronym{QA}{QA}{quantum annealing}
\newacronym{VQA}{VQA}{variational quantum algorithm}
\newacronym{QAOA}{QAOA}{quantum approximate optimization algorithm}
\usepackage[svgnames,dvipsnames]{xcolor}
\usepackage{todonotes}  % add [disable] option to disable all notes

\definecolor{notecolorbg}{HTML}{ffe0b3}
\definecolor{notecolorfg}{HTML}{b30000}

\setuptodonotes{backgroundcolor=notecolorbg,textcolor=notecolorfg,size=\small}

\input{review-clean}

\zexternaldocument{supplement_gen}

\zcLanguageSetup{english}{
  type = suppsection,
    Name-sg = Supplement,
    name-sg = supplement,
    Name-pl = Supplements,
    name-pl = supplements,
}

\author[a]{Alexey Bochkarev\thanks{Corresponding author: Alexey Bochkarev (\href{mailto:a.bochkarev@rptu.de}{a.bochkarev@rptu.de})}}
\author[b]{Raoul~Heese}
\author[a]{Sven~Jäger}
\author[c]{Philine~Schiewe}
\author[a,b]{Anita~Schöbel}
\date{\today}

\affil[a]{Department of Mathematics, RPTU Kaiserslautern-Landau, Kaiserslautern, 67663, Germany}
\affil[b]{Fraunhofer Institute for Industrial Mathematics ITWM, Kaiserslautern, 67663, Germany}
\affil[c]{Department of Mathematics and Systems Analysis, Aalto University, Espoo, 02150, Finland}

\title{Quantum computing for~discrete optimization: a~highlight of~three technologies.}

\addtokomafont{date}{\vspace*{-1em}\normalsize}

\begin{document}
\maketitle

\vspace*{-1.5cm}
\begin{abstract}
  Quantum optimization has emerged as a promising frontier of quantum computing, providing novel numerical approaches to mathematical optimization problems. %
  % begin quote abs!
  The main goal of this paper is to facilitate interdisciplinary research between the \gls{OR} and quantum computing communities by \revI{helping \gls{OR} scientists to build initial intuition for-, and offering them a hands-on gateway to} quantum-powered methods \revI{in the context} of discrete optimization. %
  % end quote
  To this end, we consider three quantum-powered optimization approaches that make use of different types of quantum hardware available on the market. To illustrate these approaches, we solve three classical optimization problems: the Traveling Salesperson Problem, Weighted Maximum Cut, and Maximum Independent Set. With a general \gls{OR} audience in mind, we attempt to provide an intuition behind each approach along with key references, describe the corresponding high-level workflow, and highlight crucial practical considerations. In particular, we emphasize the importance of problem formulations and device-specific configurations, and their impact on the amount of resources required for computation (where we focus on the number of qubits). These points are illustrated with a series of experiments on three types of quantum computers: a neutral atom machine from \gls{QuEra}, a quantum annealer from \gls{D-Wave}, and gate-based devices from \gls{IBM}.\vspace{0.5\baselineskip}

\end{abstract}

\glsresetall

\section{Introduction}\label{sec:intro}
Discrete optimization is central to many problems in \gls{OR}, often arising in
efficient organization of complex systems across various domains: logistics,
supply chain management, transportation, finance, healthcare, and more. This led
to the development of a vast variety of computational methods, with many
state-of-the-art algorithms implemented in advanced stand-alone
solvers, such as \citet{gurobi}, SCIP \citep{scip}, and others. Still, practical
problems often require substantial resources, as the
solution space grows exponentially with the problem size.

Quantum computing offers an alternative computation model, but both the
methodology and the hardware are in early stages of development, as compared to
classical\footnote{We use the term ``classical'' throughout this paper to distinguish non-quantum technologies from quantum technologies. In this sense, a ``classical computer'' is a conventional computing device that uses non-quantum principles of information processing.} computing. It is unclear to what extent, if at all, \gls{OR} will
benefit from quantum technology in the near future. There are reasons to believe that it might
yield a significant speed-up, especially for discrete
optimization \citep{preskill2018}, although no practical (exponential) quantum
advantage has yet been demonstrated for an optimization
problem \citep{hoefler2023}. We think that to be able to assess potential
opportunities at this early stage, \gls{OR} experts might want to familiarize
themselves with the topic to a certain degree. However, the necessary background
knowledge significantly differs from a typical discrete optimization specialist
training. This is mainly because the language used in the literature naturally
draws on the works from quantum physics, making it less accessible to the
general mathematical readership.

Existing literature already provides overviews of the current state of research.
A discussion of possible synergies between \gls{OR} and quantum information
science along with further research directions is presented by
\citet{parekh2023}. \Citet{klug2024} and \citet{au-yeung2023} discuss quantum optimization; 
a more comprehensive overview is provided by \citet{abbas2024}.
The current state and prospects of \glspl{QC} are discussed by \citet{scholten2024}.
There is also a significant body of literature
devoted to numerical investigation and benchmarking of quantum-powered
approaches, for example \citet{lubinski2024} and \citet{koch2025quantumoptimizationbenchmarklibrary}.

This paper differs fundamentally from the overviews mentioned above. Rather than providing a general overview, we focus on specific quantum devices and types of discrete optimization problems, and seek to achieve three goals: convey the underlying intuition of each approach, provide the necessary references for a deeper understanding, and discuss relevant workflows and key practical considerations.
An important aspect that we take into account is the required number of qubits for optimization tasks, as it is an easy-to-understand yet important metric.
Our findings highlight the fact that
the number of qubits required to solve an optimization problem depends not only
on the number of binary variables, but also on the problem structure,
formulation, and the specific quantum device. We illustrate this aspect with a
series of experiments.

% begin quote 4!
\label{q:4}
\revI{
As with any scientific work, this paper represents the state of the art at the time of writing. Quantum computing is a rapidly evolving field, with continuous advances in both hardware and algorithms. Although specific hardware implementations may change and improved algorithms may emerge, we believe that the basic principles discussed here will continue to be relevant for understanding the advances and challenges in quantum optimization in the near future. By analyzing current technologies, we aim to provide insights that will be valuable both now and as the field continues to evolve.}
% end quote

The paper is structured as follows. In \cref{sec:QC and opt}, we provide a brief introduction to quantum optimization, where we focus on the
motivation for applying quantum computing to discrete optimization.
Subsequently, we present three distinct quantum technologies in \Cref{sec:quantum workflows},
focusing on their optimization workflows. In \cref{sec:qubits}, we
discuss the application of these workflows to three specific discrete
optimization problems, highlighting the necessary number of qubits in connection
to the chosen computational approach. This discussion is further illustrated
with numerical results in \cref{sec:use cases}. Finally, we conclude with a summary and brief outlook in \Cref{sec:conclusion}.

\section{Quantum optimization}\label{sec:QC and opt}
\revI{%
A \glsentryfull{QC} is a computational device that harnesses quantum physics to process information~\citep{nielsen2010}.
Quantum information processing is fundamentally different from its classical counterpart, and offers the possibility of performing certain computations beyond the classically achievable performance. This makes \glspl{QC} a potentially powerful tool for solving optimization problems. However, while the potential is significant, improvements on both the algorithmic and hardware side are currently needed to realize an advantage for practical applications~\citep{abbas2024}.
Despite these obstacles, we believe this somewhat early stage is a particularly good moment to study the topic from an \gls{OR} perspective, in order to better understand current limitations and assess possible future opportunities.
In the following, we lay the foundation for such an investigation. For this purpose, we outline how \glspl{QC} work and motivate their use for optimization tasks in light of the selected quantum computing technologies.}

% begin quote 1 !
\revI{%
  It is important to note that quantum physics can be described using different mathematical formalisms, each based on different principles, but each equally capable of accurately predicting experimental results \citep{styer2002}. One of the most widely used approaches in quantum computing involves the linear algebraic descriptions, where the states of a quantum system are represented by vectors in Hilbert space, and the state transformations are described in terms of linear operators (matrices). For a rigorous treatment of this framework, we refer the reader to standard textbooks, such as \citet{nielsen2010} or \citet{mermin2007}.}

\revI{In the scope of the present work, we can only provide a high-level summary of the underlying concepts. Therefore, we use slightly oversimplified explanations where necessary and try to omit the mathematical background of quantum physics where possible. More in-depth information can be found in the cited references.}
% end quote 1

\subsection{Quantum computing}\label{sec:QC}
\revI{%
A \gls{QC} is typically a highly complex device with many different interoperating components, with the \gls{QPU} being the core information processing unit within a \gls{QC}, in analogy to the \gls{CPU} within a conventional computer. There are generally two modes of operation for a \gls{QPU}: the \emph{analog mode} and the \emph{digital mode}, both of which are considered in \cref{sec:quantum workflows}.
In analog mode, computations are realized with a continuous control of
an underlying quantum system, similar to how quantum systems behave in nature due to the laws of quantum physics.
This idea is closely related to classical analog computing, where the natural behavior of physical systems is exploited to store and process information \citep{MacLennan2009,zangeneh2021,wu2022,ulmann2024}.
The most popular approach for this mode of operation, \gls{AQC}, is discussed in \cref{sec:QAA}.
The digital mode, on the other hand, is based on
discrete controls, so-called \emph{gates} (in analogy to classical logical
gates, such as AND or XOR), and is therefore also commonly referred
to as \gls{GQC}. Gates are effectively a discretization of the underlying
continuous controls, making the digital mode an additional level of abstraction over the analog mode. Notably, \cite{nannicini2020} attempts to provide a relatively compact and self-contained introduction to gate-based quantum computing, without relying on the ideas from physics.}

%%%
\revI{Regardless of whether it is operated in analog or digital mode, a \gls{QPU} works with quantum bits, or \emph{qubits}, to process information. These qubits exhibit a much more complex behavior than classical bits.
The main difference is that a classical bit has, by definition, two binary states, typically denoted by $0$ and $1$.
A qubit, on the other hand, can attain infinitely many states because its state space is a continuous spectrum of possible configurations.
Based on the underlying theory of quantum physics, each state in this continuous state space can be identified with two coordinates. Due to the state space topology, a commonly used interpretation is to identify the two coordinates as angles that address a point on the surface of a unit sphere. This interpretation allows visualizing the state space of a qubit in form of the so-called Bloch sphere, as shown in \cref{fig:bloch}.}
%%%
\begin{figure}
  \begin{minipage}{0.25\textwidth}
    \includegraphics[width=\textwidth]{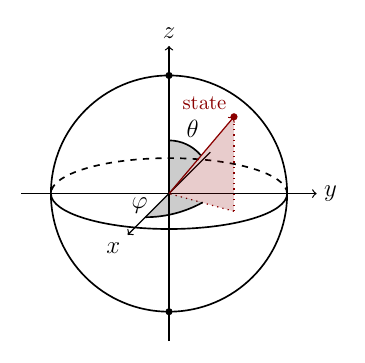}
  \end{minipage}\hfill%
  \begin{minipage}{0.75\textwidth}
    \revI{Every possible qubit state is represented by a point on the surface of this
    (three-dimensional) sphere. The axes of the underlying Cartesian coordinate
    system are labeled with $x$, $y$, and $z$. Two angles, $\theta \in [0,\pi]$
    and $\varphi \in [0,2 \pi)$, are sufficient to identify any state, as shown
    for an example state. The north pole and south pole of the sphere represent
    two classical states in analogy to the two states of a classical bit.}
  \end{minipage}
  \caption{\label{fig:bloch}\revI{The Bloch sphere: a visualization of the state space
    of a single qubit.}}
\end{figure}
%%%

\revI{%
This fundamental difference between the state spaces of classical bits and qubits is crucial to understand the difference between classical and quantum information processing. In simple terms, one could argue that the information encoding of a qubit is much \emph{denser} than of a classical bit. In fact, a single qubit could in principle store an infinite amount of information with its infinitely large state space \citep{nielsen2010}. However, this would require an infinitely precise qubit control, which is not feasible in practice. And even if we could achieve infinitely precise control over qubits, fundamental principles of quantum mechanics still limit the amount of information that can be accessed. This is formalized by the \emph{Holevo bound}, which states that the amount of classical information that can be extracted from $n$ qubits is at most equivalent to that of $n$ classical bits \citep{nielsen2010}.
The advantage of quantum computing therefore lies less in the efficient storage of data than in its efficient processing. However, being able to encode information in a dense way does not necessarily mean that processing such information is automatically efficient. In contrast, a major challenge for quantum computing is precisely how to exploit the large state space of qubits for efficient computations.}

\revI{Performing computations is equivalent to processing data, and the question arises as to how this can be done with a qubit.
Since information is encoded using the principles of quantum physics, any interaction with this information also takes place on the quantum level. In this context, there is one important feature that strongly distinguishes classical from quantum information processing: the non-deterministic nature of quantum physics. Non-determinism affects how the information that has been stored in a qubit can be read out. For a classical bit, this is a trivial task since any bit that has been written can also be read in the same way without any loss of information. The situation is different for qubits. Reading information from a qubit, also known as \emph{measuring} the qubit, means to interact with it on a quantum level. According to the laws of quantum physics, such a measurement is always a probabilistic process that reveals a binary outcome. The two possible measurement outcomes can for example be labeled $0$ or $1$ in agreement with the states of a classical bit.}

\revI{The chance of measuring either $0$ or $1$ depends on the state of the qubit before the measurement. In other words, each qubit state represents a distinct probability distribution that determines its probabilistic measurement process. Again, the Bloch sphere can be used to develop an understanding for this highly counter-intuitive behavior. The north pole and the south pole of the Bloch sphere represent the only two states with an effectively deterministic measurement process. The north pole is also called the \emph{ground state} and the south pole the \emph{excited state} of the qubit (for reasons we explain further below). By definition, measuring a qubit in the excited state always yields $1$ and, conversely, measuring a qubit in the ground state always yields $0$. Both state are therefore also referred to as \emph{classical states} in analogy to the two states of a classical bit. For any other state on the Bloch sphere (also called \emph{superposition} states), the probability distribution of measurement outcomes depends on its spherical distance to the north pole representing the ground state (or, conversely, the south pole representing the excited state). The closer it is to a respective pole, the more probable the corresponding outcome. In other words, the closer a state is located to one of the poles, the more biased is the measurement. All states on the northern hemisphere have a bias towards $0$, all states on the southern hemisphere have a bias towards $1$, and all states on the equator have an equal chance of a measurement outcome of $0$ or $1$, as for an ideal coin flip.}

\revI{It is important to clarify that this non-deterministic behavior is not the result of technical limitations, but an intrinsic property of quantum physics~\citep{bera2017}. In addition to being non-deterministic, measurements also are inherently \emph{destructive}. This means that when a measurement yields an outcome of $0$ ($1$), the state of the qubit is changed from its original state into the ground (exited) state and a subsequent measurement will always result in the same measurement outcome. In other words, the information encoded in a qubit is lost after measuring it once. The significance of single measurements may therefore be limited for quantum information processing, and it is often the case that quantum computations need to be repeated many times to get meaningful measurement results and algorithmic designs that can operate on finite samples (\ie, measurement results) of the underlying joint probability distribution. We also discuss this limitation in \cref{sec:quantum workflows}.}

\revI{To this point, we have mostly focused on a single qubit, but meaningful quantum information processing requires multiple qubits in the same sense as multiple bits are required for meaningful classical information processing. Similar to how information is stored more densely in qubits than in classical bits, the interaction between qubits is also much more complicated than the interaction between classical bits.}
\label{q:22}%
% begin quote 22!
\revIII{First and foremost, multi-qubit systems can exhibit \emph{entanglement}, a unique quantum phenomenon in which the state of one qubit cannot be described independently of the others. As a result, operations performed on one qubit may instantaneously alter the state of the entire quantum system. This is fundamentally different from classical information processing, where each bit can always be flipped independently of all others.}
% end quote
\revI{Entanglement is a key feature used in many quantum algorithms. It enables \emph{quantum parallelism}~\citep{markidis2024}, a distinct computing paradigm for quantum computers which is not equivalent to mere parallel computation, as different branches of computation can actually interact, an effect also known as \emph{interference}.
Quantum parallelism does not necessarily lead to a computational advantage, but can be used to design quantum algorithms that may benefit from entanglement.
As a prototypical example, the Deutsch-Josza algorithm~\citep{nielsen2010} determines a global property of a Boolean function
using a single quantum evaluation.}

\revI{In analogy to the measurement of a single qubit, the measurement of a multi-qubit system behaves non-deterministically, however, the results of measurements of the entangled qubits are correlated. Generally, the measurement results of an $n$-qubit system can therefore be viewed as samples drawn from a joint probability distribution with a number of parameters exponential in $n$. This may (or may not) address the
limitations of classical systems in some cases, when an exponentially large
state space leads to practical intractability of the problem at hand.}

\subsection{Quantum hardware}\label{sec:qhardware}
\revI{In addition to the fundamental challenges of quantum computing due to the unique computational model, an additional practical challenge lies in the technological limitations of the hardware. \Glspl{QC} are currently very error-prone and have extremely limited resources, which is why they are also sometimes called \gls{NISQ} devices. 
A long-term goal for hardware development is \gls{FTQC}, where the effects of hardware noise are completely eliminated using methods such as \gls{QEC}. While there is active research and continuous improvement \citep{gottesman2022,katabarwa2024}, fault-tolerant computations are not yet feasible at any practical scale. Therefore, current algorithms must take these hardware-induced errors and limitations into account.}

\revI{The physical realization of qubits requires quantum systems that can attain at least two distinct quantum states, which are used to represent $0$ and $1$. In practice, these states are typically chosen as so-called \emph{energy states} (which describe the energy of a system), since energy levels in many quantum systems naturally occur in discrete levels.
For a single-qubit system, we have referred to them as the ground state and the excited state, respectively, a terminology that becomes clear in this context: the ground state corresponds to the lower energy state of the qubit (representing $0$), while the excited state corresponds to a higher energy state (representing $1$). These energy levels provide a practical and intuitive basis for labeling the two poles of the Bloch sphere.
Computations can then be realized by shifting qubit states between different energy levels. For a multi-qubit system, while a measurement still yields a single binary digit per qubit, the system is usually configured in a problem-dependent way to rearrange its energy levels. The usual goal is for a measurement of the ground state (\ie, the minimum-energy state) not to yield a trivial bitstring of zeroes, but reveal some meaningful information about the problem.
In other words, the ground state of a multi-qubit system can be used to encode information.
We will revisit this topic in \cref{sec:QAA}.}

\revI{There are many different quantum hardware providers competing for the best solution and offering different kinds of devices based on different kinds of technologies that may each have their unique advantages and disadvantages. Most devices can be accessed remotely via an online interface or platform, eliminating the need for physical proximity between \gls{QC} and user. Practically usable \glspl{QC} include
superconducting devices (\eg, by \gls{IBM}, \gls{D-Wave}, \gls{Google},
\gls{IQM}, and \gls{Rigetti}), photonic devices (\eg, by \gls{Xanadu}), neutral
atom devices (\eg, by \gls{QuEra}, \gls{Pasqal}, and \gls{RymaxOne}), and trapped ion devices
(\eg, by \gls{Quantinuum}, \gls{Honeywell}, and \gls{IonQ}), just to name a few~\citep{Gyongyosi2019}.
In \cref{sec:quantum workflows}, we consider three quantum-powered optimization approaches running on distinct \glspl{QC} to illustrate how differently operating devices can still be summarized under a unified workflow.}

\revI{Throughout this work, we use the term \emph{logical qubit} to refer to an abstract (idealized) qubit at the algorithmic level of the underlying quantum computing model, but independent of the hardware implementation.%
  \footnote{\revI{We do not use the term \emph{logical qubit} as it is commonly used in the context of \gls{FTQC} \citep{zhao2022}.}% I suggest this edit, since otherwise
    % it produces a strange empty footnote on the next page.
  } Conversely, a \emph{physical qubit} refers to the actual implementation on a hardware device, which may differ in its physical nature from device to device. A central question in quantum optimization is whether the quantum hardware of choice has a sufficient number of physical qubits to process a given task. In practice, determining the required number of qubits typically involves a two-step process. First, the mathematical optimization problem must be translated into an abstract quantum algorithm, which requires a specific number of logical qubits. This number typically scales with the size of the optimization problem. Second, the abstract algorithm must then be mapped to a hardware-specific realization, which requires a certain number of physical qubits. This number is often significantly higher than the number of logical qubits, as we will describe in \cref{sec:rydberg,sec:QA,sec:QAOA}.}

\subsection{Quantum advantage}\label{sec:QAdv}
\revI{There is evidence that also without \gls{FTQC}, quantum computations can be practically relevant for certain use cases \citep{kim2023}. However, identifying such use cases is not trivial. Because of the different computational model, a one-to-one
implementation of classical algorithms on \glspl{QC} usually do not lead to performance gains (in fact, the opposite is true). Specialized solutions for carefully selected problem classes are needed to make the most of the hardware \citep{harrow2017}.
In this sense, the role of a \gls{QPU} can be compared to that of a \gls{GPU}. A \gls{GPU} has a clear advantage over a \gls{CPU} in certain specialized tasks, such as parallel numerical computation. Likewise, a \gls{QPU} has the potential to offer significant benefits in solving specific quantum-suited problems. However, like a \gls{GPU}, a \gls{QPU} cannot be expected to be a general-purpose problem solver.}

\revI{Demonstrating that a quantum computer can solve a problem significantly better than a classical computer (typically implying a superpolynomial speedup) is also known as a \emph{quantum advantage} or \emph{quantum supremacy} \citep{harrow2017}.
There is theoretical and experimental evidence for quantum advantage in tasks such as \emph{random circuit sampling}~\citep{hangleiter2023}, but the practical utility of these tasks remains limited.
Moreover, there are quantum algorithms that have a purely theoretical advantage when run on noise-free hardware. For example, Grover's search algorithm~\citep{GroverSTOC} is mathematically proven to provide a quadratic speedup to unstructured search problems, but there has not been any experimental verification of significant scale yet.}
\Citet{preskill2018} identified three main reasons why a quantum advantage could potentially be achieved with sufficient technological advancement:

\begin{enumerate}
  \item There are examples of mathematical problems that are believed hard to
	solve for classical computers, and for which efficient quantum-powered
	algorithms exist (\eg, factoring integers and finding discrete
	logarithms, as studied in \citealp{shor1997}).
  \item There are quantum states that are relatively easy to prepare on a
	\gls{QC} such that measuring them is equivalent to the sampling of
	random numbers from a particular probability distribution, which is
	impossible to achieve efficiently by classical means~\citep{harrow2017}.
  \item No efficient classical algorithm is known that is able to simulate
	a \gls{QC} efficiently at arbitrary scales, which is essentially due to the
	exponentially large state space of multi-qubit systems.
\end{enumerate}

\revI{It is not immediately obvious how these points translate to practical
applications in the \gls{OR} context. In general, it can be expected that problems that can be solved efficiently by classical methods today are likely to remain in the domain of classical computing.
On the other hand, problems that are currently intractable or that require an excessive amount of computational resources are promising candidates for a quantum algorithm.
Among these, discrete optimization is a particularly notable application domain \citep{koch2025quantumoptimizationbenchmarklibrary}, which justifies to study this topic from an \gls{OR} perspective.
Despite this premise, no clear proof of a quantum advantage for optimization has been achieved yet \citep{abbas2024}.
Although obtaining a quantum advantage remains a key goal, it is not a strict requirement for achieving practical benefits. Another perspective is that quantum computers only need to perform reliable computations beyond the scope of brute-force classical computations to achieve a practical benefit, which is referred to as \emph{quantum utility} \citep{kim2023}.}

\revI{Quantum computations do not necessarily have to be seen as isolated from classical computations. The term \emph{hybrid quantum-classical} is used for algorithms, where quantum and classical computations are performed together, typically in an iterative fashion. The \gls{QC} can then be used to solve only the specific tasks for which it is suitable. This makes hybrid quantum-classical algorithms particularly promising for near-term applications on \gls{NISQ} devices. We present an example in \cref{sec:QAOA}.
There is evidence that hybrid quantum-classical algorithms can also enable a polynomial runtime improvement. For example,
\citet{creemers2025} focus on applying Grover's search to achieve a quadratic speedup in the
context of algorithms such as hybrid quantum-classical branch-and-bound.}

% begin quote 2!
\revI{From a broader methodological perspective, one can identify several
  paradigms for designing potentially promising quantum algorithms. Their
  detailed description is beyond the scope of this paper, but some high-level
  discussion is given by \citet{nielsen2010}, and an in-depth overview in the
  context of discrete optimization is presented by \citet{abbas2024}. We take
  another approach here and focus on a specific example of a key idea from physics
  in \cref{sec:QAA}, which motivates a group of actively developed
  quantum optimization methods discussed in more detail further.}
% end quote

\subsection{Quantum Adiabatic Algorithm}\label{sec:QAA}

\revI{We start with the outline of a central quantum computational strategy that serves as an umbrella concept for the three quantum-powered optimization approaches we consider in this paper. This strategy, the so-called \gls{QAA}~\citep{farhi2000,Albash2018}, is based on the principle of solving optimization problems by leveraging the natural evolution of quantum systems. \Gls{QAA} founds on the computational model of \gls{AQC}, a form of analog quantum computing.}\footnote{\label{q:23}%
  % begin quote 23!
  \revIII{While adiabatic and digital quantum computing are often presented as distinct, they are theoretically equivalent up to polynomial overheads. This means that each algorithm in one model can be translated into the other \citep{Albash2018}. In practice, some quantum computations such as \gls{QAA} are more naturally expressed in the adiabatic framework, whereas others align better with the digital framework.}%
  % end quote
}

\revI{The \Gls{QAA} relies on the \emph{adiabatic theorem} of quantum physics, which, in simplified terms, guarantees that if a quantum system is in its ground state (\ie, the minimum energy state) and its ``energy landscape'' (comprised of all possible energy states including the ground state) is changed ``slowly enough,'' it will remain in the ground state.
In other words, the system may slowly ``evolve'' from one ground state to another ground state in a different energy landscape. The necessary pace of the changes in the system depends on the size of the so-called \emph{minimum energy gap}, the smallest energy difference between the ground state and the second lowest energy state throughout the entire evolution process.}\footnote{\label{q:24}%
% begin quote 24!
\revIII{The adiabatic theorem requires that the evolution runtime must be large on the timescale set by $(1/\Delta^{2})$, where $\Delta$ denotes the minimum energy gap \citep{Albash2018}. In practice, this means that $\Delta$ must also be non-zero, which means that energy levels must not cross \citep{Zhang2014}. Otherwise, achieving an adiabatic transition is fundamentally impossible.}%
% end quote
} \revI{The evolution of a quantum state that complies with the adiabatic theorem is called \emph{adiabatic evolution}.}

\revI{The energy landscape of a multi-qubit quantum system can be used to encode information. The ground state, as the lowest-energy configuration within this landscape, therefore carries information as well.}\footnote{%
  \label{fn:nozeros}\revIII{It is thus not surprising that a measurement of the ground state does not necessarily yield a trivial all-zero bitstring, as such an outcome would imply no information gain.}}
\revI{At the core of \gls{QAA} lies the ability to store and extract information from the ground state of a quantum system by carefully shaping and controlling its energy landscape, taking into account the adiabatic theorem.
The conceptual idea can (from a hardware-agnostic perspective) be outlined as follows:
First, the multi-qubit system of the \gls{QC} is initialized in a known and easy-to-prepare ground state. By design, this is a generic ground state, which is in particular independent of the optimization problem of interest. Then, over time, the system parameters (which are hardware-specific) are slowly tuned to achieve a controlled adiabatic evolution. The goal of this controlled evolution is to gradually shape an energy landscape that is a one-to-one representation of the optimization landscape of the problem of interest. This means that each candidate solution of the optimization problem can be associated with a quantum state and the corresponding objective value with an energy. After a successful adiabatic evolution, the system is by definition still in the ground state, but this new ground state now represents an optimal solution to the optimization problem.
Hence, by measuring this state, the solution can be found. The surprising insight is that it is not necessary to know the optimal solution to be able to prepare the final ground state that contains this information.
This is only possible because quantum physics itself is used to discover the solution in the sense of an analog computation.}

\label{q:25}%
% begin quote 25 !
\revI{The \gls{QAA} is inherently different from classical methods because it can benefit from quantum features such as superposition and tunneling, which allows the system to explore many possible solutions ``simultaneously'' and to ``escape'' from local energy minima.}\footnote{\label{q:25fn}%
    % begin quote 25fn!
    \revIII{The role of entanglement in the performance of \gls{QAA} and even how best to characterize and measure it in this context remains an open area of research topic \citep{Albash2018}.}%
    % end quote 25fn
  } 
\revI{As a consequence, the \gls{QAA} has the potential to outperform classical methods. However, this presumption has not been demonstrated for problems of practical scale yet, and is the subject of ongoing research \citep{abbas2024}.}
% end quote 25

\revI{Indeed, there are three significant challenges for practical applications. First, a technical requirement for an optimization problem to be solvable with this approach is that the energy landscape of the quantum system must be tuneable in such a way that it becomes a one-to-one representation of the corresponding optimization landscape of the problem. In effect, this means that only hardware-specific (or ``natural'') problem classes can be solved on every quantum device. In order to solve a problem from a different problem class, it must first be mapped onto the natural problem class through suitable modeling. If this is not possible, the problem cannot be solved with the device at hand.}

\revI{Second, the minimum energy gap is usually undeterminable in practice since determining it requires knowledge of the ground state energy, which is the optimal solution to the optimization problem to be found by the computation. Therefore, an estimated minimum energy gap (\eg, based on empirical results) is typically used instead to determine the timescale of the evolution process. However, this does not necessarily ensure an adiabatic behavior.}

\revI{Third, there is no guarantee that the complex quantum dynamics behind the analog computation reveal the desired result in all cases. Typically, the final state of the quantum system is a superposition state such that repeating the entire process might yield different measurement results each time, corresponding to new solution candidates. In an idealized scenario, these solution candidates would all be optimal solutions to the underlying problem (each with the same minimal energy). In practice, however, quantum fluctuations, non-optimal tunneling, an insufficiently slow evolution speed (due to an unfavorable estimate for the minimum energy gap), and hardware-related uncertainties of \gls{NISQ} devices lead to a final state that is in a superposition of sub-optimal (and potentially optimal) solutions. In other words, measuring the final state may reveal sub-optimal solutions, making the \gls{QAA} effectively a heuristic. It is therefore common to collect data from multiple algorithm runs, usually referred to as \emph{shots}, to choose the best outcome.}

\section{Quantum optimization workflows}\label{sec:quantum workflows}
% begin quote 3!
\revI{In this paper, we examine three different quantum-powered optimization approaches, all of which are based on the conceptional idea of \gls{AQC} and, more specifically, the \gls{QAA} from \cref{sec:QAA}.
They differ not only in the chosen hardware and algorithmic realization, but also in their applicability to specific problem classes. Definitions for these problem classes are provided in \cref{sec:qubits}. Our goal is to illustrate how different quantum computing strategies can be used for practical \gls{OR} applications. All of the algorithms we consider here are well-known in the literature, we particularly do not invent new methods to achieve better results, but focus on the performance of existing methods. We consider the three following approaches (see also \zcref{app:runs summary} for more details).}
% end quote

\begin{enumerate}
	\item \textbf{\gls{QuEra-OPT}} (\cref{sec:rydberg}): We use a quantum algorithm that can be understood as a realization of a \gls{QAA} close to the original concept of analog adiabatic evolution. As quantum hardware, we use \gls{Aquila} from \gls{QuEra}~\citep{wurtz2023}, an analog device operating with trapped neutral atoms. It can be accessed commercially online via the \emph{AWS Braket} cloud service~\citep{AWS2020}. The task is to solve \glspl{UD-MIS}, the natural problem class of the \gls{Aquila} device.
	\item \textbf{\gls{D-Wave-OPT}} (\cref{sec:QA}): We use a quantum algorithm known as \gls{QA}. In simple terms, it can be understood as a practice-oriented variant of a \gls{QAA} that may potentially involve a non-adiabatic evolution. As quantum hardware, we use the \gls{D-Wave} \gls{Advantage} quantum annealer~\citep{mcgeoch2020a}, an analog device operating with superconducting qubits. It is commercially accessible online via the \emph{Leap} cloud service~\citep{dwave2024}. The task is to solve \glspl{QUBO}, the natural problem class of the \emph{Advantage} device.
	\item \textbf{\gls{IBM-OPT}} (\cref{sec:QAOA}): We use a hybrid quantum-classical algorithm known as \gls{QAOA}. In simple terms, it can be understood as a digitized approximation of a \gls{QAA}. As quantum hardware, we use two gate-based devices from \gls{IBM}, \gls{IBMCusco} and \gls{IBMNazca}, which operate with superconducting qubits. \Gls{IBM} devices are commercially accessible online via the \emph{IBM Quantum} cloud service~\citep{ibm2023}. The task is to solve \glspl{QUBO}, although variants of \gls{QAOA} are suitable for different problem classes.
\end{enumerate}

\revI{All three approaches share a unified workflow in three high-level steps, as outlined in \cref{fig:workflows}:}

\begin{enumerate}
  \item[(I)] The given problem $P$ first has to be \emph{modeled} in form of a suitable
  % \sven{Should we say supported or suitable instead of feasible?}
  % AB: Yes, I think it is better, as we usually write "feasible" in specific sense, about solutions,
  % it might get confusing.
  problem class.
  \item[(II)] This problem formulation is then used to specify the \emph{configuration} of the algorithm and hardware.
  \item[(III)] Finally, the actual \emph{solution process} takes place, usually involving a series of quantum hardware computations and additional classical computations, (\eg, pre- and post-processing).
\end{enumerate}

\begin{figure}[ht]
	\centering
	\resizebox{0.9\linewidth}{!}{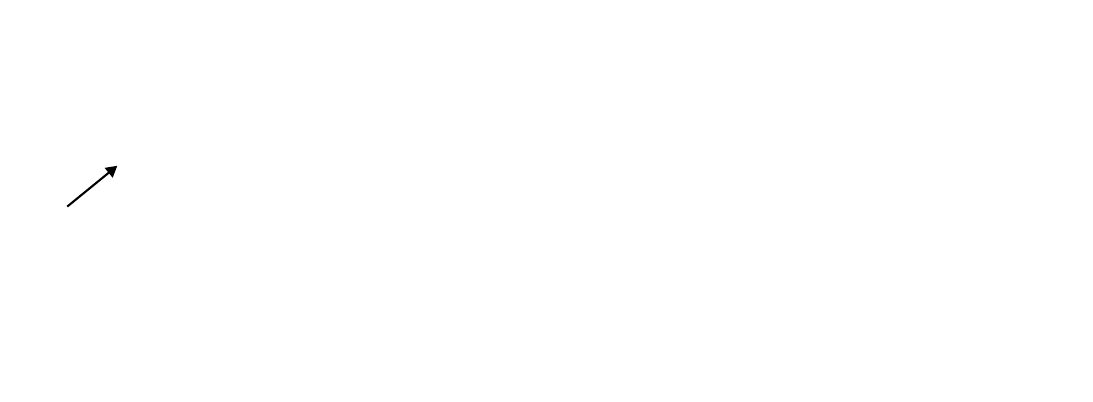}
	\caption{
		\revI{
			Unified workflow for the three quantum-powered optimization approaches \gls{QuEra-OPT}, \gls{D-Wave-OPT}, and \gls{IBM-OPT} that consists of three high-level steps: (I) Modeling, (II) Configuration, and (III) Solution. All three approaches utilize \gls{QAA}-inspired techniques.
		}
	}\label{fig:workflows}\zlabel{fig:workflows}
\end{figure}

In the following, we provide a brief outline of the three approaches in the
context of these three high-level steps. We limit \revIII{ourselves to the description}\label{q:the} of
the fundamental principles and provide further references for more details.

\subsection[\gls{QuEra-OPT}: \gls{UD-MIS} and \gls{QAA}]{\texorpdfstring{\Gls{QuEra-OPT}: solving \glspl{UD-MIS} with an analog \gls{QC} (\gls{QuEra})}{NA-OPT: solving UD-MIS with an analog QC (QuEra)}}\label{sec:rydberg}
Our first approach, \gls{QuEra-OPT}, makes use of the analog device \gls{Aquila} from \gls{QuEra}~\citep{wurtz2023} to solve \glspl{UD-MIS}~\citep{pichler2018,ebadi2021,coelho2022,ebadi2022} by exploiting the so-called \emph{Rydberg blockade effect}, a physical law originating from quantum physics. A recent review of the technology is given by \citet{wintersperger2023}.

\paragraph{Hardware.}\label{q:Rydberg-hw}
The \gls{Aquila} device operates with an array of rubidium atoms, which are trapped in a vacuum cell by lasers and can be arranged on a two-dimensional plane using optical tweezers.
The atoms realize so-called \emph{Rydberg qubits}, which possess
two energy states: the non-Rydberg state and the Rydberg state.
The Rydberg blockade effect makes it energetically favorable
that atoms within a certain distance from each other, the so-called \emph{Rydberg blockade radius}, are not simultaneously in a Rydberg state. This effect can be used to solve
\glspl{UD-MIS}, as explained in the following.

\paragraph{Algorithm.}\label{q:Rydberg-algo}
\revIII{We use a realization of a \gls{QAA} and follow the unified workflow from \cref{sec:quantum workflows}. A device-specific scheme is summarized in~\cref{fig:qopt-rydberg}.
For practical reasons, the \gls{Aquila} device
can only hold up to \num{256}~atoms in the array at once, which translates to at most \num{256}~nodes in the \gls{UD-MIS} instance. Moreover, the geometric configurations of the atoms array, and hence possible \gls{UD-MIS} graphs, are restricted to a square
two-dimensional lattice that also enforces a minimum distance between all nodes.
If the geometry of a graph from a \gls{UD-MIS} instance does not comply to these requirements, it has to be transformed first in the configuration step, which may or may not be possible depending on the instance.}

\begin{figure}[ht]
	\captionbox{Workflow of the \gls{QuEra-OPT} approach: solving \glspl{UD-MIS} with \gls{QAA}.\label{fig:qopt-rydberg}}[0.475\textwidth]{
		\includegraphics[width=\linewidth]{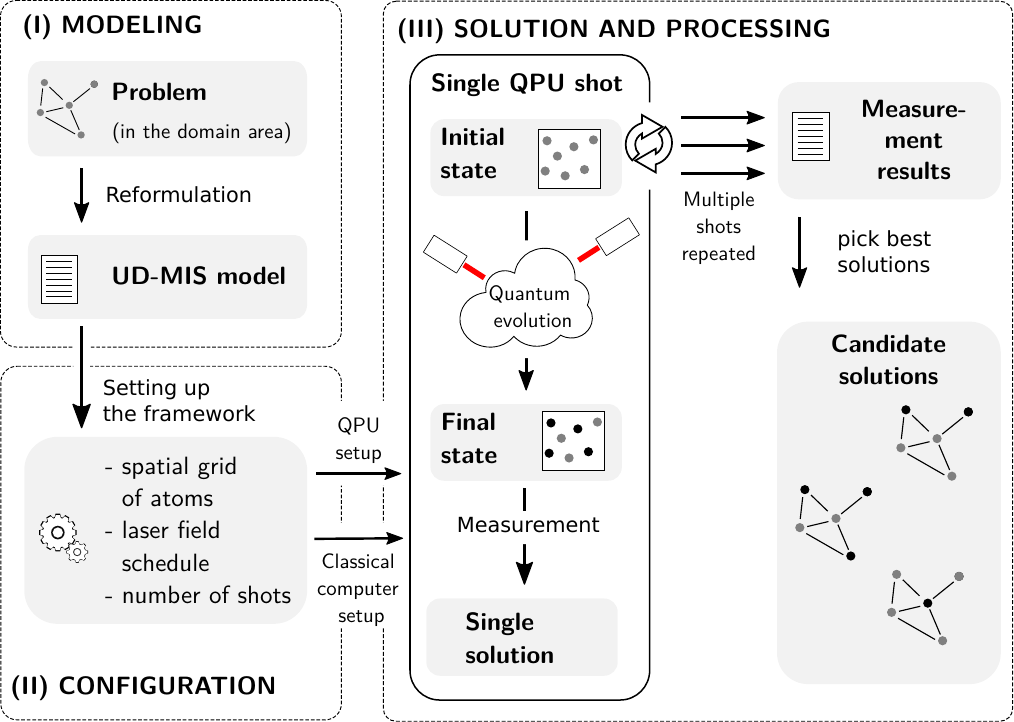}
	}\hfill
	\captionbox{Workflow of the \gls{D-Wave-OPT} approach: solving \glspl{QUBO} with \gls{QA}.\label{fig:qopt-qa}}[0.475\textwidth]{
		\includegraphics[width=\linewidth]{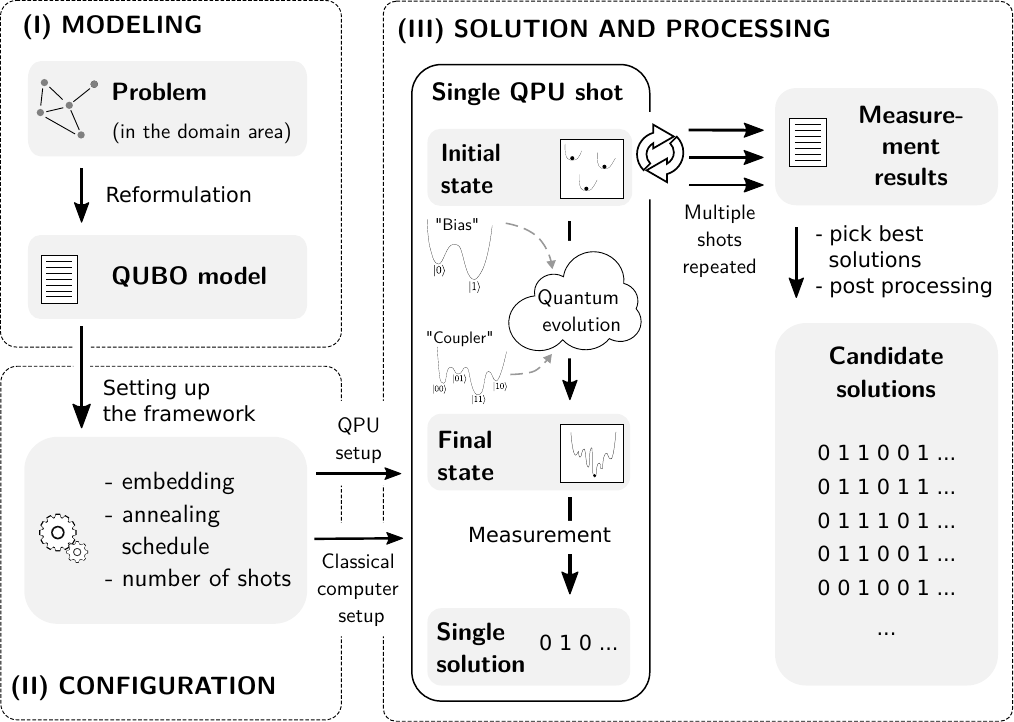}
	}
\end{figure}

\revIII{To get an intuitive understanding of how the adiabatic algorithm works, we first revisit the structure of unit disk graphs.
Presuming nodes that are localized on an Euclidean plane, a unit disk graph has an edge between two vertices if and only if the distance between them is at most one. In other words, the relative location of nodes determines their connectivity.
Therefore, provided the unit disk graph of a \gls{UD-MIS} instance, we can associate an atom with each node and arrange the atoms on the two-dimensional plane
according to the corresponding (scaled) node coordinates to achieve a one-to-one spatial correspondence between graph nodes and atoms.
The Rydberg blockade radius of the system is tuned to match the unit disk radius such that atoms that correspond to connected nodes lie within the Rydberg blockade radius.}

\label{q:26}%
% begin quote 26!
\revIII{With this problem encoding, a \gls{QAA} can be realized by evolving the system in a controlled fashion with an external laser field.
To this end, we start from a problem-agnostic
energy landscape, where the ground state corresponds to a quantum system in which all atoms are in a non-Rydberg state.
Slowly, the controls are adjusted to arrive at a problem-specific energy landscape
while the positions of the atoms remain unchanged.
The problem-specific ground state corresponds to a quantum system in which as many atoms as possible are in a Rydberg state (which might be ambiguous).
The Rydberg blockade effect---implemented by the atom locations---is intended to ensure the constraints of the \gls{MIS}, which require that neighboring (\ie, connected) nodes cannot both be part of the resulting node set. As a consequence, the ground state of the quantum system encodes a \gls{UD-MIS} solution: the atoms in the Rydberg state indicate the resulting node set.}
% end quote

\Citet{serret2020} estimate the necessary size of a neutral atom array to yield
an advantage over classical approximation algorithms for \gls{UD-MIS}, and 
\citet{wurtz2022} provide an overview of the connections between the \gls{MIS}
and many \gls{OR} problems in this context. Note that neutral atom arrays can in principle also be used to solve other
problem classes as well, which is an active field of research
\citep{nguyen2023}. For instance, a conceptional idea for a non-unit disk
(non-blockade-based) framework was suggested by \citet{goswami2023}.
Using a hardware feature called local detuning, \gls{Aquila} can also be used to solve the \gls{UD-MWIS}
in complete analogy to \gls{UD-MIS}.\footnote{At the time of performing our hardware experiments, this feature was not fully supported yet and \gls{UD-MWIS} is therefore not considered in our numerical evaluations.}

\paragraph{Number of qubits needed.}\label{q:rydberg-qubits}
% begin quote hw compliant 1!
\revIII{The number of logical qubits necessary to solve a \gls{UD-MIS} instance is equal to the number of graph nodes, which also directly translates to the number of physical qubits.
In addition to the number of nodes, a \gls{UD-MIS} graph must satisfy certain constraints to be encodable as an atom array on \gls{Aquila}. These include feasible placement of the corresponding atoms on a grid (with a size of around $16 \times 16$) and a unit disk size within a fixed range. Detailed hardware requirements are provided in \citet{wurtz2023}. We refer to any \gls{UD-MIS} instance that meets these criteria as a \emph{hardware-compliant} \gls{UD-MIS} instance.
In practice, choosing a suitable geometric arrangement of the atoms for an accurate problem representation might pose additional challenges. While this issue was never binding in our numerical experiments, for more connected problem instances it might become a problem. In a sense, a neutral-atoms-based \gls{QPU} can provide qubit \emph{connectivity} only up to a certain limit. Investigating this effect might constitute a potential direction for further research.}
% end quote

\revII{While the same \gls{UD-MIS} graph can be
	represented by different atom geometries, there are engineering constraints on the
	distances between the atoms and physical implications for each specific atomic
	configuration. For instance, positioning atoms approximately within the
	blockade radius of other atoms (not too close and not too far) might result in
	undesirable outcomes, due to the underlying physics of the protocol. Some best
	practices and discussion of the details of the underlying technology can be
	found in \cite{wurtz2023}.}

As the decision version of the \gls{UD-MIS}
problem is $\NP$-complete \citep{clark1990}, an arbitrary problem from
$\NP$ can be reduced to it with a polynomial overhead in the instance size. While the reduction following from the general proof goes via non-deterministic Turing machines, more direct reductions can exist for specific problems. We exemplarily use the results from \citet{nguyen2023}, who proposed a transformation that maps a \gls{QUBO} instance
with $N$ variables to an equivalent \gls{MWIS} instance over a unit disk graph with at
most $4N^2$ nodes, as summarized in
\cref{lm:Rydberg-qubits}.
\begin{lemma}{\cite[][Section V.C]{nguyen2023}}\label{lm:Rydberg-qubits}
  % begin quote hw compliant 2 !
  A \revIII{hardware-compliant} \gls{UD-MIS} instance with $N$ nodes can be solved on a neutral-atom-based \gls{QC} with $N$ qubits. An arbitrary \gls{QUBO} instance with $N$
  binary variables can be encoded as a \gls{UD-MWIS} instance with at most $4N^{2}$
  nodes, and hence, \revIII{assuming the geometric hardware constraints are met}\footnote{\revIII{Generally speaking, not any \gls{UD-MWIS} instance will be hardware-compliant. This comes from the fact that the grid necessary to place the $4N^{2}$ atoms might in principle require more space than that available due to the hardware constraints stemming from the optical properties of the hardware. For example, a dense \gls{QUBO} instance with 5 variables requires \num{100} atoms on a graph that spans almost the full $16 \times 16$ grid of \gls{Aquila}~\citep{nguyen2023}. However, in our dataset, such constraints were never binding, and it was the number of qubits that constituted a limiting factor to our instance sizes.}}, requires a neutral-atom-based \gls{QC} with $4N^{2}$ physical qubits.
  % end quote

\end{lemma}

\subsection[D-Wave-OPT: QUBO and QA]{\texorpdfstring{\Gls{D-Wave-OPT}: solving \glspl{QUBO} with a quantum annealer (\gls{D-Wave})}{QA-OPT: solving QUBOs with a quantum annealer (D-Wave)}}\zlabel{sec:QA}\label{sec:QA}
Our second approach, \gls{D-Wave-OPT}, makes use of the \gls{Advantage} quantum
annealer from \gls{D-Wave}~\citep{mcgeoch2020a} to solve
\glsfirstplural{QUBO}. \revIII{We review the specific formulation later, in
  \cref{sec:qubits}.} \Gls{Advantage} is a special-purpose device designed for \gls{QA}~\citep{finnila1994,das2005,morita2008,hauke2020}.

\paragraph{Hardware.}\label{q:QA-hw}
The \gls{Advantage} device is based on superconducting technology in which superconducting loops are used to physically realize qubits. The quantum system can be controlled via so-called \emph{couplers} and \emph{biases}. The bias
of each qubit is a control which allows making it more energetically favorable for this qubit to end up in
a specific state (either the ground, or the excited state). A coupler
controls the interaction between two \revIII{qubits. It}\label{q:QA-hardware} allows increasing or decrease the energy contribution of qubit correlations. For technical reasons, couplers are only implemented
between physically adjacent qubits, which means that there is a limit to the
interactions that can be realized. We will revisit this topic further below.

\paragraph{Algorithm.}\label{q:QA-algo}
Since \gls{NISQ} hardware is often too noisy to support sufficiently slow and stable adiabatic evolution and the minimum energy gap required for adiabaticity is usually unknown, \gls{QA} can be considered as a more practical variant of \gls{QAA}, which does not strictly rely on adiabatic evolution.
It is based on the same premise as \gls{QAA} by starting from the ground state of a problem-agnostic energy landscape and ending in the ground state of a problem-specific landscape.
However, it is usually implemented under conditions where adiabaticity may not be preserved due to noise or the finite runtime of the hardware. Consequently, it is clearly a heuristic method that relies on sufficiently many shots to produce reliable results.

Here, we use a realization of a \gls{QA} and follow the \revIII{unified workflow from \cref{sec:quantum workflows}. A device-specific scheme is summarized in~\cref{fig:qopt-qa}.}
Specifically, to solve a given \gls{QUBO} instance, we first encode the problem on the device by associating each binary optimization variable with a qubit. The initial energy landscape is chosen in a problem-agnostic way such that its ground state is a balanced superposition state of every qubit, for which an immediate measurement would return a uniformly random candidate solution. 
% \sveninline{How does this fit to \cref{sec:QC}, where we wrote `measuring a qubit in the ground state always yields $0$'?}
% \rhinline{@Sven: Thanks for your careful reading. We're reaching the limits of what can be said without using mathematics. The truth is that we prepare a ground state of a Hamiltonian in the x-basis, but we always measure in the z-basis. When we would measure in the x-basis, we would get 0 (as in the introduction). But since we measure in the z-basis, we get 0+1. I will think about how we can resolve this with another interpretation. Happy if someone has a suggestion.}
% \rhinline{Update: Alexey and I added a paragraph in section 2.2 and 2.4, emphasizing the role of a ground state.}
Then, using a predefined annealing schedule, \revIII{the system is driven to a problem-specific energy landscape with a ground state that represents a candidate solution.}\label{q:QA-R3R11} To this end, the biases and couplers are tuned to encode the \gls{QUBO} coefficients. Finally, a measurement reveals a solution candidate. This process is typically repeated many times to generate an ensemble of solutions.

The biases and couplers of the \gls{Advantage} device have a limited effective resolution to encode the \gls{QUBO} coefficients, which means that a \gls{QUBO} can only be expected to be solved efficiently if its coefficients have a limited dynamic range (ratio between largest and smallest values; see \citealp{muecke2025}). Furthermore, the \gls{Advantage} device has \num{5760}~qubits, which represents the largest possible \gls{QUBO} instance that can be encoded. However, choosing a one-to-one correspondence between optimization variables and qubits for the encoding, as described above, is in fact not possible for all instances. Due to the limited availability of couplers, the effective number of required qubits can be much larger than the number of variables. This is an important practical issue, which is explained in more detail in the following.

\paragraph{Number of qubits needed.}\label{q:no-qubits-QA}

For a given \gls{QUBO} instance, the number of logical qubits corresponds to the number of binary optimization variables. 
The logical qubits need to be encoded with the physical qubits of the quantum annealer, which means that each logical qubit has to be assigned to a physical qubit and each correlation between logical qubits has to be assigned to a coupler.\label{q:27}
% begin quote 27!
As already mentioned above, the \gls{Advantage} device only supports couplers between specific physical qubits \revIII{(in which case they are called \emph{connected})}.
Therefore, a straightforward encoding of a sufficiently large and dense \gls{QUBO} instance may in fact require non-existing couplers.
% end quote
To overcome this limitation,
a logical qubit can also be encoded as a \emph{chain} comprising several
physical qubits \citep{venegasandraca2018} that are strongly coupled with each other.
The encoding task then corresponds to the graph-theoretic problem of finding
a \emph{minor embedding}~\citep{choi2008} of the \emph{problem graph} (where nodes are \gls{QUBO} variables and edges correspond to non-zero \gls{QUBO} coefficients) into the \emph{topology graph} of the device (where nodes are physical qubits and edges exist for every coupler). In this context, the chains are usually called
the \emph{branch sets} associated to the nodes of the problem graph.
If possible, chains are to be avoided
because \emph{chain breaks} are an important type of errors occurring in quantum annealers, where physical qubits 
representing the same logical qubit attain different states. Some \emph{post-processing} techniques to recover consistent solutions
have been proposed~\citep{pelofske2020}.
\Citet{gilbert2024conference} conducted experiments on crafted \gls{QUBO} instances whose problem graphs are subgraphs of the hardware topology, allowing for embedding without additional overhead.

% begin quote 8!
\revII{
The \gls{Advantage} device is designed according to the so-called \emph{Pegasus} topology~\citep{dattani2019a,boothby2020}, and the aforementioned works imply a polynomial-time algorithm for embedding a complete graph into this topology graph.
We extend this result with the following lemma, formulated in terms of the number
of physical qubits required for a given \gls{QUBO} instance. It is a direct
consequence of the formal description of the topology by \citet{boothby2020},
see \zcref{app:Pegasus} for further details.}
% end quote
  
\begin{lemma}\label{lm:pegasus-qubits}\zlabel{lm:pegasus-qubits}
  The problem graph~$G_Q$ of a \gls{QUBO} instance with $N$ variables can be embedded into a Pegasus graph~$G_T$ with
  \[N_{QA}(N)\defeq 24\Bigl\lceil\frac{N+10}{12}\Bigr\rceil\Bigl\lceil\frac{N-2}{12}\Bigr\rceil \le \frac{(N+21)(N+9)}{6}\]
  nodes such that the corresponding \gls{QUBO} instance
  can be solved on a quantum annealer with Pegasus topology
  using $N_{QA}(N)$ physical qubits. If $G_{Q}$ is a subgraph of the
  Pegasus graph, $N$ qubits suffice.
\end{lemma}

\label{q:9}%
% begin quote 9!
\revII{In practice, certain physical qubits of a quantum annealer
	might turn out to be ``broken'' as a result of an imperfect construction process, which
	means that these qubits cannot be used for computations.
	This is also the case for the \gls{Advantage} device.
	As a consequence, the effective topology reduces to a subgraph of the
	Pegasus topology graph, in which the nodes corresponding to broken qubits are removed.
	Deciding whether a given problem graph is a minor of
	such an irregular topology graph is $\NP$-complete \citep{lobe2021a}, and the
	embedding problem is usually solved heuristically~\citep{cai2014,zbinden2020}.
	Such heuristics typically aim to find embeddings that use few physical qubits
	while still maintaining some redundancy to strengthen the coupling of the
	chains and reduce the probability of errors~\citep{pelofske2024}.
	Some research works study the probability that a random problem graph can be successfully embedded into a fixed topology graph (corresponding to a fixed hardware configuration). This is known as the \emph{embedding probability}~\citep{sugie2021}.
	For further information and recent research results, we refer to \citet{zbinden2020,sugie2021,gomeztejedor2025,sinno2025}.}
% end quote

\subsection[QAOA-OPT: QUBO and QAOA]{\texorpdfstring{\Gls{IBM-OPT}: solving \glspl{QUBO} with gate-based \glspl{QC}  (\gls{IBM})}{QAOA-OPT: solving QUBOs with gate-based QCs (IBM)}}\label{sec:QAOA}
Our third approach, \gls{IBM-OPT}, makes use of the gate-based \glspl{QC} from \gls{IBM} to solve \glspl{QUBO} with \gls{QAOA}~\citep{farhi2014,grange2022,blekos2024}. For our numerical experiments, we have made use of two of the 127-qubit devices \gls{IBMCusco} and \gls{IBMNazca}.

\paragraph{Hardware.}\label{q:QAOA-hw}
The \gls{IBM} devices are based on superconducting technology. As universal gate-based \glspl{QC}, they can in principle run any gate-based quantum algorithm, which can be seen as a three-step process. First, all qubits are prepared in the ground state. Then, a sequence of gates is applied, each corresponding to a discrete control that alters the quantum state of the multi-qubit system. Finally, a measurement of the resulting state yields an outcome. A shot involves performing all these steps to obtain a single bitstring. A sequence of gates acting on a set of qubits is referred to as a \emph{quantum circuit} and provides all necessary instructions to operate a gate-based \gls{QC}. In other words, a quantum circuit controls the evolution from an initial quantum state to a final one, which is then measured.

\paragraph{Algorithm.}
We consider \gls{QAOA} to solve \glspl{QUBO}%
\footnote{\label{q:fn-PUBO}%
 % begin quote pubo
 \revIII{In addition to solving \glspl{QUBO}, variants of \gls{QAOA} can also take into account constraints~\citep{hadfield2018,hadfield2019,fuchs2022,Fuchs2024lxmixersqaoa,bucher2025}. Furthermore, \gls{QAOA} can also be used to solve \glspl{PUBO} \citep{grange2022}.}%
 % end quote
},
which is a hybrid quantum-classical algorithm that can be understood as an attempt to digitize adiabatic evolution in the sense of a \emph{digital \gls{QAA}}.
As a special implementation of a \gls{VQA}~\citep{cerezo2021,grange2022,blekos2024}, the key concept of \gls{QAOA} is that of a \emph{parameterized circuit}, a quantum circuit with gates that depend on real-valued parameters. Depending on the choice of parameters, executing a parameterized circuit on a \gls{QC} will provide different measurement outcomes.
The goal of \gls{QAOA} is to iteratively adjust the circuit parameters to guide the quantum system towards states that represent optimal or near-optimal solutions of the underlying optimization problem. This is realized along the lines of a typical \gls{QAA} by starting from a problem-agnostic ground state, a balanced superposition state of every qubit, and then applying a sequence of (parameterized) gates to arrive at the problem-specific state that encodes the solution. In contrast to an analog \gls{QAA}, \gls{QAOA} only mimics the adiabatic transition (which may or may not yield an approximation) using a discrete sequence of gates instead of a continuous control.
We follow the \revIII{unified workflow from \cref{sec:quantum workflows}, where a device-specific scheme is summarized in~\cref{fig:QAOA}.}\label{q:QAOA-highlight} Note that in the solution step, we perform not a single series of shots as in the two previous approaches, but a hybrid quantum-classical optimization loop.

\begin{figure}[ht]
 \centering
 \includegraphics[width=0.7\textwidth]{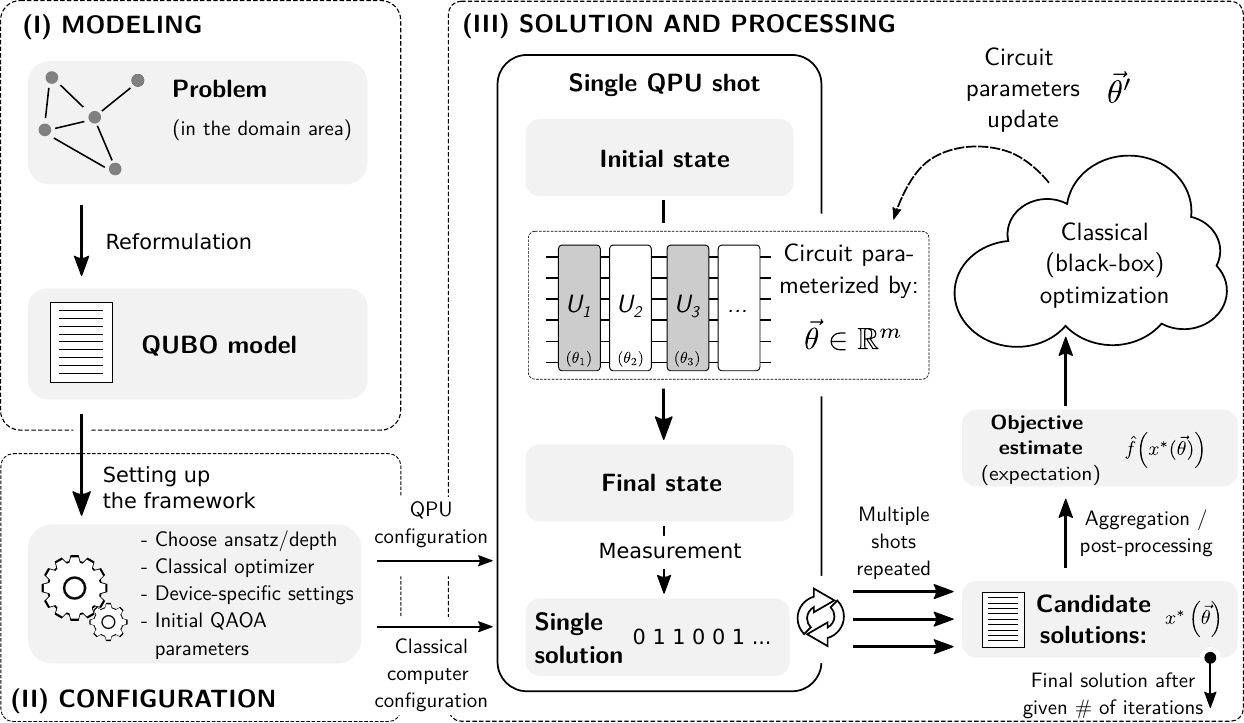}
 \caption{Workflow of the \gls{IBM-OPT} approach: solving \glspl{QUBO} with \gls{QAOA}.\label{fig:QAOA}}
\end{figure}

In the following, we explain \gls{QAOA} in more detail.
For a given \gls{QUBO} instance, each qubit encodes one optimization variable. The quantum circuit is then chosen in such a way that it represents a parameterized approximation of the annealing process. The specific circuit choice is also known as \emph{ansatz}, its \emph{depth} is a fixed parameter that denotes the number of repeating circuit elements (\ie, single gates or certain sequences of gates) in the ansatz. On the one hand, choosing a more complex ansatz (\eg, with greater depth) enables the circuit to explore more possibilities to mimic the adiabatic schedule. On the other hand, this exploration also requires tuning more parameters, which becomes computationally challenging and typically requires much more efforts to converge to a solution. These two opposing factors must be weighed against each other by the user to select a suitable ansatz.

The role of the \gls{QC} in the hybrid \gls{QAOA} setup is to run parameterized circuits and provide the measurement results, which can then be used to estimate the expectation value of the system's energy. The classical computer, on the other hand, optimizes the circuit parameters based on the quantum measurements and proposes updated parameters that aim to reduce the expected energy.

Summarized, the task of the classical computer is to effectively perform a black-box optimization under uncertainties. The objective function to be minimized is the expected energy of the quantum system and the optimization variables are the real-valued circuit parameters. In principle, this problem can be solved with any suitable classical optimization algorithm. (\Eg, gradient descent methods, \gls{SPSA}, and so on. See Section D by~\cite{cerezo2021} for a compact overview.) It is a black-box optimization because the problem structure is effectively unknown. The classical optimizer can query the \gls{QC} to provide the expected energy for a given choice of circuit parameters. Gradients of the objective function can be obtained using finite differences or specialized techniques~\citep{crooks2019}. Uncertainties arise because of two reasons: first, the expected energy is only estimated with finite samples (each representing a series of shots) and, second, \gls{NISQ} devices suffer from hardware-related uncertainties. Since evaluations on a \gls{QC} are costly, it is typically desired to find sufficiently good circuit parameters with as few iterations (or shots) as possible. Once the optimal parameters have been found, the \gls{QC} can be used to sample solution candidates.

For our \gls{IBM-OPT} approach, we only consider a foundational implementation of
\gls{QAOA} to explore the basic performance.
However, from a practical perspective, fine-tuning the algorithm is an important but challenging aspect. For example, there are separate works on circuit parameter optimization \citep{zhou2020a} and initial parameter values \citep{sack2021,sack2023}.
% , and the choice of classical optimizers \citep{cerezo2021}\sven{Is this true? From the abstract it doesn't seem to be specifically addressing this point.}.
Improvements to the original algorithm include extensions of the aggregation function \citep{barkoutsos2020}, counterdiabatic driving \citep{chandarana2022}, and warmstarts \citep{egger2021}.

In addition, there are two general technical challenges when running quantum circuits on \gls{IBM} devices that also apply to all \gls{QC} evaluations within \gls{QAOA}:
\begin{enumerate}
 \item Limited gates: Only specific one-qubit and two-qubit gates, the so-called \emph{basis gates}, can be executed. The set of basis gates is universal in the sense that any other gate can be decomposed into a sequence of basis gates.
 \item Limited connectivity: Two-qubit gates can only be executed for specific pairs of physical qubits according to the prescribed hardware connectivity. These gaps can be bridged by the use of additional gates.
\end{enumerate}
To overcome these obstacles, a classical preprocessing step called \emph{transpilation} is necessary to transform a given quantum circuit into a hardware-compliant quantum circuit.
In fact, there are infinitely many hardware-compliant quantum circuits that lead to the same measurement outcomes on an idealized (\ie, noise-free) \gls{QC}. On actual \gls{IBM} devices, however, some mathematically equivalent circuits will perform better than others. For example, circuits with fewer basis gates typically perform better. Moreover, each physical qubit exhibits an individual level of noise, which might also change over time~\citep{baheri2022}. Therefore, the practical performance of a quantum circuit may also depend on which physical qubits are used for its implementation. Thus, the effective goal of the transpilation is to find a hardware-compliant and well-performing representation of a given quantum circuit~\citep{li2019,wilson2020,hua2023,nation2023,waring2024,quantum7010002}.
Essentially a classical optimization problem, transpilation within \gls{IBM-OPT} is solved in the \gls{IBM} software package \emph{Qiskit}~\citep{ibmquantum2023,javadiabhari2024} with multi-step heuristics that include mapping logical qubits to physical qubits, decomposing all gates into basis gates, routing qubits with additional gates according to hardware connectivity, and an overall circuit optimization to improve the performance.
% begin quote qaoa connectivity
\revIII{Note that similarly to the previous two approaches, qubit connectivity manifests as a separate issue here.}
% end quote

\paragraph{Number of qubits needed.}\label{q:QAOA-qubits}

In analogy to \gls{D-Wave-OPT}, the number of logical qubits for \gls{IBM-OPT} corresponds to the number of binary optimization variables for a given \gls{QUBO} instance, which can then also be translated one-to-one into physical qubits, as summarized in \cref{rem:QAOA-qubits}.
\begin{remark}\label{rem:QAOA-qubits}
  An arbitrary \gls{QUBO} instance with $N$ variables that is solved with
  \gls{QAOA} requires $N$ physical qubits. While transpilation might affect the
  output quality, it does not necessarily require additional physical qubits.
\end{remark}

\section{Application to selected problems}\label{sec:qubits}\zlabel{sec:qubits}

In this section, we focus on three classes of optimization problems: \gls{UD-MIS}, the \gls{MaxCut}, and the \gls{TSP}. All three are graph-based problems, meaning that each instance is defined on an undirected graph~$G := (V, E)$, where~$V$ is the set of nodes and~$E \subseteq \binom{V}{2}$ is the set of edges (with no duplicates or self-loops).
For each class, we present a formal problem definition as a \gls{QUBO} and study the required number of physical qubits it takes to solve
them with the three quantum-powered optimization approaches from \cref{fig:workflows}.
Solving the
same \gls{QUBO} instance may require
a different number of physical qubits
for each approach, depending on the problem structure and
device implementation details. 

\label{q:6}%
% begin quote 6 !
\revI{We emphasize that we do not present novel ways of modeling the problems,
	nor do we provide a comprehensive overview of possible quantum-aware modeling
	approaches for optimization problems. In fact, finding suitable formulations
	of classical optimization problems that allow an effective treatment with
	quantum optimization is a task known as \emph{problem encoding} and constitutes a
	research direction of its own. For more general discussions of different
	formulations for a wide range of optimization problems in the context of
	quantum computing and \glspl{QUBO}, see for example
	\cite{lucas2014,dominguez2023,glover2022}. Furthermore, \citet{ruan2020b},
	\citet{gonzalez-bermejo2022,salehi2022}, and \cite{codognet2024} focus on \gls{TSP},
	while \citet{hadfield2021} considers more general classes of functions and aims to
	provide a ``design toolkit of quantum optimization''. \Citet{hadfield2017}
	discuss encoding constraints without penalty terms, specifically in the
	context of \gls{QAOA}. \cite{padmasola2025solvingtravelingsalesmanproblem} focus on \gls{TSP},
	but compare and contrast a few other quantum technologies. Possible problem encoding strategies are diverse, and we limit ourselves here to common formulations suitable in the context of the considered quantum-powered optimization approaches.}
% end quote

We start by reviewing \glspl{QUBO}. An instance is defined by
a symmetric matrix $Q \in \mathbb{R}^{N\times N}$ as:
\begin{align} \label{eq:QUBO} \zlabel{eq:QUBO}
	\min_x\quad x^{T} Q x, \quad x\in\{0,1\}^{N}
\end{align}
with $N$ binary decision variables $x \defeq (x_{1},\ldots, x_{N})$.
Note that the requirement to represent problems in this form is in
fact not too restrictive. \Citet{lucas2014} discusses so-called \emph{Ising
  formulations}, which are very close to \gls{QUBO}, for many combinatorial
optimization problems, including the 21 $\NP$-hard problems from \citeauthor{karp1975}'s \citeyearpar{karp1975} list.
A tutorial by \citet{glover2022} focuses on formulating
combinatorial optimization problems as \glspl{QUBO}.

\subsection{\glsentrytext{UD-MIS}}\label{sec:UD-MIS}\zlabel{sec:UD-MIS}

Given an undirected graph~$G:=(V, E)$,
the \gls{MIS} problem consists of finding a subset of nodes of maximum cardinality such that no two nodes in the subset are adjacent.
The problem can be modeled as an \gls{ILP} using one binary variable~$x_i$ for each node~$i \in V$, indicating whether it constitutes a part of the solution:
\begin{align}
	\max\quad & \sum_{i\in V} x_{i} &&  \label{eq:MIS-ILP} \\
	\subjecto \quad & x_{i} + x_{j} \leq 1 &&\text{ for all } \{i,j\}\in E,\nonumber\\
	& x_{i} \in \{0,1\} &&\text{ for all } i\in V. \nonumber
\end{align}
If the graph $G$ is a unit disk graph, the problem is also called \gls{UD-MIS}. In a unit disk graph, each node corresponds to a point in the plane, and an edge exists between two nodes if and only if the Euclidean distance between their corresponding points is at most one. In general, given a graph from an arbitrary \gls{MIS} instance, it is $\NP$-hard to decide whether it can be represented as a unit disk graph or not \citep{breu1998}. Heuristics such as a force-based approach \citep{coelho2022} can be used to find a unit disk representation. \revII{Alternatively, \gls{MIS} on an arbitrarily connected graph can also be transformed into a \gls{UD-MWIS} by increasing the graph size \citep{nguyen2023,bombieri2025}.} The \Gls{MWIS} problem is a generalization of \gls{MIS}, which takes into account a non-negative weight~$w_i \in \mathbb{R}_{\geq 0}$ for each node~$i \in V$. A \Gls{MWIS} instance is therefore defined by a weighted graph~$G:=(V, E, w)$ and reduces to a \Gls{MIS} instance for units weights.\footnote{As before, if the graph $G$ is a unit disk graph, the problem is called \gls{UD-MWIS}.}
% \sven{I find it strange that this name is introduced here, when it was already used in the previous paragraph.}
In analogy to formulation~\eqref{eq:MIS-ILP}, the \Gls{MWIS} problem can be formulated as:
\begin{align}
	\max\quad & \sum_{i \in V} w_i x_{i} &&  \label{eq:WMIS-ILP} \\
	\subjecto \quad & x_{i} + x_{j} \leq 1 &&\text{ for all } \{i,j\}\in E,\nonumber\\
	& x_{i} \in \{0,1\} &&\text{ for all } i\in V. \nonumber
\end{align}

In the following, we only consider the unweighted problem class.
To obtain an equivalent \gls{QUBO} for formulation~\eqref{eq:MIS-ILP}, we can introduce constraints as quadratic penalty terms corresponding to all edges $\{i,j\} \in E$ and switch to a minimization, which results in:
\begin{flalign}\label{eq:UDMIS-QUBO}
  \max_{x\in\{0,1\}^{|V|}}\quad \sum_{i\in V} x_{i} - M\sum_{\{i,j\}\in E}x_{i}x_{j}
  = -\min_{x\in\{0,1\}^{|V|}} x^{T} Q x, ~ \text{ where } Q_{ij} =
  \begin{cases}
    -1 &\text{ if } i=j,\\
    M/2  &\text{ if } \{i,j\}\in E,\\
    0 &\text{ otherwise.}
  \end{cases}
\end{flalign}
In the definition of $Q_{ij}$, we consider all ordered pairs $i,j$ to obtain a symmetric matrix.
\label{q:17}%
% relevant quote 17 main text ! (not cited)
\revIII{This formulation is equivalent to~\eqref{eq:MIS-ILP} for large
  enough $M$ (\eg, $M = |V| + 1$, see \zcref{app:big-M}).}\label{q:17text}
% end relevant quote
This yields the necessary number of
binary variables, as stated in the following lemma.

\begin{lemma}\label{lm:UDMIS-qubits}
  For a graph $G=(V, E)$, the \gls{UD-MIS} instance given by formulation~\eqref{eq:MIS-ILP}
  can be solved on a neutral-atom-based machine with $N \defeq |V|$ physical qubits. The
  respective \gls{QUBO} formulation \eqref{eq:UDMIS-QUBO} requires $N$
  logical qubits. Therefore, it can be solved on a quantum annealer with Pegasus
  topology using at most $24\lceil\frac{N+10}{12}\rceil \lceil\frac{N-2}{12}\rceil$ physical qubits
  or on a general gate-based \gls{QC} with $N$ physical qubits.
\end{lemma}

\begin{proof}
  Follows from \cref{lm:Rydberg-qubits,lm:pegasus-qubits,rem:QAOA-qubits}.
\end{proof}

\subsection{\glsentrytext{MaxCut}}\label{sec:MaxCut}

Given an undirected graph~$G = (V, E, w)$, where $w_{ij} = w_{ji} \in \mathbb{R}_{\geq 0}$ denotes the weight of the edge $\{i,j\}\in E$, the \gls{MaxCut} problem consists of finding a partition of the node set~$V$ into two disjoint subsets that maximizes the total weight of the edges connecting nodes in different subsets.
This can be reformulated as the following \gls{ILP}:
follows.
\begin{align}
	\max \quad & \sum_{\{i,j\}\in E} e_{ij}w_{ij}, \label{eq:MaxCut-ILP}\\
	\subjecto\quad & e_{ij} \leq x_{i} + x_{j} \text{ for all } \{i,j\} \in E, \label{c:e1}\\
	& e_{ij} \leq 2 - (x_{i} + x_{j}) \text{ for all } \{i,j\}\in E,\label{c:e2}\\
	& x_{j}, e_{ij} \in \{0, 1\} \text{ \revIII{ for all }} \{i,j\}\in E. \nonumber
\end{align}
Here, variables $x_{j}$ denote which set in the partition vertex $j$
belongs to, and the constraints~\eqref{c:e1} and \eqref{c:e2} ensure that in an
optimal solution, the variable $e_{ij} = 1$ if and only if $\{i,j\}$ is in the
cut (\ie, exactly one of $x_{i}$ and $x_{j}$ equals $1$), and $0$ otherwise.
We assume that $i<j$ for all double-indexed variables~$e_{ij}$.

\label{q:18}%
% relevant quote 18! (not cited)
\revIII{Based on the formulation given, \eg, in the textbook chapter of \citet{laurent1997a}}, the problem~\eqref{eq:MaxCut-ILP} can be posed as a \gls{QUBO}:
% end relevant quote
\begin{align}\label{eq:MaxCut-QUBO}
	\max_{x\in\{0,1\}^n} \sum_{\{i,j\}\in E} w_{ij}(x_i + x_j - 2x_i x_j) =
  -\min_{x\in\{0,1\}^{n}} x^{T} Q x,  \text{ where }
  Q_{ij} =
  \begin{cases}
    \sum\limits_{k:\{i,k\}\in E} (-w_{ik}) &\text{ if } i=j, \\
    w_{ij} &\text{ if } \{i,j\} \in E, \\
    0 &\text{ otherwise.}
  \end{cases}
\end{align}
Whenever $x_{i}=1$ and $x_{j}=0$ (or vice versa), the term
$w_{ij}(x_{i}+x_{j} - 2x_{i}x_{j})$ contributes $w_{ij}$ to the objective.
Otherwise, if $x_{i}=x_{j}$, the term's contribution will be zero.

\begin{lemma}\label{lm:MaxCut-QUBO}
  For a graph $G=(V, E, w)$, the \gls{QUBO} formulation of
  the \gls{MaxCut} problem \eqref{eq:MaxCut-QUBO}, requires $N \defeq |V|$ logical qubits. Therefore, it can be
  solved on a quantum annealer with Pegasus topology using at most
  $24\lceil\frac{N+10}{12}\rceil\lceil\frac{N-2}{12}\rceil$ physical qubits, on a neutral-atom-based device using $4N^2$ physical qubits, or on a general gate-based device
  using $N$ physical qubits.
\end{lemma}

\begin{proof}
 Follows from \cref{lm:Rydberg-qubits,lm:pegasus-qubits}, and Remark~\ref{rem:QAOA-qubits}.
\end{proof}

Note that for \gls{MaxCut}, the number of binary variables of the \gls{QUBO} is only
linear in the number of nodes, while for the \gls{ILP} it may be quadratic, making
this problem particularly suitable for \glspl{QC}.

\subsection{\glsentrytext{TSP}}\label{sec:TSP}

Given a complete undirected graph $G=(V,E,w)$, where $w_{ij} = w_{ji} \in \mathbb{R}_{\geq 0}$ denotes the weight of the edge $\{i,j\}\in E$ and $N \defeq |V|$ the number of nodes, the \gls{TSP} seeks for a shortest possible route in
terms of the cumulative edge weight that visits each vertex exactly once and
returns to the start node. On classical computers, the \gls{TSP} is typically solved by
sophisticated \gls{ILP} techniques, combined with heuristics to find a good
initial solution~\citep{applegate2006}. Different \gls{ILP} formulations
have been developed for this problem, which makes it convenient to use it as an
example to compare and contrast a few alternative formulations in the context of
quantum computing.

\paragraph{Dantzig-Fulkerson-Johnson formulation.}%

The classical formulation \citep{dantzig1954} used in most \gls{OR} textbooks
implies binary variables~$x_{ij}$, $i < j$, for all edges $\{i,j\} \in E$, indicating that an edge
is traversed by the tour:
\begin{align}
  \min & \sum_{i<j} c_{ij} x_{ij},\label{obj:DFJ}\tag{DFJ}\\
  \subjecto & \sum_{j: j>i} x_{ij} + \sum_{j: j<i}x_{ji} = 2 && \text{for all } i=1,\ldots,N,\label{DFJ:visits}\\
       & \sum_{\{i,j\} \in E(S)} x_{ij} \leq |S|-1 && \text{for all } S\subseteq V: 3 \le |S| \le \frac N 2, \label{DFJ:subtour}\\
       &x_{ij} \in \{0,1\}&& \text{for all } i<j,\nonumber
\end{align}
where $E(S)\defeq\bigl\{\{i,j\}\bigm| i,j \in S, i<j\bigr\}$ denotes a
set of all edges between nodes in $S \subseteq V$. There are $N$ degree
constraints \eqref{DFJ:visits}, and if $N > 5$, we also have subtour elimination
constraints \eqref{DFJ:subtour}. Note that inequalities~\eqref{DFJ:subtour} for subsets $S$
of size~$2$ is equivalent to the upper bound of the variables. In the original
formulation by \citet{dantzig1954} given above, we get redundant constraints if
$N$ is even, because for every set $S$ with $|S| = \frac N 2$ we have a
constraint corresponding to $S$ and a constraint corresponding to $V \setminus S$. If we
remove the redundant constraints, \eg, by only taking subsets of size
$\frac N 2$ containing node~1, we are left with
$2^{N-1} - 1 - N - \frac{N(N-1)}{2}$ subtour elimination constraints.

To create an equivalent \gls{QUBO} formulation, one needs to
incorporate constraints into the objective function as penalty terms, involving
binary variables only.
% relevant quote 19 main text! (not cited)
\revIII{As before, the \gls{QUBO} formulation \eqref{eq:TSP-QUBO} will be equivalent to \eqref{obj:DFJ} for large enough $M$, \eg, $M=\max c_{ij}+1$ (see \zcref{app:big-M} for details).}\label{q:19txt}
% end relevant quote
Each constraint from the family \eqref{DFJ:visits} induces a term of
the form $M (\sum_{j: j>i} x_{ij} + \sum_{j: j<i}x_{ji}-2)^2$, but no new variables. The subtour
elimination constraints~\eqref{DFJ:subtour} require slack variables in binary
representation. A constraint corresponding to subset size $|S|$ requires
$\log_{2}|S| \leq \log_{2} N$ new binary variables to represent the possible integer slack.
Therefore, the family of constraints~\eqref{DFJ:subtour} requires at most
$\bigl(2^{N-1} - 1 - N - \frac{N(N-1)}{2}\bigr)\log_{2}N$ additional binary
variables. More careful calculation presented in \zcref{app:TSP-logical-qubits}
yields the following result.
\begin{lemma}\label{lem:no_qubits_DFJ}\zlabel{lem:no_qubits_DFJ}
 The Dantzig-Fulkerson-Johnson \gls{ILP} can be reformulated as a \gls{QUBO} having
 \[n^{\text{DFJ}}\defeq \frac{N(N-1)}{2} + \sum_{k=3}^{\lceil N/2 \rceil-1} \binom N k \cdot \bigl\lfloor \log_2(k-1) + 1 \bigr\rfloor + \frac{1+(-1)^N}{4} \cdot \binom{N}{\lfloor N/2 \rfloor} \cdot \bigl\lfloor \log_2(N/2-1) + 1 \bigr\rfloor\]
 binary variables. Therefore, it requires $n^{DFJ}$ logical qubits and can be
 solved on a neutral-atom-based \gls{QC} with $4(n^{DFJ})^{2}$ physical
 qubits, a quantum annealer with
 $24\lceil\frac{n^{DFJ}+10}{12}\rceil\lceil\frac{n^{DFJ}-2}{12}\rceil$ physical qubits,
 or on a general gate-based \gls{QC} with $n^{DFJ}$ physical qubits.
\end{lemma}

Finally, we remark that formulation~\eqref{obj:DFJ} is rarely passed completely
to a solver even as an integer program, but is instead used as basis for row
generation approaches. It might be possible to apply a similar approach,
generating penalty terms and slack variables on the fly, in an iterative
quantum-classical hybrid algorithm, which might constitute a potential
direction for further research.

\paragraph{Miller-Tucker-Zemlin formulation.}%

An alternative approach \citep{miller1960} introduces variables $u_j\in\mathbb{N}$ denoting the number of each node $j\in V$ in a tour (\eg, $u_5=7$ means that node $5$ is the seventh city in the tour). We also use variables $x_{ij}$ indicating whether edge $\{i,j\}$ is traversed. This time, however, we also have to keep track of the traversal direction, \ie, we have two different variables $x_{ij}$ and $x_{ji}$ for all $i,j \in V$ with $i \neq j$. Here, $x_{ij} = 1$ indicates that $j$ is visited directly after $i$. Since we assume that a tour starts at node $1$, for all $i \in V,\ j \in V \setminus \{1\}$, we must have $u_j \geq u_i +1$ if $x_{ij}=1$.
Note that such constraint would be violated by any subtour not passing through
node $1$. This yields:
\begin{align}
    \min & \sum_{i=1}^N \sum_{j:j\neq i} c_{ij}x_{ij},\label{obj:MTZ}\tag{MTZ}\\
    \subjecto & \sum_{j: j\neq i} x_{ij} = 1 && \text{for all } i=1,\ldots, N,\nonumber\\
    & \sum_{i: i\neq j} x_{ij} = 1 && \text{for all } j=1,\ldots, N,\nonumber\\
    & u_i - u_j + (N-1) x_{ij} \leq N-2 && \text{for all } 2\leq i, j\leq N,\ i \neq j, \label{eq:MTZse} \\
    &x_{ij}\in\{0,1\}&& \text{for all } 1\leq i, j\leq N,\ i \neq j,\nonumber \\
    &u_j \geq 2 && \text{for all } j=2,\ldots, N \nonumber,
\end{align}
where we assume the distance matrix to be symmetric: $c_{ij}=c_{ji}$. We have $N(N-1)$ binary variables and $(N-1)$ integer variables. However, now there is only a polynomial number of constraints, namely, $2N$ degree constraints and $(N-1)(N-2)$ constraints~\eqref{eq:MTZse} involving the new variables. In the above formulation we did not demand integrality of the $u_j$ variables because it is not necessary and for classical hardware, omitting this constraint may accelerate the solver. In the following transformation to a \gls{QUBO}, we will however assume that the $u_j\in\mathbb{N}$ and, \revIII{again, that $M$ is large enough, \eg, $M=2\max c_{ij}+1$ (see \zcref{app:big-M})}\label{q:20}.

\begin{lemma} \zlabel{lem:no_qubits_MTZ}\label{lem:no_qubits_MTZ} The Miller-Tucker-Zemlin integer program can
  be reformulated as a \gls{QUBO} with
  \[n^{MTZ}\defeq (N-1) \bigl((N-1) \cdot \lfloor \log_2(N-2) \rfloor + 2N - 3\bigr)\] binary
  variables. So, it requires $n^{MTZ}$ logical qubits and can be solved on a
  neutral-atom-based \gls{QC} with $4(n^{MTZ})^{2}$ physical qubits, a quantum
  annealer with $24\lceil\frac{n^{MTZ}+10}{12}\rceil\lceil\frac{n^{MTZ}-2}{12}\rceil$
  physical qubits, or a general gate-based \gls{QC} with $n^{MTZ}$ physical
  qubits.
\end{lemma}
\begin{proof}
  The first two families of constraints give us $2N$ penalty terms in the
  objective. Each integer variable $u_j$ must be represented by binary
  variables. In this case we can replace $u_j$ by
  $2 + \sum_{r=0}^{\lfloor \log(N-2) \rfloor} 2^r b_{j,r}$, introducing $\lfloor \log(N-2) \rfloor + 1$
  binary variables~$b_{j,r}$. The slack in constraint~\eqref{eq:MTZse} can take values from
  $0$ (if $j$ is the successor of $i$) to $2(N-2)$ (if $u_i = 2$, $u_j = N$). To
  accommodate this range, we need $\lfloor \log_2(N-2)\rfloor + 2$ binary variables. Since
  we have this for every pair of $i,j \in \{2,\dotsc,N\}$ with $i \neq j$, the total
  number of such variables is $(N-1) \cdot (N-2) \cdot (\lfloor \log_2(N-2)\rfloor + 2)$. Therefore,
  the original \gls{ILP} can be transformed into a \gls{QUBO} model with
  $2N+(N-1)(N-2)$ penalty terms, involving
 \[(N-1) \cdot (\lfloor \log_2(N-2) \rfloor + 1) + (N-1) \cdot (N-2) \cdot (\lfloor \log_2(N-2) \rfloor + 2) = (N-1) \bigl((N-1) \lfloor \log_2(N-2) \rfloor + 2N - 3\bigr)\]
 binary variables.
\end{proof}

\paragraph{Quadratic assignment formulation.}

While these \gls{ILP} approaches are very successful on classical computers, a
direct translation of the formulations to \glspl{QUBO} results in a large number
of variables and thus of physical qubits required. In quantum computing, it
makes sense to move away from \gls{ILP}-based approaches and consider alternative
strategies. This is well illustrated by the \gls{TSP}, which actually has a very
natural and compact formulation as a quadratic program, more specifically a
quadratic assignment problem \citep{lawler1963,garfinkel85}. Let us introduce binary variables $y_{ik}$, which
equal one if and only if node $i$ is visited during step $k$ in the tour. Since
we are interested only in tours that return to the original location, assume
without loss of generality that $y_{11}=1$. These ideas yield the following
formulation:
\begin{align}
  \min & \sum_{j=2}^N c_{1j} y_{j2} + \sum_{i=2}^N \bigg(\sum_{\substack{j=2\\j\neq i}}^N c_{ij} \sum_{k=2}^{N-1} y_{ik}y_{j(k+1)} + c_{1i} y_{iN}\bigg),\label{obj:TSP-QUBO}\tag{QAP}\\
  \subjecto &\sum_{k=2}^{N} y_{ik} = 1\quad \textrm{ for all } i=2,\ldots, N.\label{eq:singleMoment}\\
       &\sum_{i=2}^{N} y_{ik} = 1\quad \textrm{ for all } k=2,\ldots, N. \label{eq:singleCity}\\
    &y_{ik}\in\{0,1\},\quad \text{ for all } i=2,\ldots, N \textrm{ and } k=2,\ldots, N \nonumber
\end{align}
Each coefficient $c_{ij}$ encodes the cost of edge $\{i,j\}\in E$. Whenever nodes $i$ and $j$ appear as consecutive steps in the tour, $y_{ik}=y_{j(k+1)}=1$ for some moment $k$, which contributes $c_{ij}$ to the objective. The sums in the first and the last terms in~\eqref{obj:TSP-QUBO} implement the same logic for the first step in the tour (from node~$1$ to $j$), and the last one (returning from node~$i$ to $1$), respectively. Constraints~\eqref{eq:singleMoment} and~\eqref{eq:singleCity} ensure that a solution represents a permutation of cities. It yields the following \gls{QUBO}:
\begin{flalign}
  \min & \sum_{j=2}^N c_{1j} y_{j2} + \sum_{i=2}^N \bigg(\sum_{\substack{j=2\\j\neq i}}^N c_{ij} \sum_{k=2}^{N-1} y_{ik}y_{j(k+1)} + c_{1i} y_{iN}\bigg) +  M\sum_{i=2}^{N}\bigg(\sum_{k=2}^{N} y_{ik} - 1\bigg)^{2} + M  \sum_{k=2}^{N}\bigg(\sum_{i=2}^{N} y_{ik} - 1\bigg)^{2}\label{eq:TSP-QUBO} \\
       &y_{ik}\in\{0,1\},\quad i=2,\ldots, N \textrm{ and } k=2,\ldots, N. \nonumber
\end{flalign}
% relevant quote 21txt !
\revIII{As before, this \gls{QUBO} problem is equivalent to the original
  formulation in the sense that optimal solutions coincide, if the penalty
  coefficient $M$ is chosen to be large enough, \eg,
  $M=\frac{N(N-1)}{2}\max c_{ij}+1$ (see \zcref{app:big-M}).}\label{q:21txt}
% end relevant quote
  Formulation~\eqref{obj:TSP-QUBO} involves $(N-1)^2$ binary variables, $2(N-1)$
  linear equality constraints, and a quadratic objective with $(N-1)((N-2)^2+2)$
  terms, which we summarize in the following result.
\begin{lemma}\label{lm:TSP-qubits}
  \gls{TSP} over a complete graph $G=(V, E)$ can be formulated as a \gls{QUBO}
  with $(N-1)^{2}$ binary variables. Therefore, it requires $(N-1)^{2}$ logical
  qubits and can be solved on a neutral-atom-based \gls{QC} with $4(N-1)^{4}$
  physical qubits, a quantum annealer with
  $24\lceil\frac{(N-1)^{2}+10}{12}\rceil\lceil\frac{(N-1)^{2}-2}{12}\rceil$ physical qubits, or a
  general gate-based \gls{QC} with $(N-1)^{2}$ physical qubits.
\end{lemma}

\begin{table}[!ht]
{\newcolumntype{Y}{>{\centering\arraybackslash}X}
  \setlength{\abovedisplayskip}{0pt}
  \setlength{\belowdisplayskip}{0pt}
  \setlength{\abovedisplayshortskip}{0pt}
  \setlength{\belowdisplayshortskip}{0pt}
\begin{threeparttable}
   \setlength{\extrarowheight}{0pt}
   \caption{\label{tab:formulations} Number of variables for a \gls{TSP} instance
     with $N$ nodes, depending on the chosen \gls{QUBO} formulation. The corresponding number of required qubits is provided by \cref{lm:TSP-qubits}.}
  \begin{tabularx}{\linewidth}{lYYY}
    \toprule
    & Dantzig-Fulkerson-Johnson \cref{obj:DFJ} & Miller-Tucker-Zemlin \cref{obj:MTZ} & Quadratic assignment \cref{obj:TSP-QUBO}\\
    \midrule
    Constraints: &&&\\
    \quad{}equality & $N$ & $2N$ & $2(N-1)$\\
    \quad{}inequality & $2^{N-1}-1-\frac{N(N+1)}{2}$ & $(N-1)(N-2)$& ---\\[1ex]
    Variables: &&&\\
    \quad{}binary & $N(N-1)/2$ & $N(N-1)$ & $(N-1)^{2}$\\
    \quad{}integer & --- & $N-1$ & ---\\
    \midrule
    \multicolumn{4}{l}{\textbf{QUBO:} the number of variables required}\\
     \quad{}per binary & 1 & 1 & 1 \\
     \quad{}per integer\tnote{(a)} & $\leq \log_2 N$ & $\leq \log_2 N$ & --- \\
    \quad{}total & $O(2^{N}\log_{2}N)$ & $O(N^{2}\log_{2}N)$ & $O(N^2)$\\
     \quad{}exact formula & see \cref{lem:no_qubits_DFJ} & see \cref{lem:no_qubits_MTZ} & see \cref{lm:TSP-qubits}\\
    \midrule
    \multicolumn{4}{l}{\emph{\textbf{Example:} a \gls{TSP} on $N=10$ nodes.}}\\
    Constraints: &&&\\
    \quad{}equality & $\num{10}$ & $\num{20}$ & $\num{18}$\\
    \quad{}inequality & $\num{456}$ & $\num{72}$& ---\\[1ex]
    Variables: &&&\\
    \quad{}binary & $\num{45}$ & $\num{90}$ & $\num{81}$\\
    \quad{}integer & $\num{456}$ & $\num{9}$ & ---\\
    \multicolumn{4}{l}{\textbf{QUBO:} the number of variables required}\\
     \quad{}per binary & 1 & 1 & 1 \\
     \quad{}per integer & $\leq 4$ & $\leq 4$ & --- \\
    \quad{}total & $\sim \num{1100}$ \tnote{(b)} & $\sim \num{400}$ & $\num{81}$\\
    \bottomrule
  \end{tabularx}
  \begin{tablenotes}
    \item[a] This includes slack variables (one per inequality constraint).
    \item[b] The value of $456\times4 + 45=\num{1869}$ is an upper bound, but we
    used a refined estimate based on \cref{lem:no_qubits_DFJ}.
  \end{tablenotes}
\end{threeparttable}}
\end{table}

Although less widespread in
\gls{OR} literature, these three formulations yield \gls{QUBO} instances of significantly different sizes for the same problem, leading to different qubit requirements in a quantum computing context. A summary is provided in \cref{tab:formulations},
where we list the general expressions in terms of the number of nodes
$N=|V|$, and illustrative numerical values for $N=10$. 
For example, a naive,
full implementation of the formulation~\eqref{obj:DFJ} requires the fewest binary variables, but
the exponential number of inequality constraints contributes to an exponential
number of variables in the \gls{QUBO} instance, which translates to more than
a thousand for $N=10$ nodes. Furthermore, the formulation~\eqref{obj:MTZ} reduces the number
of inequality constraints and consequently brings the number of variables in the \gls{QUBO}
instance down to around \num{400}. Finally, the formulation~\eqref{obj:TSP-QUBO}
does require slack variables at all, resulting in a relatively
compact \gls{QUBO} instance with only \num{81} variables.

\subsection{Summary}\label{sec:qubits-summary}

The problem classes we consider in \cref{sec:UD-MIS,sec:MaxCut,sec:TSP} scale differently in terms of the number of binary variables with respect to the graph size. Depending on the chosen quantum-powered optimization approach from \cref{sec:quantum workflows}, the number of binary variables can be associated with a number of physical qubits that are necessary to run the problem.
Upper bounds on the numbers of physical qubits are provided
by~\cref{lm:Rydberg-qubits,lm:pegasus-qubits,rem:QAOA-qubits}. These bounds are not necessarily tight and depend on
properties of the \gls{QUBO} matrix~$Q$, such as sparsity. A more detailed discussion is provided in \zcref{app:instances}.
An overview of the qubit requirements
is presented in \cref{tab:qubits-embeddings}.

In summary, the qubit requirements are significant for all three problem classes, considering the capacities of currently available quantum devices.
Especially for \gls{TSP}, the quadratic scaling leads to very demanding requirements. The situation is better
for \gls{MaxCut} and \gls{UD-MIS}, which both scale only linearly.
For \gls{QuEra-OPT}, both \gls{TSP} and \gls{MaxCut} lead to prohibitively high qubit requirements.
On the other hand, we
were able to solve fairly large \gls{UD-MIS} instances on the neutral-atom-based device, as we will present in the next section.

Overall, given the necessities of the device-dependent setups and the current state of the technology, we see \gls{D-Wave-OPT} as the most practice-ready approach of the three. Similarly, \gls{IBM-OPT} represents a highly promising research direction with great versatility, which poses relatively modest qubit requirements. However, the available hardware devices can only be used to solve problems of a very limited scale. While \Gls{QuEra-OPT} is a feasible approach for \gls{UD-MIS}, its applicability to other problem classes requires a large number of physical qubits.

Beyond the mere qubit requirements, there are other factors that have to be taken into account when choosing a suitable quantum-powered optimization approach. In the next section, we will therefore conduct numerical experiments to evaluate some practical aspects of quantum optimization.

\afterpage{%
\begin{landscape}
  \begin{table}
 \begin{threeparttable}
   \caption{\label{tab:qubits-embeddings}Number of the binary variables and number of physical qubits necessary, depending on the quantum-powered approach and problem class.}
 \begin{tabularx}{\linewidth}{XX|XXX}
   \toprule
   \multicolumn{2}{c}{\textbf{Number of variables}} & \multicolumn{3}{c}{\textbf{Required number of physical qubits}} \\
   Original formulation & \gls{QUBO} & \gls{QuEra-OPT} (\cref{sec:rydberg}) & \gls{D-Wave-OPT} (\cref{sec:QA}) & \gls{IBM-OPT} (\cref{sec:QAOA})\\
   && {\footnotesize \gls{UD-MIS} representation, \citep{nguyen2023}} & {\footnotesize clique embedding (\cref{lm:pegasus-qubits})} & {\footnotesize no reformulation\tnote{(a)}}\\
   \midrule
\multicolumn{5}{l}{\emph{\gls{MIS} (\cref{sec:UD-MIS}) over a graph \revII{with $N$ vertices (examples are given for $N=10$)}, ILP formulation \cref{eq:MIS-ILP}:}}\\[0.5ex]
\begin{minipage}{\linewidth}
	\begin{itemize}\itemsep0pt
		\item $N$ binary variables (\eg: 10)
		\item up to $N(N-1)/2$ inequality constraints (\eg: 45)
	\end{itemize}
\end{minipage} &
\begin{minipage}{\linewidth}
	$N$ binary variables, see Lemma~\ref{lm:UDMIS-qubits}.\par
	(\eg: $10$)
\end{minipage}  & \begin{minipage}{\linewidth}
	Arbitrary graph: $4N^2$\\ (\eg: $\sim 400$ qubits)\vspace{\baselineskip}\par
	Unit disk graph: $N$\par (\eg: $10$ qubits)
\end{minipage}
& $24\lceil\frac{N+10}{12}\rceil\lceil\frac{N-2}{12}\rceil$\par (\eg: $\sim 50$ qubits)
& $N$\par (\eg: 10 qubits) \\
   \midrule
   \multicolumn{5}{l}{\emph{\gls{MaxCut} (\cref{sec:MaxCut}) over a complete graph \revII{with $N$ vertices (examples are given for $N=10$)}, ILP formulation \cref{eq:MaxCut-ILP}:}}\\[0.5ex]
   \begin{minipage}{\linewidth}
     \begin{itemize}\itemsep0pt
       \item $N(N+1)/2$ binary variables (\eg: 55)
       \item $N(N-1)$ inequality constraints (\eg: 90)
     \end{itemize}
   \end{minipage} \vspace{0.25\baselineskip}&
   \begin{minipage}{\linewidth}
     $N$ binary variables, see Lemma~\ref{lm:MaxCut-QUBO}.\par
     (\eg: $10$)
   \end{minipage} & $4N^2$ \par (\eg: $\sim 400$ qubits)
                  & $24\lceil\frac{N+10}{12}\rceil\lceil\frac{N-2}{12}\rceil$\par (\eg: $\sim 50$ qubits)
                  & $N$\par (\eg: 10 qubits) \\[3ex]
   \midrule
\multicolumn{5}{l}{\emph{\gls{TSP} (\cref{sec:TSP}) over a complete graph \revII{with $N$ vertices (examples are given for $N=10$)}. Quadratic assignment formulation \cref{obj:TSP-QUBO}:}}\\[0.5ex]
\begin{minipage}{\linewidth}
	\begin{itemize}\itemsep0pt
		\item $(N-1)^{2}$ binary variables (\eg: 81)
		\item $2(N-1)$ equality constraints (\eg: 18)
	\end{itemize}
\end{minipage}\vspace{0.25\baselineskip} &
\begin{minipage}{\linewidth}
	$(N-1)^{2}$ binary variables, constraints as penalties, see Lemma~\ref{lm:TSP-qubits}.\par
	(\eg: $81$)
\end{minipage} & $4(N-1)^4$\par (\eg: $\sim \num{26000}$ qubits)
& $24\lceil\frac{(N-1)^2 + 10}{12}\rceil\lceil\frac{(N-1)^2-2}{12}\rceil$\par (\eg: $\sim \num{1350}$ qubits)
& $(N-1)^{2}$\par (\eg: 81 qubits) \\
   \midrule
   \multicolumn{2}{l}{\emph{For reference: number of qubits for largest \gls{QC} available}} & $\num{256}$ & $\num{5000}+$& $\num{1000}+$\tnote{(b)}\\
   \bottomrule
 \end{tabularx}
 \begin{tablenotes}
  \item[a] \gls{IBM} gate-based \gls{QC} requires device dependent
  \emph{transpilation} to ensure the necessary qubit connectivity, which does not require additional qubits.
  \item[b] \citet{castelvecchi2023}
 \end{tablenotes}
\end{threeparttable}
\end{table}
\end{landscape}
}  %<-- from \afterpage

\section[Numerical results]{Numerical experiments}\label{sec:use cases}

The primary goal of our numerical experiments is not to compare the performance of state-of-the-art classical and quantum optimizers, but rather to illustrate the ``out-of-the-box experience'' with various quantum optimization strategies, and to discuss the associated challenges and opportunities. Informally, this can be seen as a comparison of the user experience \emph{assuming default parameters} across the different approaches. We begin by describing the experimental setup in \cref{sec:setup}, followed by a presentation of the results in \cref{sec:results}. Since the number of considered instances is relatively small (particularly for \gls{IBM-OPT}), our findings should be considered illustrative.

\subsection{Experimental setup}\label{sec:setup}\zlabel{sec:setup}
We compare five optimization approaches across three problem classes, as summarized in \cref{tab:num-approaches}.

\paragraph{Problem instances.} %
We consider the following randomized procedures to generate instances for the three problem classes. More details can be found in \zcref{app:instances}.

\begin{itemize}
\item \textbf{\gls{UD-MIS}:} To generate hardware-compliant instances for \emph{Aquila}, we construct unit disk graphs effectively by sampling the nodes from a grid randomly. (Graph connectivity is then determined by unit disk radius.)
\item \textbf{\gls{MaxCut}:} Instances are generated using the Erd\H{o}s-R\'enyi random graph model~\citep{erdos1959} with randomly selected edges and uniformly random edge weights within specified bounds.
\item \textbf{\gls{TSP}:} We select instances from the TSPLIB~\citep{reinelt1991} and generate smaller \gls{TSP} instances from them by randomly sampling subsets of nodes while preserving edge weights. Each sampled subset forms a complete subgraph, retaining the original distance structure.
\end{itemize}

\paragraph{Optimization approaches.} We consider three types of optimization
approaches:

% NOTE: a very strange error here. (Dependents on the position within the page?..)
%\rhinline{clearpage as temporary fix for the strange error -- remove}
%\clearpage
%
\begin{itemize}
  \item \textbf{Quantum-powered:} The three quantum-powered approaches
  		discussed in \cref{sec:quantum workflows}.
  \item \textbf{Quantum-simulated:} In addition to the \gls{IBM-OPT} approach, we also consider 
  		a simulated variant running entirely on classical hardware.
        To that end, the
        \gls{IBM} \gls{QPU} is replaced by the cloud-based \emph{QASM Simulator} from \gls{IBM}~\citep{javadiabhari2024}.
        We configure the simulator in such a way that it performs a noise-free, idealized simulation of a 32-qubit quantum device.\footnote{The simulator does not model physical noise effects, but it does simulate \emph{shot noise} by performing \num{1000} shots per iteration, thereby introducing a certain level of (pseudo-)randomness.}
        We refer to this approach
        as \gls{IBM-SIM-OPT} in the following.
        \item \textbf{Baseline:} We use the commercial solver by
        % for MaxCut we used the quadratic formulation, so, it is not ILP
        \gls{Gurobi} \citep{gurobi} with default parameters.
        % This essentially repeated the information in Table 3
        The time limit is set to
        5~minutes per instance, or beyond that time until the optimality gap of
        5\% is reached, but no more than 20~minutes.\label{q:baseline-choice}
        % begin quote 10!
        \revII{Our choice of models was driven by the aim to take a small
        integer model that would seem as a natural first choice from an \gls{OR}
        perspective. We took this approach for \gls{UD-MIS} and \gls{TSP}, but not
        for the \gls{MaxCut} problem, because the preliminary experiments proved
        that for this problem the quadratic
        formulation could be solved faster by \gls{Gurobi} than the simple \gls{ILP}.         
        Further, note that we are choosing an exact solution approach as a baseline, because our goal is not to compare the state-of-the-art heuristics, but to capture some simple characteristic of instance ``hardness'' and highlight a group of instances that may warrant further investigation.
        We briefly revisit this issue further. See \zcref{app:baselines} for more details.}
        % end quote
\end{itemize}

\label{q:netw}%
% begin quote networking !
\revII{%
For the quantum-powered optimization approaches, we make use of the \glspl{QC} by using the cloud-based services of the respective hardware providers.
To that end, for each use of the \gls{QC}, the corresponding quantum computing task (comprising high-level device instructions) is sent to a service platform, where it is placed in a scheduler queue.
This scheduler is necessary because we rely on publicly available services, which may be accessed by multiple users concurrently.
As a result, significant delays may occur in completing a quantum computing task, depending on hardware utilization at the time of the request.
Especially for \gls{IBM-OPT}, these delays can become significant due to the hybrid nature of the approach, which involves multiple iterative quantum computing tasks.
Although the scheduler can account for hybrid algorithms by adjusting task prioritization, this is only partially effective.
No guaranteed response times were provided on any of the platforms during our experiments.
}
% end quote

\begin{table}[t]
 	\centering	
 	\caption{Each optimization approach is used to solve a selection of instances across problem classes.}\label{tab:num-approaches}
 	\begin{tabular}{ccc}
 		\toprule
 		Approach          & Type              & Considered problem classes and formulations \\
 		\midrule
 		\gls{QuEra-OPT}   & quantum-powered   & \acrshort{UD-MIS} (analog / native) \\
 		\gls{D-Wave-OPT}  & quantum-powered   & \acrshort{UD-MIS} \eqref{eq:UDMIS-QUBO}, \acrshort{MaxCut} \eqref{eq:MaxCut-QUBO}, \acrshort{TSP} \eqref{eq:TSP-QUBO} \\
 		\gls{IBM-OPT}     & quantum-powered   & \acrshort{UD-MIS} \eqref{eq:UDMIS-QUBO}, \acrshort{MaxCut} \eqref{eq:MaxCut-QUBO}, \acrshort{TSP} \eqref{eq:TSP-QUBO} \\
 		\gls{IBM-SIM-OPT} & quantum-simulated & \acrshort{UD-MIS} \eqref{eq:UDMIS-QUBO}, \acrshort{MaxCut} \eqref{eq:MaxCut-QUBO}, \acrshort{TSP} \eqref{eq:TSP-QUBO} \\
		\gls{Gurobi}     & classical         & \acrshort{UD-MIS} \eqref{eq:MIS-ILP}, \acrshort{MaxCut} \eqref{eq:MaxCut-QUBO}, \acrshort{TSP} \eqref{obj:MTZ} \\
 		\bottomrule
 	\end{tabular}
\end{table}

\subsection[Solutions]{Results}\label{sec:results}\zlabel{sec:results}
The evaluation of our numerical experiments is focused on two key
characteristics for the end user: end-to-end algorithm runtimes and the
resulting solution quality. In the following, we refer to each attempt to solve
a single instance with a specific approach as a \emph{run}. Note that for
\gls{QuEra-OPT} and \gls{D-Wave-OPT}, a run implies aggregation over a sequence of
shots obtained from the respective \glspl{QC}, which yields a single
objective value (using the best value across all samples). 
Furthermore, a run for \gls{IBM-OPT} implies a sequence of quantum-classical iterations, where each quantum iteration involves a sequence of shots, which are used for intermediate calculations and the resulting objective value (again, using the best value across all final samples).
The runtime includes both the
necessary configuration steps performed on a classical computer and the actual
\gls{QC} runtime (total of all shots), but does not include the classical
procedure of the problem reformulation to \gls{UD-MIS} (for \gls{QuEra-OPT}) or to
\gls{QUBO} (for \gls{D-Wave-OPT} and \gls{IBM-OPT}).

For the evaluation presented here, we consider a set of runs where each solution approach attempts to solve each problem instance across all three problem classes exactly once (see also \zcref{app:runs summary}).

\paragraph{Scope.}
The scope of the numerical experiments is visualized in \cref{fig:inst sizes}.
The height of each bar indicates the number of instances of a certain size (horizontal axis) of the respective problem class (in columns) that is to be solved with a specific optimization approach (in rows). In other words, each counted instance represents a run.
\label{q:scope}
% begin quote scope1 !
The three colors indicate the \revIII{nature} of
the optimization result for each run:
\revII{%
A feasible, though not necessarily optimal, solution
(``success''), an infeasible solution (``infeasible''), or no result at all due to a technical reason (``fail''),
for example because of an early termination due to a timeout or a network error.
In contrast, some runs were not started at all, because the waiting time in the queue exceeded the time frame planned for this study. These non-started runs are not further considered and do not contribute to \cref{fig:inst sizes}.
}%
% end quote
Selected example for solutions from the quantum-powered approaches are
presented in \zcref{app:sample-outputs}.

\label{q:14}
% begin quote 14!
\revIII{%
	Our original aim was to consider a similar number of instances per problem size across approaches and problem classes. However, the results obtained are not fully balanced in this regard for various reasons.
	First, we did not attempt to solve \gls{MaxCut} and \gls{TSP} instances with the \gls{QuEra-OPT} approach due to the infeasible scaling of their \gls{UD-MIS} reformulations.
	Second, \gls{IBM-SIM-OPT} supports only instances with up to \num{32} binary variables.
	Finally, some minor irregularities in the number of solved instances are attributable to technical and organizational constraints, such as varying queue times and limited computational resources available in the scope of this study. This has primarily an effect on \gls{IBM-OPT} because of the hybrid quantum-classical algorithm, which requires a reoccurring availability of the quantum hardware over sufficiently many iterations, an elaborate overall task. For some instances, the waiting time in the queue exceeded the time frame planned for this study.
}%
% end quote

We find that the success rates vary significantly between the different approaches.
For the \gls{QuEra-OPT} approach, we only consider \gls{UD-MIS} instances, but achieve a total success rate of $100\%$.
For the \gls{D-Wave-OPT} approach, the total success rate is $75\%$. Smaller instances have a higher success rate than larger instances, where minor embeddings become a significant bottleneck. We will revisit this topic further below.
For the \gls{IBM-OPT} approach, the total success rate is only $21\%$.\footnote{These results do not necessarily characterize the theoretical performance of variational algorithms
	on \gls{IBM} hardware. However, we believe it is a good illustration of the
	point that \gls{QAOA}-based approaches require significantly more work beyond a
	naive implementation in order to achieve meaningful results.}
Among the key
practical reasons for so few instances solved are network connectivity issues
and significantly longer queuing times in comparison with the other approaches.
Most of the instances solved using the \gls{IBM-OPT} approach belong to the \gls{UD-MIS} problem class, which is known to have sparser \gls{QUBO} matrices, see \zcref{app:instances}.
Finally, for the \gls{IBM-SIM-OPT} approach, the total success rate is $65\%$.
The reason why this success rate is higher than for \gls{IBM-OPT}
can be partially attributed to convergence issues of the classical
optimization loop in presence of noise (the simulator is chosen to be noise-free).

% (See post_processing/inst_summary.R for details.)
\begin{figure}[ht]
	\centering
	\includegraphics[width=0.8\textwidth]{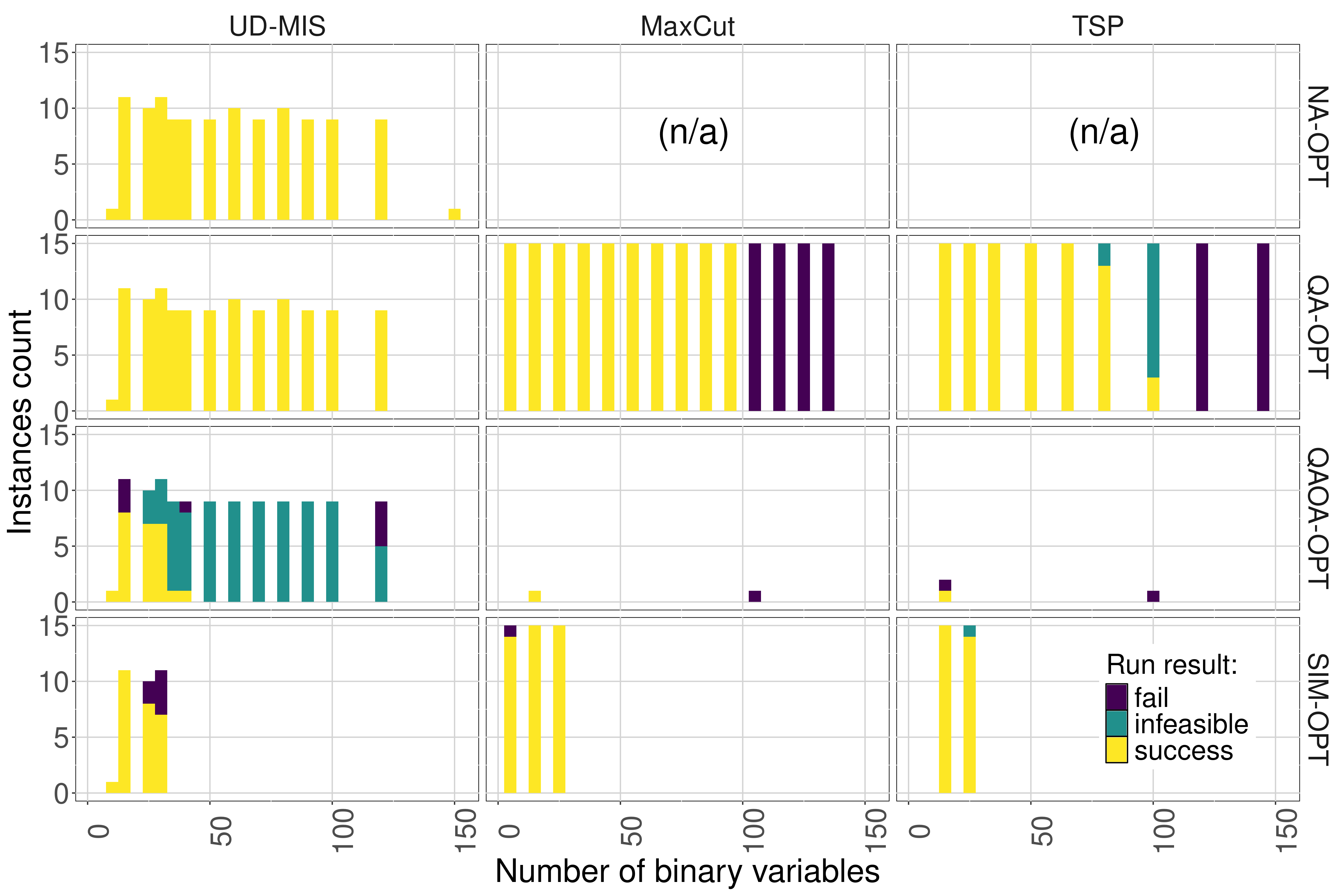}
	\caption{Number of instances across different problem sizes and classes,
		which constitute the runs of our numerical experiments.
      % begin quote 11!
      \revII{The colors classify the outcomes of each run: ``success'' (a feasible, but not
        necessarily optimal solution was found), ``infeasible'' (an infeasible solution was found), or ``fail'' (no solution was obtained due to technical issues).
        Runs that did not start at all are not counted here, leading to empty spots in the plots.}
      % end quote
    % begin quote 15!
	\revIII{The distribution of runs is unbalanced for various reasons (also leading to blank spots): insufficient qubits to realize \gls{UD-MIS} reformulation on \gls{QuEra-OPT} or certain \gls{QUBO} instances on \gls{IBM-SIM-OPT}, as well as limited computational resources and varying queue times.}
	% end quote
    }\label{fig:inst sizes}\zlabel{fig:inst sizes}
\end{figure}

\paragraph{Baseline results.}
\gls{Gurobi}, as the classical baseline, found a feasible solution for all instances and is therefore not included in \cref{fig:inst sizes}.
Specifically, all
\gls{TSP} instances were all solved to optimality with the largest runtime of about 15~minutes.
Furthermore, the \gls{UD-MIS} instances were all solved to optimality in under 1~second. Larger \gls{MaxCut} instances were more time-consuming to solve. In
total, 84\% of \gls{MaxCut} instances were solved with optimality gap of at most
5\% (more than 99\% with a gap of at most 10\%). Note that there are
specialized classical methods that might yield better runtimes (\eg,
\citealp{rehfeldt2023} for \gls{MaxCut}).\label{q:heuristic}
% begin quote 12 !
\revII{Moreover, just imposing a time
  limit on the classical solver already yields a heuristic that takes only a
  fraction of time of the exact method, while ensuring comparable quality of
  solutions: For our selection of instances, \gls{Gurobi} was very successful in finding good
  solutions early on, and spent significant amount of time on improving the
  bounds. (See also \zcref{app:baselines}.)}
% end quote

\paragraph{Performance metrics.} 
To evaluate the performance of each run, we consider the \emph{relative objective value deviation from the baseline} $R_{f}$ and the \emph{relative runtime deviation from the baseline} $R_{t}$,
\begin{equation} \label{eqn:rel-obj}\zlabel{eqn:rel-obj}
	R_{f} \defeq \frac{f_{q} - f_{c}}{f_{c}} \quad\text{and}\quad R_{t} \defeq \frac{t_{q} - t_{c}}{t_{c}},
\end{equation}
respectively. Here, $f_{q} \in \mathbb{R}$ denotes the resulting objective value from the run of interest and $t_{q} \in \mathbb{R}_{>0}$ its runtime. Analogously, $f_{c} \in \mathbb{R}$ denotes the best objective value found by \gls{Gurobi} for the same instance (which is not necessarily the optimal solution for some \gls{MaxCut} instances), and $t_{c} \in \mathbb{R}_{>0}$ denotes the corresponding runtime. We presume here that $f_{c} > 0$ for all considered instances.

\paragraph{Runtimes.} 
A runtime comparison of the considered approaches to \gls{Gurobi} is presented in \cref{fig:runtime-vs-classic-a}. Each point represents a single run, and the
tilted line in the top left corner reflects the situation where the compared runtimes are
equal. 
For the comparably small problem instances of interest, the runtimes for
quantum-powered approaches \label{q:if-feas}\revIII{(if feasible)} were
mostly longer than the classical baseline (points below the tilted line).
However, there is a small set of problem instances for which the runtime of
\gls{Gurobi} is longer than for the \gls{D-Wave-OPT} approach (points above the tilted line). 
For this subset of ``hard'' \gls{MaxCut} instances, \gls{Gurobi} required more
than 20~seconds for a solution.
The ``hard'' instances are all larger \gls{MaxCut} instances, and investigating their structure
in the context of \gls{OR} applications
might constitute a promising direction for further quantum optimization research.
A detailed view of the ``hard'' instances is shown in \cref{fig:runtime-vs-classic-b}.

\begin{figure}[ht]
	\centering
	\begin{subfigure}[t]{.48\textwidth}
		\centering
		\includegraphics[width=\linewidth]{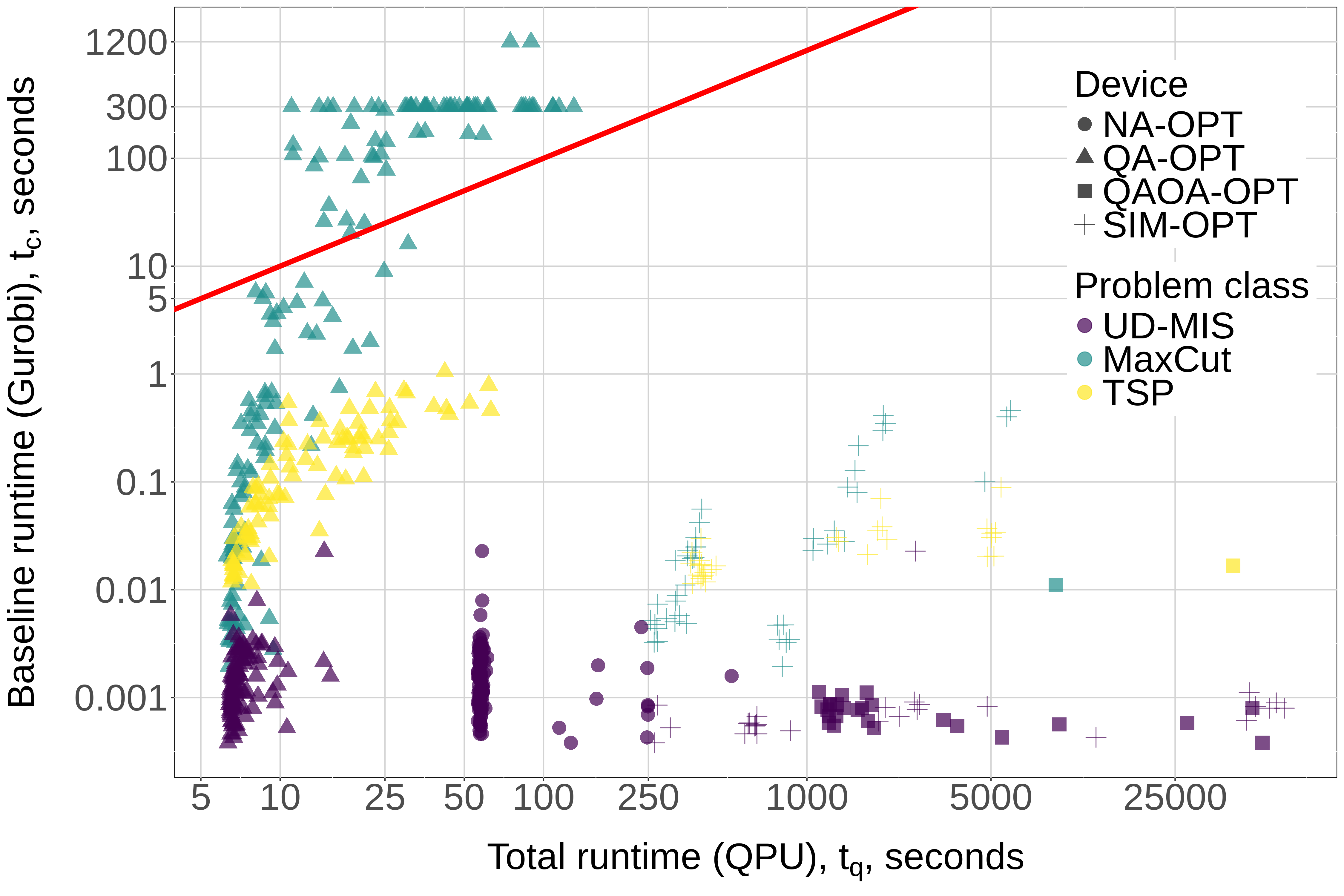}
		\caption{Runtimes for all \revIII{feasible} runs.} \label{fig:runtime-vs-classic-a}
	\end{subfigure}%
	\hspace{5em}
	\begin{subfigure}[t]{.32\textwidth}
		\centering
		\includegraphics[width=\linewidth]{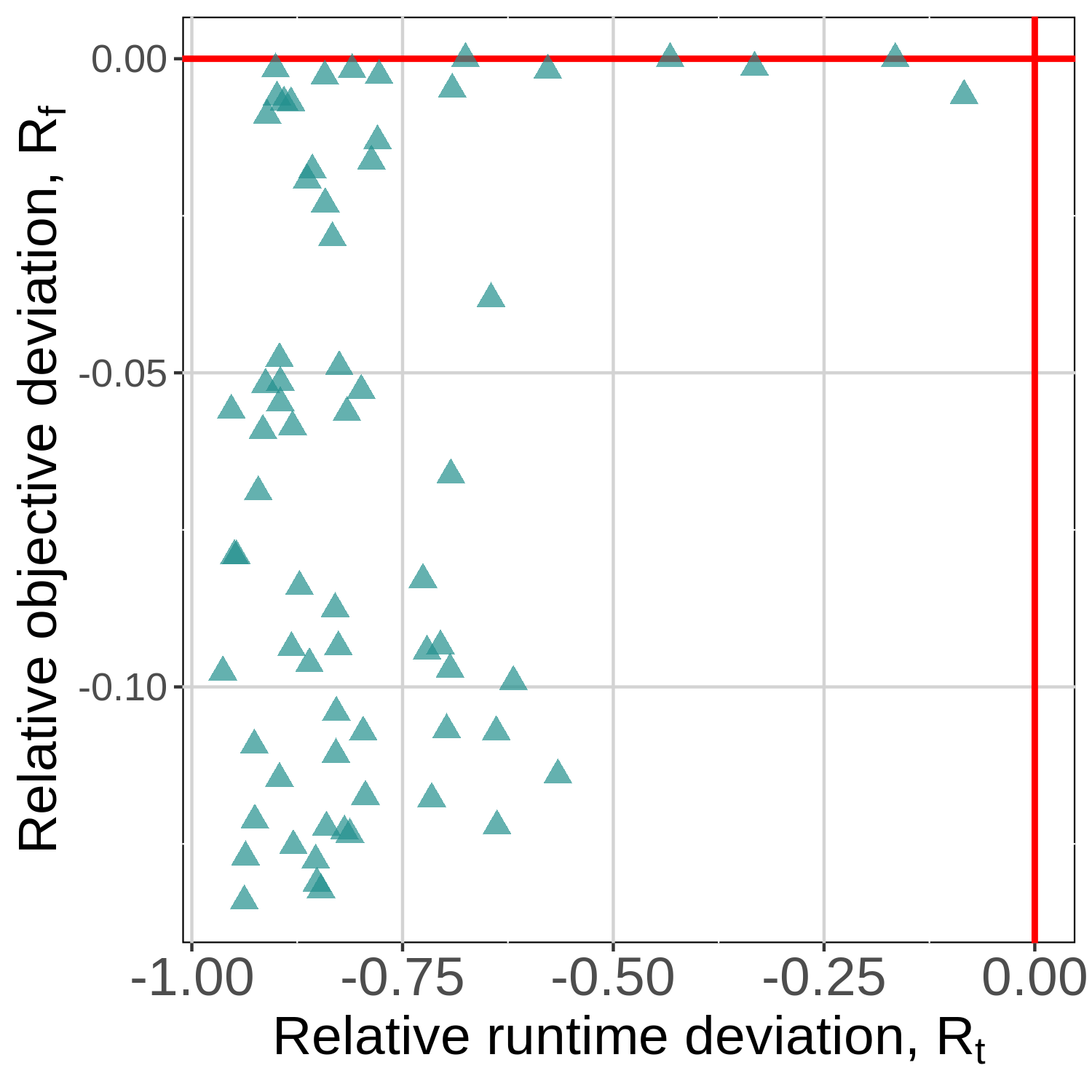}
		\caption{Runtime and objective value deviations, ``hard'' \gls{MaxCut} instances.} \label{fig:runtime-vs-classic-b}
	\end{subfigure}
	\caption{Performance evaluation of the quantum-powered approaches with \gls{Gurobi} as a baseline. We observe in (a) that there are ``hard'' \gls{MaxCut} instances for which the runtime of \gls{Gurobi} is longer than the \gls{D-Wave-OPT} approach and takes more than 20~seconds to solve. These instances are shown in (b).} \label{fig:runtime-vs-classic}
\end{figure}

\paragraph{Objective values.}
A comparison of the objective values from the quantum-powered approaches with the objective values from \gls{Gurobi} is presented in \cref{tab:obj-summary}.
No quantum-powered run has a resulting objective value better than the corresponding \gls{Gurobi} result.
We observe that for both \gls{MaxCut} and \gls{TSP}, the quantum-powered approaches found
many solutions within 25\% of the \gls{Gurobi} solutions.
For the
subset of ``hard'' \gls{MaxCut} instances, \cref{fig:runtime-vs-classic-b}, where \gls{Gurobi} required more
than 20~seconds for a solution, the \gls{D-Wave-OPT} approach was able to find some solutions within 10--15\% from the baseline objective in only half of the respective runtime.\footnote{\label{fn:hard}%
	% begin quote fn:hard
	\revII{The set of instances where \gls{Gurobi} took more than 20 seconds coincided with the set of instances where \gls{D-Wave-OPT} took no more time than \gls{Gurobi}. We refer to such instances as ``hard.''}
	% end quote
}

\label{q:UD-MIS-obj}
% begin quote 13!
\revII{While the proportion of instances that were solved by \gls{D-Wave-OPT} no worse than \gls{Gurobi} is
  comparable across problem classes, the share of instances with a gap within 5\%
  or 10\% for the \gls{UD-MIS} problem is noticeably lower than that for \gls{MaxCut}.
  This might be surprising at first glance, given that the \gls{QUBO}
  matrices for \gls{UD-MIS} instances are significantly less dense. Understanding this effect
  might constitute a possible direction for further research. One possible reason could be the different structure of constraints and their effects of the penalty terms. In particular,
  relatively small deviations of the solution in terms of the Hamming distances
  (number of bitflips) might lead to large relative changes in the objective
  value. Analysis of the solution landscapes in terms of Hamming distances and
  effects of coefficients normalization might therefore constitute reasonable first steps
  towards this direction.}

\revII{Further, each independent set constraint corresponding to a single edge,
  when formulated as a penalty term, is still relatively local and does not
  require many interaction terms in the objective. Therefore, the embedding
  process in the context of \gls{D-Wave-OPT} approach turns out less
  problematic. Formulating and falsifying such hypotheses might constitute a
  viable line of further research. A more careful approach to benchmarking
  involving a quantum annealer is discussed by \cite{lubinski2024}.}

\revII{For the \gls{QuEra-OPT} approach, the structure of constraints might also
  play a role for the performance. As discussed in \cref{sec:rydberg}, the
  implementation of \gls{UD-MIS} instances on neutral-atom devices requires not
  only a hardware-compliant problem graph, but also a suitable geometric
  arrangement of the atoms. These geometric considerations, which directly
  affect the realizability and accuracy of the problem encoding, remain an
  important practical aspect of the \gls{QuEra-OPT} approach~\citep{wurtz2023}.
  The atom placement in our experiments was directly driven by the instance
  generation process. A more sophisticated exploration of alternative placement
  strategies encoding the same \gls{UD-MIS} instance may serve as a strategy to
  further enhance performance. While such a study lies beyond our intention to
  provide an ``out-of-the-box experience,'' it could serve as a starting point
  for future research.}

\revII{In general, studying which characteristics of problem instances make them more or less suitable for quantum-powered optimization constitutes a possible further research direction \citep{koch2025quantumoptimizationbenchmarklibrary}, and the set of problem instances presented here could serve as a starting point for such an analysis. }%
% end quote

% Below is edited table from comps/figures/qpu_vs_cpu.tex
% (See post_processing/QPU_vs_classic_figures.R for details.)
\begin{table}[!tbp]
  \caption{\label{tab:obj-summary}%
    % begin quote tabrunlogs!
    \revIII{Comparison of the objective values from the quantum-powered approaches $f_q$ with the values from \gls{Gurobi} $f_c$, as per \cref{eqn:rel-obj}. ``Total'' is the total number of feasible runs for the given problem class and solution approach, ``within $P$\%'' shows the number of feasible runs with $|R_f|\leq P/100$ and its proportion of the ``Total.''\textsuperscript{(a)} The last column describes the number of feasible runs no worse than the \gls{Gurobi} objective, \ie, with $R_{f} \geq 0$ for \gls{MaxCut} and \gls{UD-MIS}, and $R_{f} \leq 0$ for \gls{TSP}.}%
    % end quote
  }
  \centering
  \begin{tabular}{ll|c|cc|cc|cc|cc}
    \toprule
    \multicolumn{1}{l}{Problem}&\multirow{2}{*}{Approach}&\multicolumn{9}{c}{Number\ of instance runs with relative objective deviations}\\
    \multicolumn{1}{l}{class} && Total&\multicolumn{2}{c}{within $25\%$}&\multicolumn{2}{c}{within $10\%$}&\multicolumn{2}{c}{within $5\%$} & \multicolumn{2}{c}{no worse}\\
    \midrule
UD-MIS & \gls{QuEra-OPT} & $117$ & $ 93$ & $ 79.5\%$ & $ 48$ & $ 41.0\%$ & $ 34$ & $ 29.1\%$ & $30$ & $25.6\%$\tabularnewline
       & \gls{D-Wave-OPT} & $116$ & $116$ & $  100.0\%$ & $ 77$ & $ 66.4\%$ & $ 54$ & $ 46.6\%$ & $46$ & $39.7\%$\tabularnewline
       & \gls{IBM-OPT}   & $ 25$ & $  9$ & $ 36.0\%$ & $  2$ & $  8.0\%$ & $  2$ & $  8.0\%$ & $ 2$ & $ 8.0\%$\tabularnewline
       & \gls{IBM-SIM-OPT}   & $ 27$ & $  8$ & $ 29.6\%$ & $  3$ & $ 11.1\%$ & $  3$ & $ 11.1\%$ & $ 3$ & $11.1\%$\tabularnewline
                                                              \midrule
MaxCut & \gls{D-Wave-OPT} & $149$ & $149$ & $100.0\%$ & $128$ & $ 85.9\%$ & $105$ & $ 70.5\%$ & $50$ & $33.6\%$\tabularnewline
       & \gls{IBM-OPT}   & $  1$ & $  1$ & -- & $  0$ &   --   & $  0$ &   --   & $ 0$ & --\tabularnewline
       & \gls{IBM-SIM-OPT}   & $ 44$ & $ 44$ & $100.0\%$ & $ 44$ & $100.0\%$   & $ 44$ & $100.0\%$   & $20$ & $45.5\%$\tabularnewline
                                                                \midrule
TSP & \gls{D-Wave-OPT} & $ 91$ & $ 61$ & $ 67.0\%$ & $ 49$ & $ 53.8\%$ & $ 46$ & $ 50.5\%$ & $37$ & $40.7\%$\tabularnewline
    & \gls{IBM-OPT}   & $  1$ & $  1$ & --   & $  1$ & --   & $  0$ & --   & $ 0$ & --\tabularnewline
    & \gls{IBM-SIM-OPT}   & $ 29$ & $ 26$ & $ 89.7\%$ & $ 22$ & $ 75.9\%$ & $ 20$ & $ 69.0\%$ & $17$ & $58.6\%$\tabularnewline
    \bottomrule
  \end{tabular}\vspace{0.25\baselineskip}

  % begin quote tabrunlogs-fn !
  \revIII{\footnotesize (a): For example, with the approach \gls{QuEra-OPT}, we obtain feasible solutions for 117 \gls{UD-MIS} instances. Out of these 117 runs, 93 (\ie, 79.5\% of the total of 117) achieve a 25\% deviation from the \gls{Gurobi} objective, and so on.}
  % end quote
\end{table}

\paragraph{Minor embeddings.}\label{q:minor-emb}
% begin quote 16!
To conclude, we report two key aspects of the necessary embedding process within
the \gls{D-Wave-OPT} approach. First, the runtime of a classical computer to
find the embedding and, second, the number of physical qubits necessary to
realize it. As explained in \cref{sec:QA}, the number of required physical
qubits can exceed the number of variables of the underlying \gls{QUBO} due to
the introduction of chains. \revIII{See \zcref{app:Pegasus} for an illustration of a toy
\gls{MaxCut} instance embedding.}

We found in our experiments that the bounds on the necessary numbers of qubits
presented in \cref{lm:pegasus-qubits} and \cref{tab:qubits-embeddings} were not
always tight, \ie, the heuristic algorithm used to find
embeddings~\citep{cai2014} performed better than the upper bound given
in~\cref{lm:pegasus-qubits}, \revIII{with the specific results depending on the
  qubit connectivity required by the problem instances at hand. We quantified
  this effect by assessing the observed relation between the numbers of logical
  qubits $N$ (from the problem instance) and physical qubits $N_{e}$ (from the
  embedding) for the \gls{D-Wave-OPT} approach. Our statistical model, ordinary
  least squares regression, yields estimates for $N_{e} \in O(N^{\hat{\beta}})$,
  where $\hat{\beta}\approx 1.8$ for \gls{MaxCut} and \gls{TSP} instances, and
  $\hat{\beta}\approx 1.1$ for \gls{UD-MIS} instances in our dataset. (See
  \zcref{app:regression} for further details.) This agrees to our expectation
  that the share of nonzero entries in the \gls{QUBO} matrix scales differently
  for different problem classes (see \zcref{app:instances}), which affects the
  embedding overhead. }
% end quote

The embedding process also entails significant (classical) runtime costs. In our
case, it dominates the solution time for large enough \gls{TSP} and \gls{MaxCut}
instances, as illustrated in \cref{fig:embedding-time-bars}. For sparser
\gls{UD-MIS} instances with comparable numbers of binary variables, finding
embeddings was relatively easy. While the embedding time varies with the number
of binary variables in general, we did not scale the \revII{allotted}\label{q:allotted} annealing time.
Therefore, for large and dense instances, up to 90\% of the total runtime was
spent trying to find an embedding, before starting the actual computation on the
\gls{QC}.

\begin{figure}[ht]
	\centering
	\begin{minipage}[t]{0.3\linewidth}
		\centering \gls{UD-MIS} \\[1ex]
		\includegraphics[width=\textwidth]{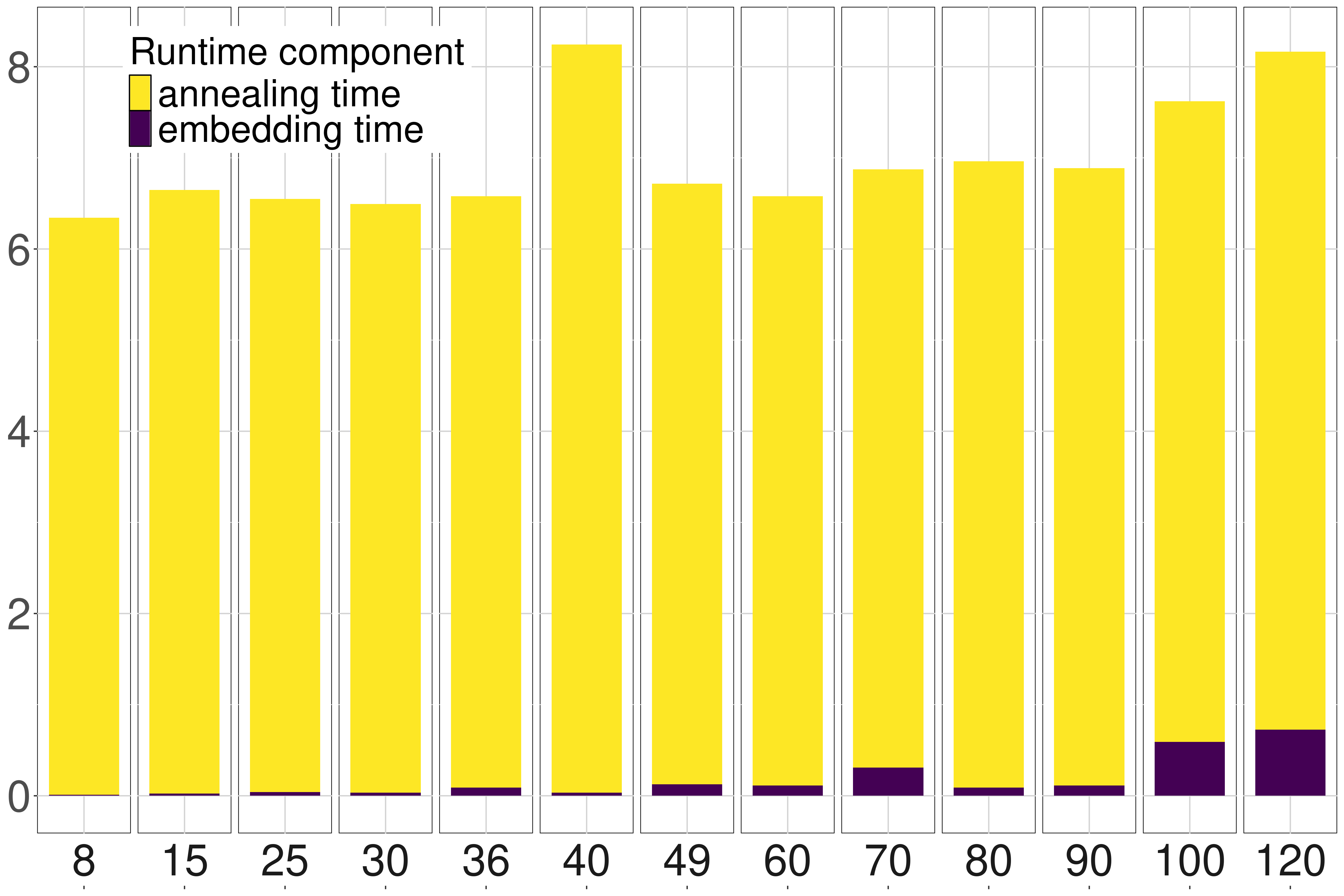}
	\end{minipage}\hfill%
	\begin{minipage}[t]{0.3\linewidth}
		\centering \gls{MaxCut} \\[1ex]
		\includegraphics[width=\textwidth]{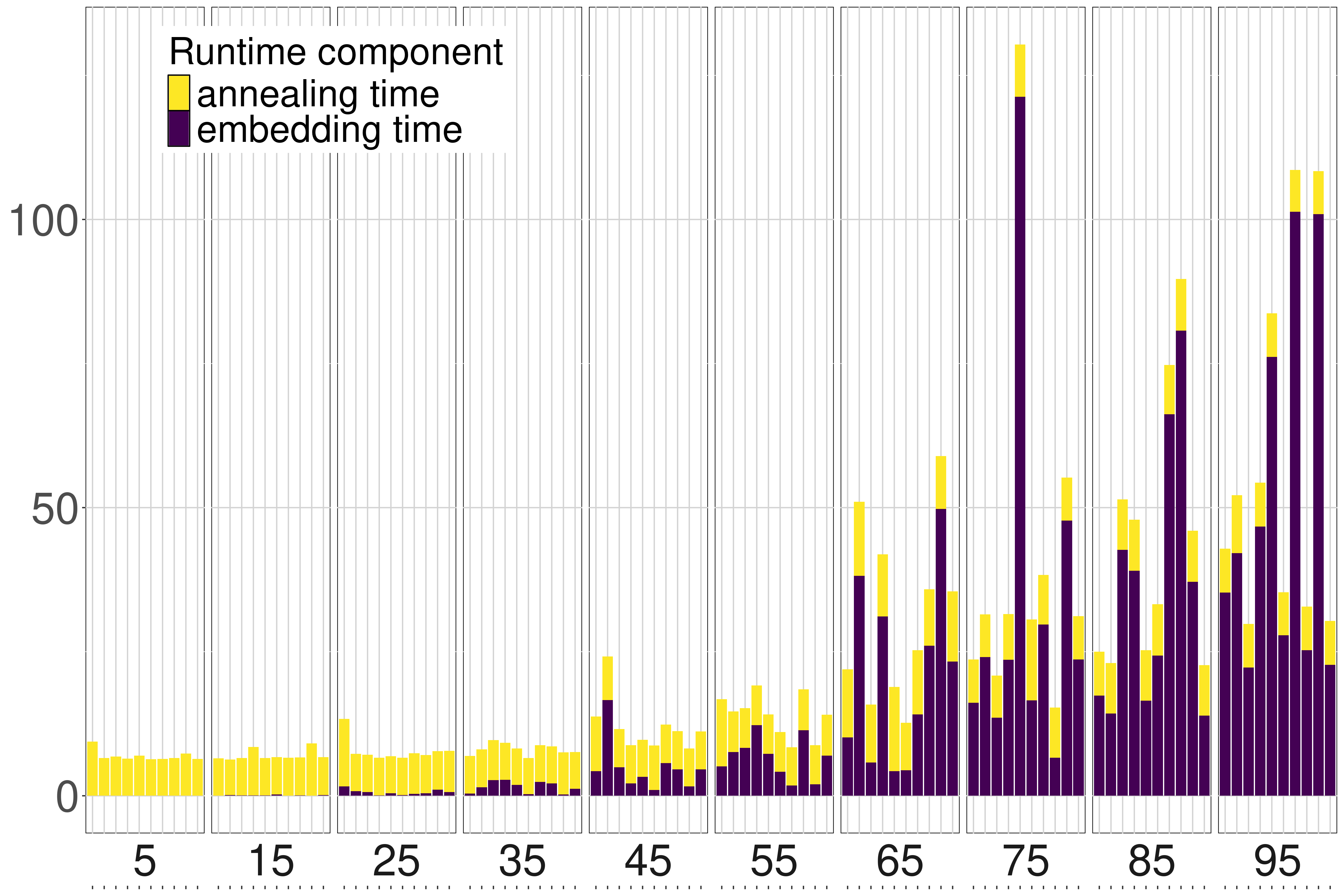}
	\end{minipage}\hfill%
	\begin{minipage}[t]{0.3\linewidth}
		\centering \gls{TSP} \\[1ex]
		\includegraphics[width=\textwidth]{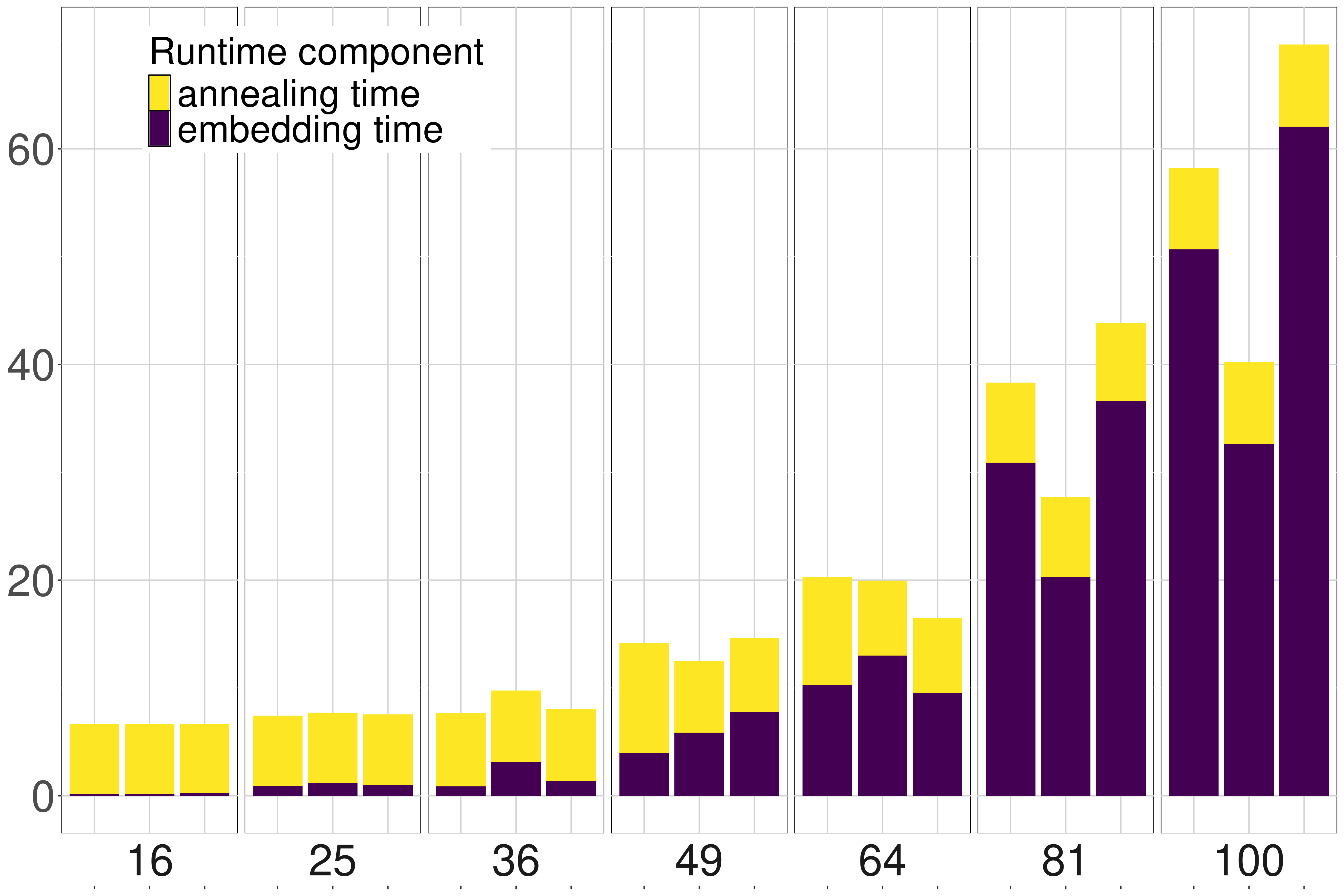}
	\end{minipage}
	\caption{Runtime histograms for the \gls{D-Wave-OPT} approach, per problem class. Each bar represents a
		randomly selected instance with a number of variables as indicated on
		the horizontal axis. Bar height represents the total runtime in seconds, as indicated on the vertical axis
		(different scale for each problem class). The total runtime constitutes the computation of a minor embedding
      (classical computer; dark color) and the \gls{QA} computation (\gls{QC}; light
      color).
      } \label{fig:embedding-time-bars}
\end{figure}

\section{Conclusion}\label{sec:conclusion}

It is difficult to assess which of the currently available quantum computing architectures is superior for \gls{OR} applications. Different quantum algorithms may have different hardware requirements, and device characteristics may affect the details of algorithm implementations. As a result, the development of hardware and the development of algorithms for practical applications go together. For this reason, rather than attempting to benchmark a broad collection of methods, we highlight three specific quantum-powered approaches to solve discrete optimization problems. Specifically, we consider two analog approaches,
\gls{QuEra-OPT} (solving \glspl{UD-MIS} with an analog \gls{QC} based on neutral
atom technology) and \gls{D-Wave-OPT} (solving \glspl{QUBO} with a quantum
annealer based on superconducting technology), as well as a digital approach,
\gls{IBM-OPT} (solving \glspl{QUBO} with a gate-based \gls{QC} based on
superconducting technology). The paper offers three groups of key contributions.

First, we provide a high-level overview of the relevant workflows, aiming at a
general \gls{OR} audience. To the best of our knowledge, this is the first
attempt to describe several different quantum approaches in a uniform framework as sketched in \cref{fig:workflows}. 
We believe this effort will help motivate further
involvement of the \gls{OR} community into quantum optimization research. Fine-tuning of quantum algorithms and careful modeling of optimization problems may allow alternative quantum-powered solution strategies and therefore pose novel research opportunities~\citep{blekos2024}.

Second, we show that the required number of qubits for an optimization problem
is not just a characteristic of the problem itself, but can be significantly
affected by the chosen quantum technology, problem type, and the specific
formulation, as listed in \cref{tab:formulations,tab:qubits-embeddings}.
\label{q:concl}
% begin quote 7!
\revI{As discussed in \cref{sec:qubits}, presenting new problem encodings is beyond the scope if this work, but we would like to emphasize the relevance of this issue in the context of quantum computing.}
% end quote
% % begin quote
% \revI{We emphasize that we do not present novel ways of modeling the problems,
%   nor do we provide a comprehensive overview of possible quantum-aware modeling
%   approaches for optimization problems. In fact, finding suitable formulations
%   of classical optimization problems that allow an effective treatment with
%   quantum optimization is a task known as \emph{problem encoding} and constitutes a
%   research direction of its own. For more general discussions of different
%   formulations for a wide range of optimization problems in the context of
%   quantum computing and \glspl{QUBO}, see for example
%   \cite{lucas2014,dominguez2023,glover2022}. Furthermore, \citet{ruan2020b},
%   \citet{gonzalez-bermejo2022,salehi2022}, and \cite{codognet2024} focus on \glspl{TSP},
%   while \citet{hadfield2021} considers more general classes of functions and aims to
%   provide a ``design toolkit of quantum optimization,'' and \citet{hadfield2017}
%   discuss encoding constraints without penalty terms, specifically in the
%   context of \gls{QAOA}. \cite{padmasola2025solvingtravelingsalesmanproblem} focuses on \gls{TSP},
%   but compares and contrasts a few other quantum technologies. This wide range of references demonstrates the
%   diversity of possible problem encoding approaches.}
% % end quote
% @Raoul: I've moved this passage to Sec. 4, as you suggested.

Third, we highlight the emergence of auxiliary optimization problems that arise in the course of applied quantum optimization in order to properly configure the respective quantum devices. Typically, such auxiliary problems are solved on a classical computer and could be a source for interdisciplinary collaborations between the quantum computing and \gls{OR} communities. 
As a prototypical example, we discuss the embedding process within the \gls{D-Wave-OPT} approach and analyze the associated runtime and required number of physical qubits. We did not further explore the atom encoding step in the \gls{QuEra-OPT} or the transpilation step in the \gls{IBM-OPT}, both of which represent distinct research directions. In general, improving heuristic methods for generating effective device configurations could enhance the practicality of all three quantum-powered approaches considered here.

Our numerical experiments have not shown a clear quantum advantage.
The quality of the obtained numerical solutions depends on the problem class, but
in general our classical baseline (solving an integer programming model using a
state of the art commercial solver) outperformed the quantum-powered approaches,
given enough time. However, even with all the restrictions of the available
quantum devices, we were able to generate a collection of \gls{MaxCut} instances
that were hard enough for the classical solver, and for which we were able to
obtain heuristic solutions from the quantum device with reasonable quality
faster than using a classical solver, \revIII{although it is not difficult to
  design classical heuristics that would outperform the quantum-powered
  approaches we discuss over our selection of instances.} Our experiments highlight the obvious
fact that the sparsity of the \gls{QUBO} matrix affects the problem complexity, both for
the classical baseline and the quantum-powered approaches.

\label{q:concl2}%
% begin quote concl !
\revIII{We have found that the quantum-powered optimization approaches discussed in this paper provide very different user experiences. While they can be described in a largely unified framework, each presents its own unique strengths and challenges. For example, the \gls{D-Wave-OPT} approach typically produced feasible solutions for instances up to a certain size, provided suitable embeddings could be identified. However, finding embeddings for highly connected problem instances can become very difficult. The \gls{QuEra-OPT} approach was even more robust in terms of results. On the other hand, it was more limited in terms of the number of qubits and required a representation of the problem as a \gls{UD-MIS}, which quickly leads to infeasible resource requirements. Finally, the \gls{IBM-OPT} approach, while the most flexible for algorithm development, required the most effort to establish a reliable computational workflow for technical reasons. These differences highlight the value of continued exploration across all quantum computing architectures. Rather than prematurely favoring one technology over another, it is important to recognize that each may be well suited to different types of \gls{OR} problems. Improving our understanding of their computational capabilities and identifying promising application domains remain important directions for future research.}
% end quote

\section{Acknowledgements}\label{sec:thanks}
The work was partially funded by the \emph{Research Initiative Quantum Computing
  for Artificial Intelligence (QC-AI)}. Access to quantum hardware was funded by
the \emph{Fraunhofer Quantum Now} program.

\bibliographystyle{elsarticle-harv-sven}
\bibliography{references}

\end{document}

\maketitle
\vspace{-2.5\baselineskip}\begin{center}\adforn{21} \end{center}

% \rhinline{I have not been able to check whether cross-references with zref are resolved correctly. It can be that some references might require to add some prefix like ``equation'' or ``figure'' for clarity/consistency.}

This document supplements the main text by providing additional technical
details. Note that to~avoid confusion, labels (\eg, for the figures or tables)
in these supplementary materials are prefixed with the letter ``S.'' For
example, ``Figure 5'' refers to a figure in the main text, while ``Figure
S.5'' points to a supplementary figure below.\medskip

The source code along with the data used to obtain the results presented in the
paper can be downloaded from:

\href{https://github.com/alex-bochkarev/qopt-overview}{https://github.com/alex-bochkarev/qopt-overview}\medskip

\noindent Corresponding technical documentation, including the discussion of
relevant data formats and auxiliary programs, is available via:

\href{https://alex-bochkarev.github.io/qopt-overview}{https://alex-bochkarev.github.io/qopt-overview}

\section{Instance generation}\zlabel{app:instances}\label{app:instances}

As presented in \cref{sec:setup}, we use randomized procedures to generate our collection of problem instances, which comprises three problem classes: \Gls{UD-MIS}, \Gls{MaxCut}, and \Gls{TSP}. For each class, we create instances of
varying sizes and problem structures with the following procedures:

\begin{itemize}
	\item \textbf{\Gls{UD-MIS}:} To ensure that all generated instances are
	hardware-compliant (for \emph{Aquila}), we consider the generation of
	unit disk graphs with nodes on a fixed grid. To that end, we use a fixed
	coordinate window $W \times H$, specifying the maximum integer values
	for each node coordinate on a two-dimensional plane. Then we create a
	grid of nodes with all possible integer coordinates $(x,y)$ for
	$0\leq x \leq W$ and $0\leq y \leq H$. Finally, the nodes are deleted
	one by one\footnote{The candidate for deletion being selected uniformly at random.} until the
	desired total number of nodes is reached. The unit disk graph is then
	specified by this set of nodes and a real parameter $R$ as the unit disk
	radius. Variations in the coordinate window size allow us to generate
	instances of different node density on the coordinate plane, leading to
	different graph connectivity. Namely, for each studied problem size~$N$,
	we fix the window to
	$W_{N} \times H_{N}\defeq \bigl(\lceil \sqrt{N} \rceil + 1\bigr) \times \bigl(\lceil \sqrt{N}\rceil + 1\bigr)$
	to generate a set of points. We use three different values of $R$ to
	create different connectivity patterns in the grid of points\footnote{Specifically, we used values of \num{1.25}, \num{1.42}, and \num{1.85} --- which are, respectively lower, higher, and comparable to the diagonal in the unit grid of $\sqrt{2}\approx\num{1.42}$. See the code documentation and specifically module \texttt{MIS\_inst.py} for further implementation details.},
	and repeat the procedure to generate
	three instances for each set of parameters. In addition to this
	procedure, we also build $9$ instances with hand-picked parameters for
	manual inspection: with a fixed value of $R=1.5$, three different window
	sizes, and varying number of nodes between $8$ and $150$. The provided
	unit disk radius $R$ for each instance is used as the blockade radius
	for \emph{Aquila} within the \gls{QuEra-OPT} approach.
	\item \textbf{\Gls{MaxCut}:} The collection of instances is created by generating a set of random graphs
	$G(N, p)$ of different sizes using the Erd\H{o}s-R\'enyi random graph model
	\citep{erdos1959}, \ie, given $N$ nodes, each pair of nodes is
	connected by an edge with the probability $p$. This procedure results in graphs of different connectivity depending on the parameter $p$.
	Across the studied problem sizes $N$, we generate graphs for $p \in \{0.25,0.5,0.75\}$. Weights are generated uniformly at random, within given limits.\footnote{For each instance, we fix the maximum weight $C_{\max}$ uniformly at random between \num{1} and \num{10}, and then pick edge weights uniformly at random between zero and $C_{\max}$. See code documentation and, specifically, module \texttt{MWC\_inst.py} for specific implementation.}
	\item \textbf{\Gls{TSP}:} We pick $15$ instances (uniformly) at random
	from TSPLIB~\citep{reinelt1991}, and sample the desired number of nodes from each one randomly,
	preserving edge weights.\footnote{The original instances represent complete graphs, which, however, might contain fewer or more nodes than we require. In the former case we discard the instance, in the latter --- pick a uniformly random subset of nodes of the necessary cardinality. See the code documentation and, specifically, module \texttt{TSP\_inst.py} for further implementation details.}
	Each sampling procedure results in a complete sub-graph of the original graph.
	Therefore, each selected TSPLIB instance yields a
	collection of \gls{TSP} instances constituting complete graphs of the given
	sizes, with the same distances structure as compared to the original instance.
\end{itemize}

A summary of the instance sizes is shown in \zcref{fig:inst sizes}. For \gls{UD-MIS} and \gls{MaxCut}, example instances are shown in \cref{fig:UDMISes,fig:MaxCuts}, respectively.

\begin{figure}[ht]
	\centering
	\includegraphics[width=.95\textwidth]{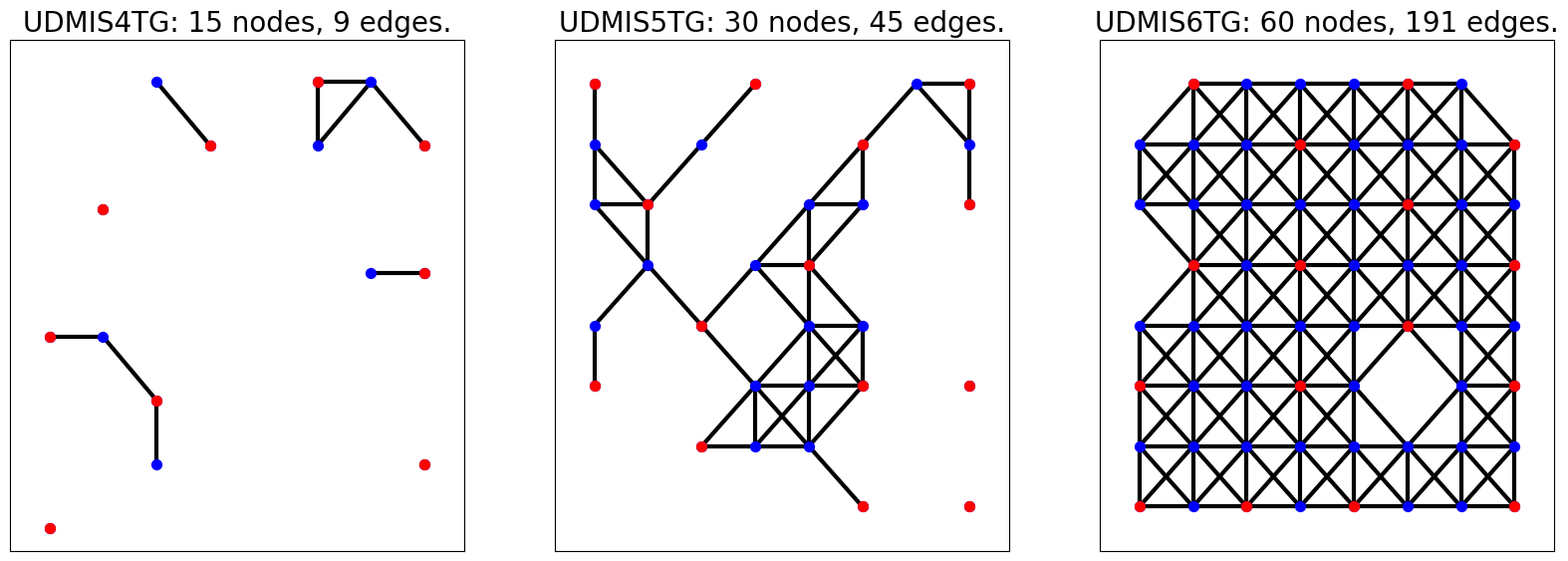}
	\caption{\label{fig:UDMISes} Exemplary \gls{UD-MIS} instances with the same unit disk radius
		and window size, but a varying number of nodes. According to the definition of a unit disk graph, nodes within the radius of a unit disk are connected by an edge.}
\end{figure}

\begin{figure}[ht]
	\centering
	\includegraphics[width=.95\textwidth]{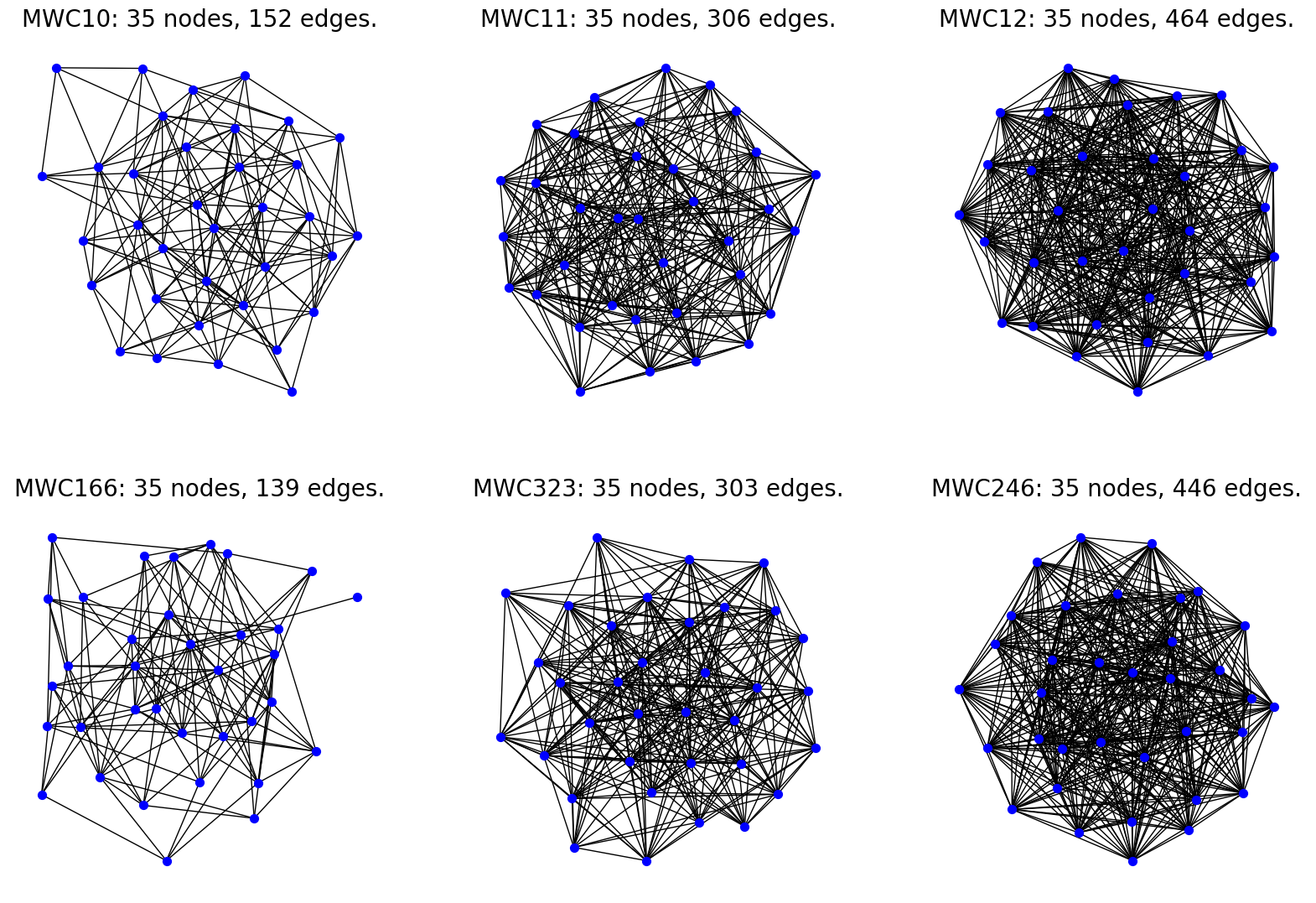}
	\caption{\label{fig:MaxCuts} Exemplary \gls{MaxCut} instances. The graphs in the
		left column correspond to $p=0.25$, in the middle column to $p=0.5$, and in
		the right column to $p=0.75$ (with the same number of nodes).}
\end{figure}

A randomized instance generation allows us to consider different problem structures with
respect to the resulting \gls{QUBO} matrices, as visualized in \cref{fig:share of nonzeros}.
In particular, the generation procedure for \gls{MaxCut} instances allows us to vary the density of
\gls{QUBO} matrix in our instance collection. The three values of the node connectivity
parameter $p$ correspond to
the three groups of \gls{MaxCut} instances (dark triangles) in the plot. For
large enough instances, these values approximately equal the shares of nonzero
entries in the matrix. Our \gls{UD-MIS} instances (dark circles) imply
significantly less dense \gls{QUBO} matrices, mostly due to the node degree
restriction from the hardware requirements. \gls{TSP} instances
(light squares) occupy the middle position in terms of the matrix sparsity.
In the following, we provide a brief theoretical illustration for the share of nonzero entries for each of the three problem classes of interest.

\begin{figure}[ht]
	\centering
	\includegraphics[width=0.7\textwidth]{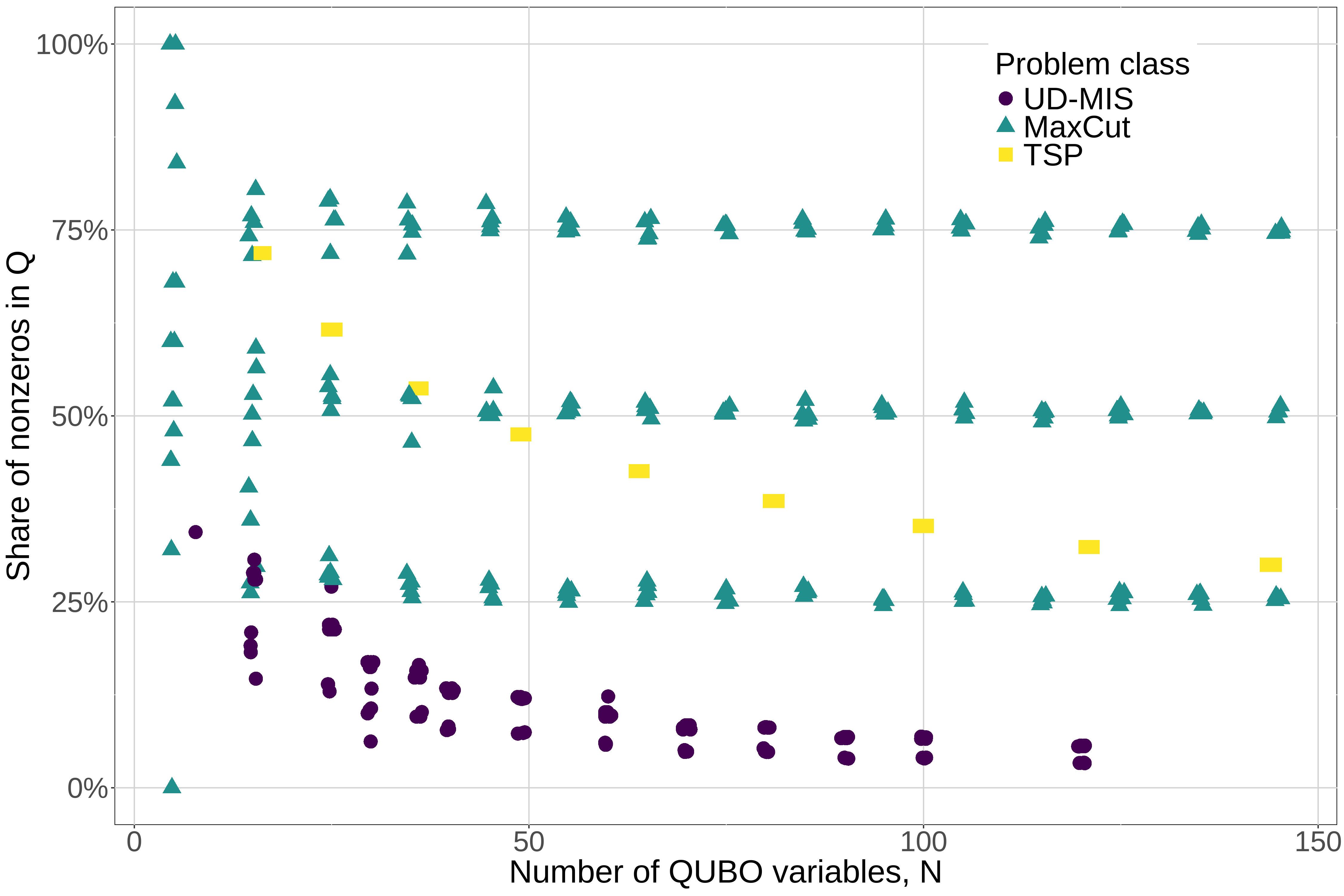}
	\caption{\zlabel{fig:share of nonzeros}\label{fig:share of nonzeros} Share of nonzero entries in the \gls{QUBO}
		matrix $Q$, formulation~\zref{eq:QUBO}, for different instance sizes and different classes of problems. The three bands of points for
		\gls{MaxCut} instances correspond to different parameter
		values $p=0.25$, $p=0.5$, and $p=0.75$ (bottom to top).}
\end{figure}

\paragraph{\texorpdfstring{\gls{UD-MIS}}{UD-MIS}.} The \gls{QUBO} matrix $Q$
defined by formulation~\eqref*{eq:UDMIS-QUBO} has $2|E|+|V|$ nonzero entries. If the node
degree is bounded from above by some value $d$ that is independent of $N$ (as it
is the case in our experiments), the total number of nonzero entries in $Q$ is
at most $2 N(d/2) + N = N(d+1)$. Since the total number of entries in $Q$ is
$N^{2}$, the share of nonzero entries is $N(d+1)/N^{2} = (d+1)/N$, decreasing as
the problem scales up.

\paragraph{\texorpdfstring{\gls{MaxCut}}{MaxCut}.} The number of nonzero entries
in matrix $Q$ given by~\eqref*{eq:MaxCut-QUBO} is $2|E| + |V|$. Note that
we deliberately generate instances with different node connectivity by varying
node connectivity parameter $p$ in Erd\H{o}s-R\'enyi model, taking values of $0.25$,
$0.5$, and $0.75$. If the number of edges is $q N(N-1)/2 \in O(N^{2})$ for some
real number $q$ between zero and one, the share of nonzero entries in $Q$ will
constitute:
\[ \frac{2qN(N-1)/2 + N}{N^{2}} = q + (1-q)/N \in O(1), \] as the problem scales
up with fixed $q$. Namely, the share of nonzero entries in $Q$ will approach the
parameter~$p$ used during the instance generation, as the number of nodes $N$
grows.

\paragraph{\texorpdfstring{\gls{TSP}}{TSP}.} The number of nonzero entries in
the matrix $Q$ defined in formulation \eqref*{eq:TSP-QUBO} is in $O(N^{3})$, out of
$(N-1)^{4}$ entries in $Q$ in total, and hence the share of nonzero entries
belongs to $O(1/N)$.

\revIIIcomment{The following section is new.}
\section{On the choice of penalty coefficients $M$ for \texorpdfstring{\gls{QUBO}}{QUBO} formulations}\zlabel{app:big-M}
% begin quote big-Ms !
Let us briefly revisit the choice of the penalty coefficients (denoted $M$) that
we used to formulate \gls{QUBO} instances in \zcref{sec:qubits}. Specifically,
for each of the problem classes under consideration, we create an equivalent \gls{QUBO}
formulation, at least in the sense that optimal solutions of the original problem and
the \gls{QUBO} coincide.

  \begin{itemize}
    \item \textbf{\gls{UD-MIS}} (Section~{\zref{sec:UD-MIS}}, page~\pageref*{eq:MIS-ILP}).
          Formulation (\ref*{eq:UDMIS-QUBO}) is equivalent to (\ref*{eq:MIS-ILP}) for
          $M = |V| + 1$, since in this case a single independent set constraint
          violation ($x_{i}=x_{j}=1$ for some $\{i,j\}\in E$) makes this
          infeasible solution worse than a trivial feasible solution of
          $x_{1}=x_{2}=\cdots=x_{|V|}=0$.
    \item \textbf{\gls{MaxCut}} (Section~{\ref*{sec:MaxCut}},
		  page~\pageref*{eq:MaxCut-QUBO}). We discuss a natural quadratic
		  formulation that does not involve penalty terms.
	\item \textbf{\gls{TSP}--\ref*{obj:DFJ}} (Section~{\ref*{sec:TSP}},
		  page~\pageref*{obj:DFJ}). For both families of constraints, it
		  suffices to set the penalty term to $M = \max c_{ij} + 1$. Assume for
		  a contradiction that there is an optimal solution that violates some
		  constraint. If there is a node of degree $\ge 3$ whose $k$ neighbors
		  all have degree~$1$, then replace this star by an arbitrary cycle
		  through the set of involved vertices. This adds at most $k-1$ new
		  edges and reduces the penalty of $k+1$ vertices. If there is a
		  node~$v$ of degree $\ge 3$ with a neighbor~$w$ of degree~$\ge 2$, then
		  removing the edge $(v,w)$ will reduce the penalty of $v$ and add the
		  penalty for $w$, while losing some positive cost $c_{vw}$. Assume now
		  that all nodes have degree $\le 2$. If there are at least two nodes of
		  degree smaller than $2$, we can add an edge between them, while
		  reducing the penalty by at least $2M$. If there is a single node~$v$
		  of degree $\le 1$, then this node must have degree~$0$ and all other
		  nodes have degree~$2$. Then replacing an arbitrary edge $(u,w)$ by the
		  edges $(u,v)$ and $(v,w)$ inserts two edges, and reduces the penalty
		  twice. When all degree constraints are satisfied, then the solution is
		  a union of cycles. For any two cycles, we can merge the two cycles by
		  replacing two edges by two potentially longer edges, while correcting
		  two violated subtour constraints, which reduces the penalty by at
		  least $2M$.
		\item \textbf{\gls{TSP}--\ref*{obj:MTZ}} (Section~\ref*{sec:TSP}, page~\pageref*{obj:MTZ}).
		  In the penalty terms corresponding to the in-degree and out-degree
		  constraints, we can still use the penalty $M = \max c_{ij} + 1$ by a
		  similar argument as for the \ref*{obj:DFJ} formulation. For the node
		  ordering constraints, we can use the penalty~$M = 2\max c_{ij} + 1$.
		  If there is a violation of an ordering constraint, this gives the
		  possibility to have a subtour. This can be repaired by replacing two
		  edges by two other edges. Hence, for the selected $M$, such a merge
		  reduces the objective value.

		\item \textbf{\gls{TSP}--\ref*{obj:TSP-QUBO}} (Section~\ref*{sec:TSP}, page~\pageref*{obj:TSP-QUBO}).
          One specific value to ensure the equivalence is
		  $M=\frac{N(N-1)}{2}\max c_{ij}+1$: In this case, violation of a single
		  constraint makes the objective value larger (worse) than any of the
		  feasible solutions.
      \end{itemize}
% end quote
\section{Required number of logical qubits for the \texorpdfstring{\gls{QUBO}}{QUBOs} formulation of \texorpdfstring{\gls{TSP}}{TSP}}\zlabel{app:TSP-logical-qubits}

We derive the result of \zcref{lem:no_qubits_DFJ}, stating the necessary
number of logical qubits for the \gls{QUBO} reformulation of the \gls{TSP}
integer program~\eqref*{obj:DFJ}. This formulation involves binary variables
only, but we still need to introduce penalty terms in order to transform it to
an equivalent \gls{QUBO}. The equality constraints can be directly translated into
quadratic penalty terms: For each $i \in V$, the constraint
$\sum_{j: j>i} x_{ij} + \sum_{j: j<i}x_{ji} = 2$ induces the term
$M (\sum_{j: j>i} x_{ij} + \sum_{j: j<i}x_{ji}-2)^2$ in the \gls{QUBO} objective
function, where $M$ must be sufficiently large (\eg,
$M = N^{2} \max_{\{i,j\}\in E} c_{ij}$) to separate feasible and infeasible
solutions in terms of the objective values. Inequality
constraints~\eqref*{DFJ:subtour} have to be converted to equalities by
introducing integer slack variables. So, we replace every inequality constraint
of the form
\[
\sum_{\{i,j\} \in E(S)} x_{ij} \leq |S|-1
\]
with the corresponding pair of constraints
\[
\sum_{\{i,j\} \in E(S)} x_{ij} + s_S = |S|-1, \quad s_S\geq 0.
\]
This transformation introduces new integer variables~$s_S$, and we further
substitute them with their binary representations. For every subset~$S$ with
$3 \le |S| \le \frac N 2$ the sum $\sum_{\{i,j\} \in E(S)} x_{ij}$ can take values
between $0$ (if no two nodes from $S$ are visited consecutively) and $|S|-1$ (if
all nodes from $Q$ are visited consecutively). Therefore, we need binary
variables $b_{S,r}$ for $r = 0,\dotsc,\lfloor \log_2(|S|-1) \rfloor$, so that we replace
$s_S$ by $\sum_{r=0}^{\lfloor \log (|S|-1) \rfloor} 2^r b_{S,r}$, which can take all values
between $0$ and $|S|-1$. After this replacement, we can again derive quadratic
penalty terms of the form
\[M \cdot \bigg(\sum_{\{i,j\} \in E(S)} x_{ij} + \sum_{r=0}^{\lfloor \log(|S|-1)\rfloor} 2^r b_{S,r} - (|S| - 1)\bigg)^2\]
for $M$ large enough.

Therefore, for a Dantzig-Fulkerson-Johnson model, we have a corresponding \gls{QUBO} with
$\frac{N(N-1)}{2} + O(2^N \log N)$ binary variables and
$2^{N-1}-\frac{N(N-1)}{2}-1$ penalty terms in the objective.
To calculate the number of slack variables more precisely, note that we have
$\binom{N}{3}$ constraints with slack at most two, $\binom{N}{4}$
constraints with slack at most three, and so on. The resulting number of
slack variables is then
\begin{align*}
	&\sum_{k=3}^{(N-1)/2} \binom{N}{k} \cdot \bigl\lfloor \log_2(k-1)+1 \bigr\rfloor && \textrm{ if $N$ is odd,}\\
	&\sum_{k=3}^{N/2-1} \binom{N}{k} \cdot \bigl\lfloor \log_2(k-1) +1 \bigr\rfloor + \frac{1}{2}\binom{N}{N/2}\cdot \bigl\lfloor \log_2(N/2-1)+1 \bigr\rfloor && \textrm{ if $N$ is even.}
\end{align*}

Therefore, we arrive at the desired result: When applying the standard procedure to reformulate the Dantzig-Fulkerson-Johnson model \eqref*{obj:DFJ}, as a \gls{QUBO}, the result has
\[\frac{N(N-1)}{2} + \sum_{k=3}^{\lceil N/2 \rceil-1} \binom N k \cdot \bigl\lfloor \log_2(k-1) + 1 \bigr\rfloor + \frac{1+(-1)^N}{4} \cdot \binom{N}{\lfloor N/2 \rfloor} \cdot \bigl\lfloor \log_2(N/2-1) + 1 \bigr\rfloor\]
binary variables. The first summand here gives the number of binary variables in the integer linear program. The next two summands give the number of new binary variables to represent the possible slacks. They generalize the two terms given above for odd or even $N$ and are derived by summing over all considered subset sizes and taking the number of such subsets times the necessary number of binary variables to represent the possible slack.

\section{Regression model for physical qubit requirements}\zlabel{app:regression}\label{app:regression}

In \cref{fig:dwave-qubit-cost-logs}, we present the observed relation between
the numbers of logical qubits $N$ (from the problem instance) and physical
qubits $N_{e}$ (from the embedding) for the \gls{D-Wave-OPT} approach. Each point
in the figure represents a solved instance, with random horizontal jitter added
for visibility and the color and shape indicating the problem class. Both axes
are on a natural logarithm scale. The dashed lines represent ordinary least
squares regression in logarithmic scale.

% - embeddings --- qubit cost (regression)
\begin{figure}[ht]
  \centering
  \includegraphics[width=0.7\textwidth]{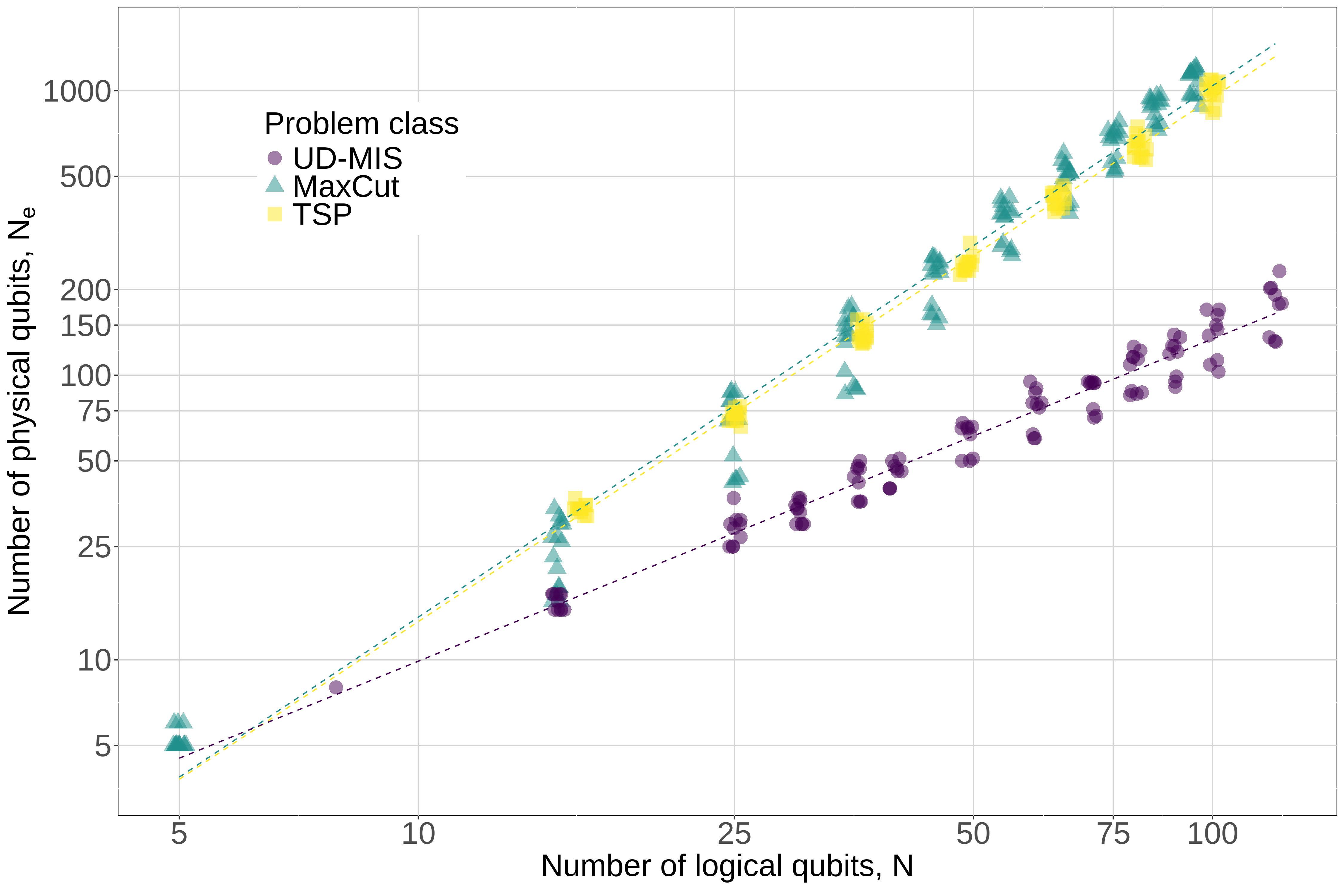}
  \caption{\label{fig:dwave-qubit-cost-logs}\zlabel{fig:dwave-qubit-cost-logs} Number of logical qubits
  	(from the problem instance)	and physical qubits (from the embedding) for the \gls{D-Wave-OPT} approach, shown on a logarithmic scale.}
\end{figure}

\pagebreak % Remove in next revision!
To perform this regression, consider first a model equivalent to the power law
for the number of physical qubits $N_{e}$ used after embedding as a function of
the number of variables in the original \gls{QUBO} model $N$, \ie, hypothesizing
that $N_{e} \sim N^{\beta}$. Taking the logarithm of both parts suggests the
following statistical model:
\begin{equation} \label{eqn:embeddingreg} \zlabel{eqn:embeddingreg}
	\log N_{e} = \beta_0 + \beta\log N + \varepsilon,
\end{equation}
which depends on two parameters, the slope $\beta \in \mathbb{R}$ and the
constant offset $\beta_0 \in \mathbb{R}$, and where the error term
$\varepsilon\defeq \log N_{e} - \beta_{0} - \beta\log N$ captures the
discrepancies of the predicted values
$\hat{N_{e}}\defeq \beta_{0} + \beta\log N$ with the true values $N_{e}$. This
model can be fitted to our embedding data, logarithms of the number of physical
qubits and logarithms of the respective numbers of binary variables, by
minimizing the sum of squared errors using the ordinary least squares regression
technique \citep{dalgaard2008}. The resulting standard summary table for the
model is shown in \cref{tab:regressions}.

% The output is adapted from figures/regression_summary.tex
\begin{table}[!htbp] \centering
	\caption{Regression results for the relation between logical qubits $N$ and physical qubits $N_{e}$ for the \gls{D-Wave-OPT} approach according to \cref{eqn:embeddingreg}. Standard errors are in parentheses, \texttt{df} denotes degrees of freedom, $^{***}~p<0.01$. The resulting relations (lines) are visualized in \cref{fig:dwave-qubit-cost-logs}. \label{tab:regressions}}
	\begin{tabular}{@{\extracolsep{5pt}}lccc}
		\\[-1.8ex]\hline
		\hline \\[-1.8ex]
		& \multicolumn{3}{c}{\textit{Dependent variable: $\log(N_{e})$ in different regressions}} \\
		\cline{2-4} \\[-1.8ex]
		& \gls{UD-MIS} & \gls{MaxCut} & \gls{TSP} \\
		\hline \\[-1.8ex]
		$\beta$: $\log(N)$ & $1.132^{***}$ & $1.867^{***}$ & $1.840^{***}$ \\
		& (0.020) & (0.023) & (0.012) \\
		& & & \\
		$\beta_0$ & $-$0.316$^{***}$ & $-$1.652$^{***}$ & $-$1.626$^{***}$ \\
		& (0.080) & (0.085) & (0.044) \\
		& & & \\
		\hline \\[-1.8ex]
		Observations & 116 & 150 & 105 \\
		R$^{2}$ & 0.964 & 0.978 & 0.996 \\
		Adjusted R$^{2}$ & 0.964 & 0.978 & 0.996 \\
		Residual Std. Error & 0.139 (\texttt{df} = 114) & 0.243 (\texttt{df} = 148) & 0.072 (\texttt{df} = 103) \\
		F Statistic & 3,064.768$^{***}$ & 6,679.321$^{***}$ & 25,468.300$^{***}$\\
		& $(\texttt{df} = 1;~114)$  & $(\texttt{df} = 1;~148)$ & $(\texttt{df} = 1;~103)$ \\
		\hline
	\end{tabular}
\end{table}

\section{Computational workflows summary}
Details for the three key steps of the computational workflow mentioned in
\zcref{fig:workflows} are summarized in \cref{tab:comp-steps}. First, the
discrete optimization problem at hand must be reformulated into a suitable problem class: \gls{QuEra-OPT} admits only \gls{UD-MIS}, while
\gls{D-Wave-OPT} and \gls{IBM-OPT} allow only \gls{QUBO}. Then,
the \gls{QC} and classical host machine (defining the computational framework)
need to be configured. The quantum device configuration is usually prepared in a
device-independent way, and then translated into the language of device-specific
instructions. Different technologies require different amounts of computational
overhead for this step. For the machines from \gls{IBM} and \gls{QuEra}, device
configuration is relatively straightforward (depending on how much
optional hardware-dependent optimization of the instructions is considered),
while the \gls{D-Wave} quantum
annealer configuration involves a computationally difficult task related to graph
embedding. Finally, the solution step is different
across the three technologies. The neutral-atom-based device and the quantum
annealer function in a similar (analog) mode, where the initial state is
prepared, then a pre-defined evolution schedule is executed, and the result is
read out. This process is repeated multiple times to obtain a distribution of
solutions and try to alleviate the possible errors due to hardware limitations.
\Gls{QAOA}, on the other hand, is a \gls{VQA}, a hybrid quantum-classical method.
As such, its computation stage is different as
it comprises a classical outer black-box optimization loop. However, after the
solution is found, the processing is usually similar: the quantum state
corresponding to the best found ansatz parameters is constructed and measured
multiple times to obtain a sample of solutions to the original problem.

\begin{landscape}
	\begin{threeparttable}
		\caption{Key high-level steps of the computational workflow.\label{tab:comp-steps}}
		\setlength{\extrarowheight}{7pt}
		\footnotesize
		\begin{tabularx}{\linewidth}{%
				>{\hsize=0.025\hsize}X
				>{\hsize=0.475\hsize}X
				>{\hsize=1.3\hsize}X
				>{\hsize=1.6\hsize}X
				>{\hsize=1.6\hsize}X}
			\toprule
			\textbf{Step} & & \textbf{\gls{QuEra-OPT}} & \textbf{\gls{D-Wave-OPT}} & \textbf{\gls{IBM-OPT}}\\
			\midrule
			&\textbf{\mystep\label{s:model}}~Formulation &  The problem is reformulated as \gls{UD-MIS}\tnote{a}.& \mergecol{2}{>{\hsize=1.5\hsize}X}{The problem is reformulated as \gls{QUBO}\tnote{b}.}\\
			\midrule
			\multirow{4}{*}[-4em]{\rotatebox[origin=c]{90}{\textbf{CONFIGURATION}}}& \textbf{\mystep\label{s:QPU-config}}~\gls{QC} setup & \mergecol{3}{>{\hsize=1.5\hsize}X}{Describing the initial state and key parameters of the quantum device:}\\
			& \textbf{\thesteps.1)}~Device-independent configuration &  Prepare problem-dependent spatial configuration for the atoms (the atoms array) and define the blockade radius.&
			Set up annealing time. &
			\begin{tabular}[t]{l}
				Design the quantum circuit (ansatz):\\[-5pt]
				\textbullet~derive the problem-specific circuit form,\\[-5pt]
				\textbullet~choose circuit depth.\end{tabular}\\
			& \textbf{\thesteps.2)}\label{step:device setup}.~Device-dependent configuration &  Translate the setup to physical device commands (automatically).&%
			\begin{tabular}[t]{p{1.6\linewidth}}
				Embedding: map logical qubits to the device nodes: \\[-5pt]
				\textbullet~choose required ``chain strength'',\\[-5pt]
				\textbullet~find the mapping.\end{tabular}&%
			\begin{tabular}[t]{l} The circuit is ``transpiled'' for the specific device: \\[-5pt]
				\textbullet~map the gates to device-native gates,\\[-5pt]
				\textbullet~add \texttt{SWAP} gates to ensure necessary qubit connectivity.\\[-5pt]
			\end{tabular}\\
			
			& \textbf{\mystep\label{s:cpu-config}}~Framework parameters. &\begin{tabular}[t]{p{1.3\linewidth}} \textbullet~Set up the system evolution schedule.\\[-5pt]
				\textbullet~Choose the number of shots.\\[-5pt]
			\end{tabular} &\begin{tabular}[t]{p{1.6\linewidth}} \textbullet~Choose the number of shots (sample size),\\[-5pt]
				\textbullet~implement the objective calculation procedure.\end{tabular}&%
			\begin{tabular}[t]{p{1.6\linewidth}} \textbullet~Choose initial circuit parameters,\\[-5pt]
				\textbullet~implement objective function estimation,\\[-5pt]
				\textbullet~set up a classical solver to find parameters.\end{tabular}\\
			\midrule
			\multirow{2}{*}[-2.5em]{\rotatebox[origin=c]{90}{\textbf{SOLUTION / PROCESSING}}} & \textbf{\mystep\label{s:run}}~Computation & \begin{tabular}[t]{p{1.3\linewidth}} Candidate solutions are sampled from the multiple shots of the \gls{QC}. During each step: \\
				\textbullet~prepare and initialize the device,\\[-5pt]
				\textbullet~realize the evolution schedule,\\[-5pt]
				\textbullet~sample a (single) solution.\\[-5pt]
			\end{tabular}&%
			\begin{tabular}[t]{p{1.6\linewidth}} Candidate solutions are sampled from the multiple shots of the \gls{QC}, each one comprising the following steps: \\[-5pt]
				\textbullet~prepare and initialize the device,\\[-5pt]
				\textbullet~anneal (converges the state to a low-energy one),
				\\[-5pt] \textbullet~sample a (single) solution.\end{tabular}&%
			\begin{tabular}[t]{p{1.6\linewidth}}%
				\textbullet~Given the circuit parameters, the \gls{QC} is set up to output a candidate solution. \\[-5pt]
				\textbullet~An auxiliary objective is set up as a function that takes circuit parameters, samples candidate solutions from the \gls{QC}, and returns the \gls{QUBO} objective.\\[-5pt]
				\textbullet~A classical black-box optimizer searches for circuit parameter values that minimize the auxiliary objective, iteratively querying the \gls{QC}.\\[-5pt]
				\textbullet~A set of solutions is sampled from multiple shots from the \gls{QC}, given the best circuit parameters found.\end{tabular}\\
			& \textbf{\mystep\label{s:pp}}~Post-processing & Recover \gls{MIS} solution(s): \eg, the most frequently sampled one.&%
			\begin{tabular}[t]{p{1.6\linewidth}}Recover \gls{QUBO} solution(s): broken chains resolved (e.g., majority vote) $\rightarrow$ \gls{QUBO} solution for each shot.\\[-5pt]
			\end{tabular} &%
			Recover \gls{QUBO} solution(s), analyze convergence data.\\
			\bottomrule
		\end{tabularx}
		\begin{tablenotes}
			\item[a] A weighted version of the problem and the respective \gls{UD-MWIS} reformulation \citep{nguyen2023} can also be used, if atom-specific detuning is allowed by the hardware.
			\item[b] The conversion to \gls{QUBO} (\eg, incorporating
			constraints, integer variables etc.) can be automatized to some extent by existing software tools.
		\end{tablenotes}
	\end{threeparttable}
\end{landscape}

\section{Summary of the different runs across instances}\zlabel{app:runs summary}

Our experimental setup involves an exploration phase, in which we test the
feasibility of our approaches with regard to the proposed instance generation
methods. Key information regarding the considered quantum hardware is summarized
in \cref{tab:qc-hardware}.

\begin{threeparttable}[t]
	\centering\small
	\caption{Quantum hardware used for our quantum-powered approaches.}\label{tab:qc-hardware}
	\begin{tabular}{ccccc}
		\toprule
		Hardware          & Technology                        & Qubits & Used in approach & References  \\
		\midrule
		\gls{Aquila}	  & neutral-atom, analog  & \num{256}	   & \gls{QuEra-OPT}   & \citet{wurtz2023} \\
		\gls{Advantage}	  & superconducting, analog, chip v4.1& $\sim \num{5000}$\tnote{(a)}      & \gls{D-Wave-OPT} & \citet{mcgeoch2022} \\
	  \gls{IBMCusco},	  & superconducting, gate-based,&\multirow{2}{*}{\num{127}}       & \multirow{2}{*}{\gls{IBM-OPT}}    &\citet{ibm2023},\\
		\gls{IBMNazca}	  & chip family \emph{Eagle} &&&\citet{chow2021}\\
		\bottomrule
	\end{tabular}
  \begin{tablenotes}
    \item[a] Device topology contains \num{5760} physical qubits, out of which at least \num{5000} must be available, according to the referenced technical report. The exact number depends on the specific device configuration and calibration, and may vary over time.
  \end{tablenotes}
\end{threeparttable}

Calculations for some of the instances were repeated several times, to possibly
obtain a feasible solution, or for technical reasons related to the interactions
with the remote machines. A summary of all attempts is presented in \cref{fig:runs
	summary}. Each bar corresponds to a single instance, and counts correspond to
the number of times the respective instance of the given type (in columns) was
solved by each approach (in rows). Colors mark the run resultsm \ie, whether we were
able to retrieve a feasible solution, only infeasible ones, or no solutions at
all. The instances are sorted by size in the number of \gls{QUBO} variables for each of the three problem classes,
which is summarized into three groups (no more than 25 variables, between 26 and
100, and more than 100 variables) for readability.

For our analysis in the main text of the paper, we select one run per instance
(which is summarized in \zcref{fig:inst sizes}), as follows. If we were able to
obtain a feasible solution, we used the latest run that yielded one. Otherwise,
we used the latest run among those that yielded at least an infeasible solution,
or just the latest run if all of them failed.

\begin{landscape}
	\begin{figure}[ht]
		\centering
		\includegraphics[width=\linewidth]{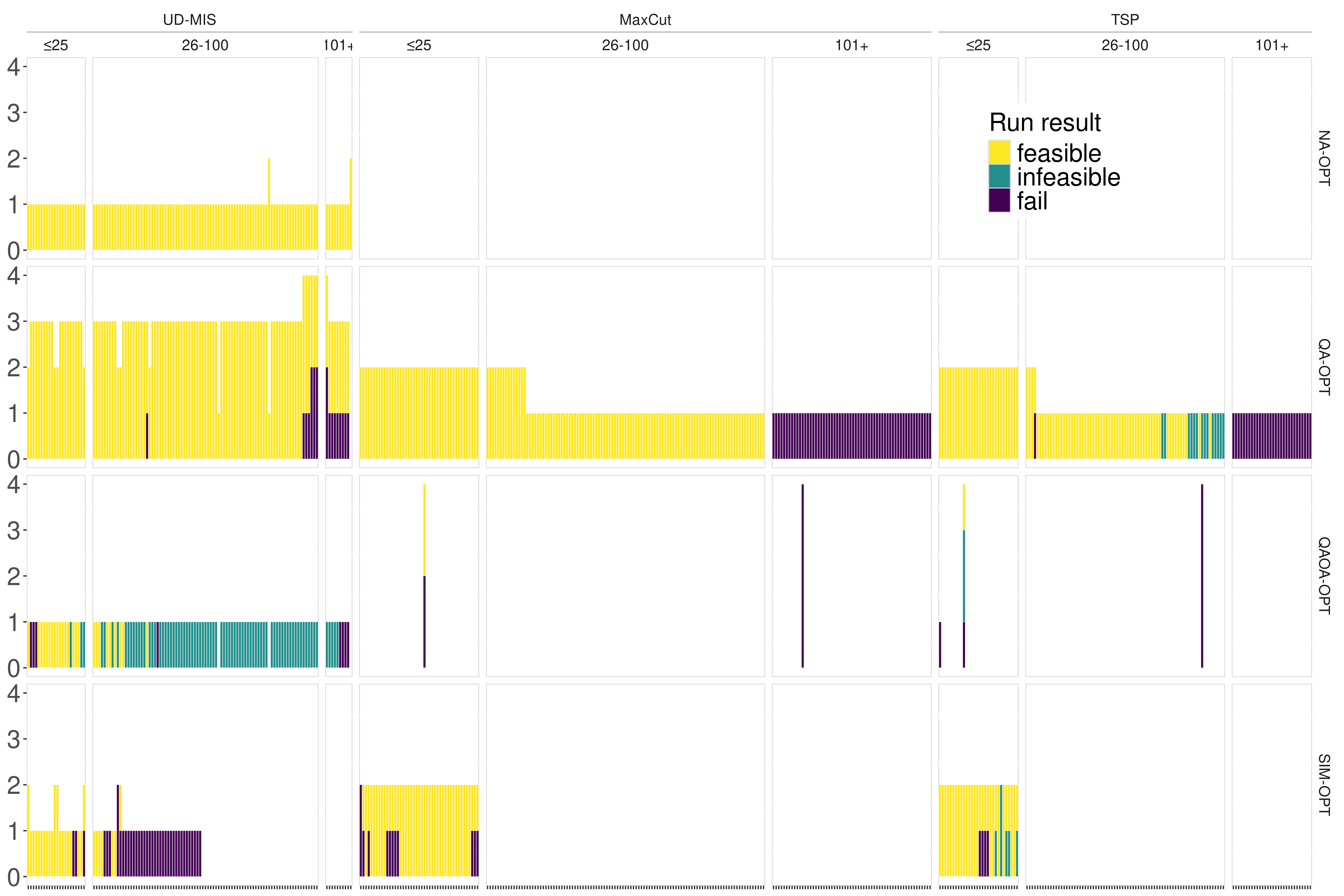}
		\caption{Summary of all runs across all problem instances. {\protect%
				% begin quote suppl !
				\revIII{Each panel represents the number of runs (counts), and each vertical bar corresponds to a single instance. Columns correspond to problem types and sizes (number of binary variables), rows correspond to solution approaches.}%
				% end quote
			}\label{fig:runs summary} }
	\end{figure}
\end{landscape}

\section{Example outputs from \texorpdfstring{\glspl{QC}}{QC}}\zlabel{app:sample-outputs}

Typically, the output of a \gls{QC} constitutes a sequence of bitstrings, each of which is obtained from a single shot.\footnote{The number of shots is a parameter chosen at the configuration step. In our experiments, we aim for a sufficiently large number of shots, which still maintains reasonable total runtimes.}
For quantum optimization, each bitstring can usually be transformed into a binary solution vector of the considered optimization problem.
For demonstration purposes, we present exemplary \gls{QC} outputs in the following for the different quantum-powered approaches, \ie, \gls{D-Wave-OPT}, \gls{QuEra-OPT}, and \gls{IBM-OPT} (including \gls{IBM-SIM-OPT}).

\paragraph{\texorpdfstring{\gls{D-Wave-OPT}}{QAOA-OPT}.}
\Cref{fig:DWave samples} summarizes three output examples for the
\gls{D-Wave} device \gls{Advantage} in our experiment. Each bar represents a single solution,
``solution count'' denotes the number of times this solution was sampled, and
``energy'' reflects the objective representation for the solution within the
\gls{QC} (which is an actual energy of the physical system in the corresponding
state). ``QUBO objective'' presents an objective value calculated from a
\gls{QUBO} formulation, and ``original objective'' corresponds to the objective
for the original problem.\footnote{For feasible solutions, the latter two might
differ only in sign, and no original objective values correspond to infeasible
solutions.} Thick horizontal line in the bottom panel represents the
optimal objective value. The panel entitled ``ch.~breaks'' represents the share
of chain breaks---a measure of the quality of the measured solution.\footnote{Chain breaks appear only for the largest instance in the figure.}

We see a varying quality of the representation for the true objective with the
quantum state. The top panel of \cref{fig:DWave samples} shows a typical
``favorable case:'' all the solutions are feasible (there are no feasibility
constraints for \gls{MaxCut}), and most of the sampled solutions are optimal or
close to optimal. The middle panel, highlighting a five-nodes \gls{TSP} instance
with internal ID \texttt{TSP53} (having $(N-1)^{2}=16$ binary variables in the
\gls{QUBO}), paints a more mixed picture. Most of the sampled solutions were
feasible (although not all), and the algorithm was able to find a true optimum.
However, it was not the most frequently sampled solution. The bottom panel
presents an even more unfavorable case of a larger \gls{TSP} with ID
\texttt{TSP82} and 49 \gls{QUBO} variables, an 8-node \gls{TSP} instance. First
of all, most of the sampled solutions were actually infeasible for the original
problem (hence, representing significantly suboptimal solutions to the
corresponding \gls{QUBO}). The solution frequency profile is not very
pronounced, without clear maxima and most of the solutions sampled once or
twice. The quality of the solution is worse, and there is a strictly positive
absolute gap between the best found solution and a true optimum.

\begin{figure}[ht]
	\centering
	\includegraphics[width=0.7\textwidth]{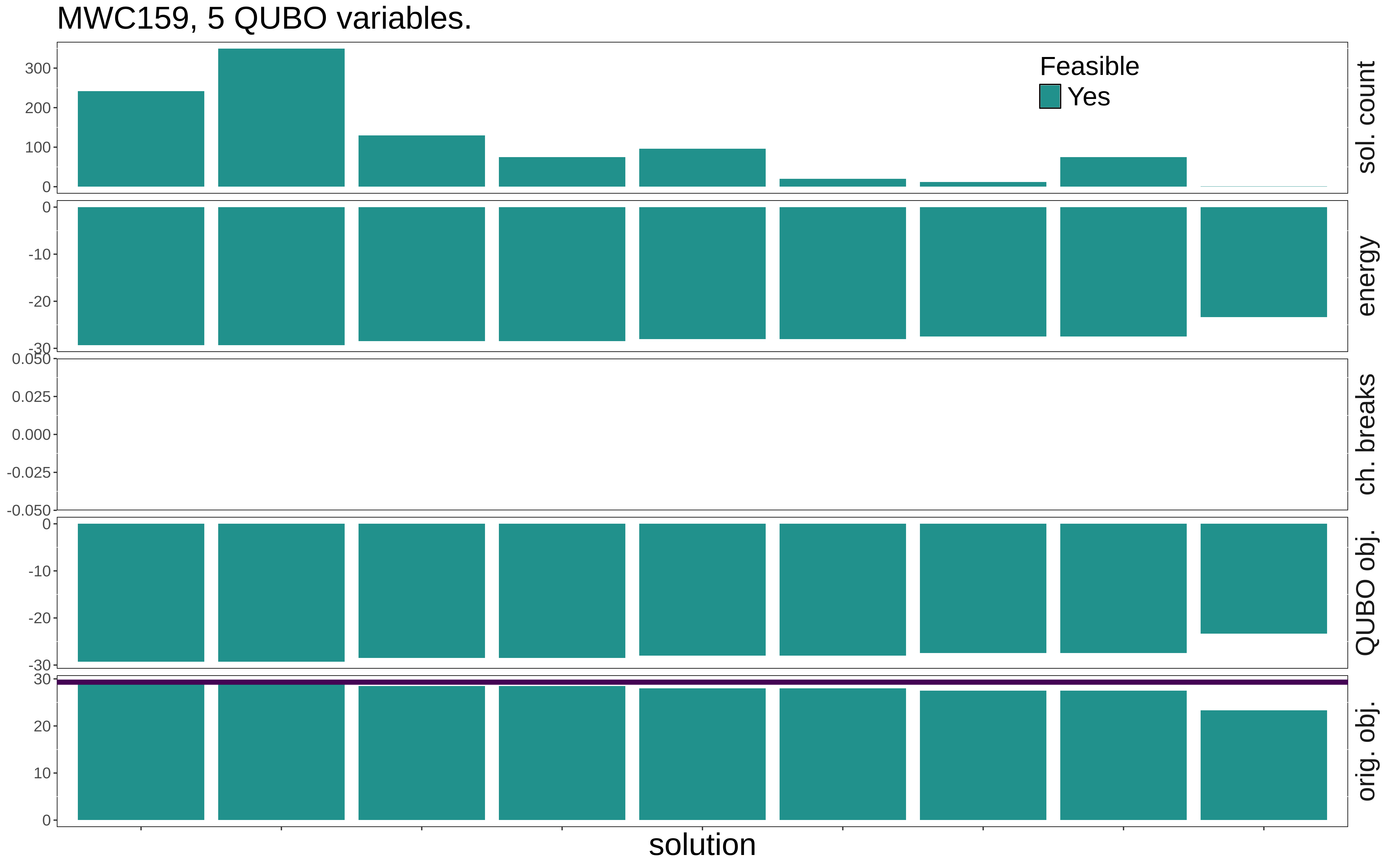}\vspace{\baselineskip}
	\includegraphics[width=0.7\textwidth]{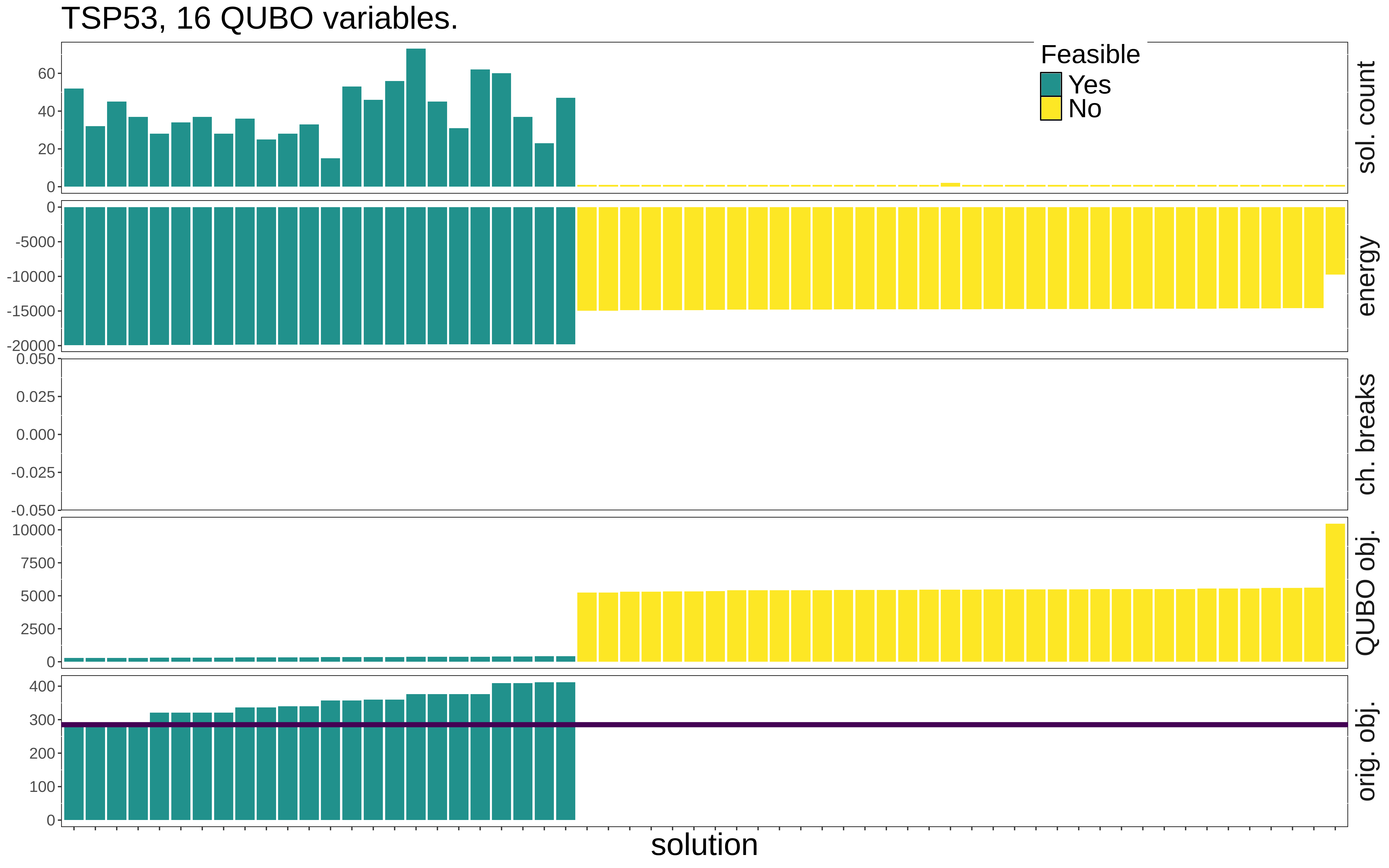}\vspace{\baselineskip}
	\includegraphics[width=0.7\textwidth]{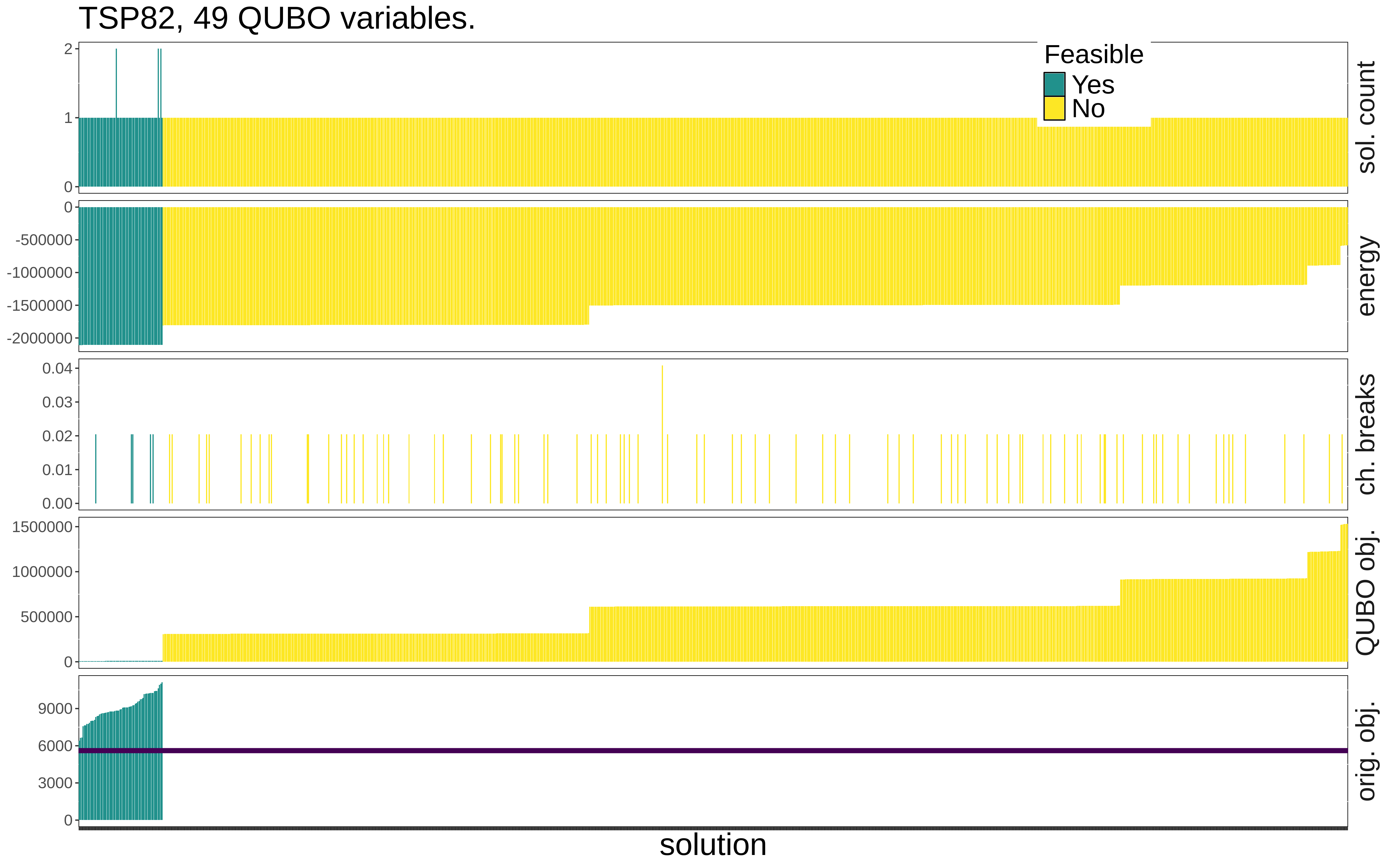}
	\caption{\label{fig:DWave samples} Output examples for the \gls{D-Wave-OPT} approach. Final sample for three selected instances on \gls{Advantage}. Dark color marks feasible solutions.}
\end{figure}

\paragraph{\texorpdfstring{\gls{QuEra-OPT}}{NA-OPT}.}
A similar situation is illustrated in \cref{fig:QuEra samples} for the
\gls{Aquila} device. Again, the top panel represents a relatively favorable
situation where more than half of the solutions are feasible, and the most
frequently sampled one corresponds to a true optimum. In the middle panel, a
significant number of solutions were sampled one or two times, but the frequency
profile still has a maximum around a true optimal solution. The bottom panel
represents yet another instance (which is approximately 4 times as large as
the one corresponding to the top panel), and here we see a completely flat
frequency profile of the sample: each solution was sampled exactly once. Such a
difference in output quality is somewhat surprising, as the three instances
mentioned here, \cref{fig:UDMIS1-4-7}, are very close in terms of the structure
and in fact represent different numbers of similar, but
unrelated \gls{UD-MIS} problems, which are solved in parallel within a single
shot. This highlights the fact that quantum-powered algorithms are complex and
sometimes might require additional fine-tuning. Some ``best practices'' and
practical considerations aiming specifically at the \gls{Aquila} device are discussed
by \citet{wurtz2023}.

\begin{figure}[ht]
	\centering
	\includegraphics[width=0.7\textwidth]{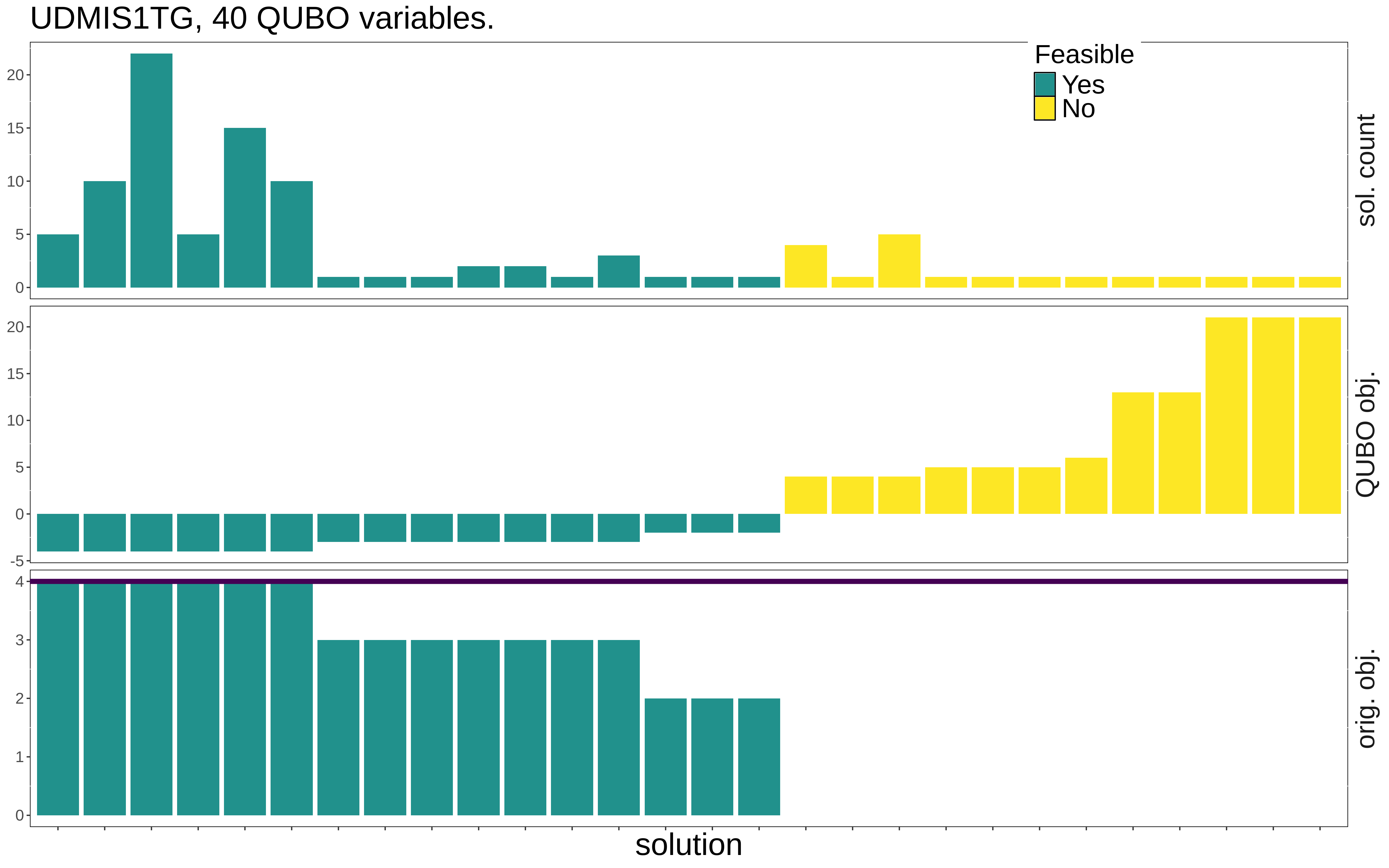}\vspace{\baselineskip}
	\includegraphics[width=0.7\textwidth]{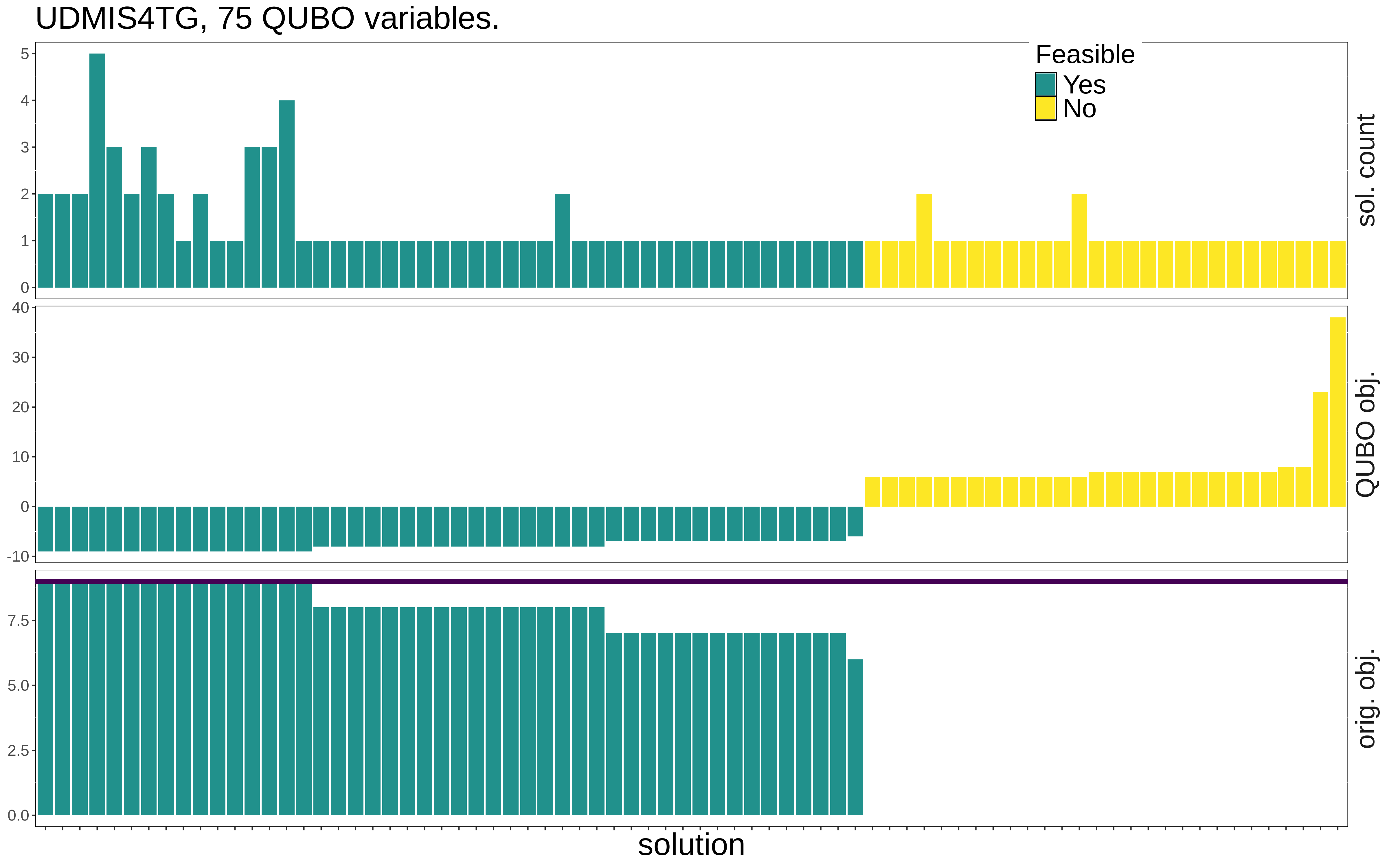}\vspace{\baselineskip}
	\includegraphics[width=0.7\textwidth]{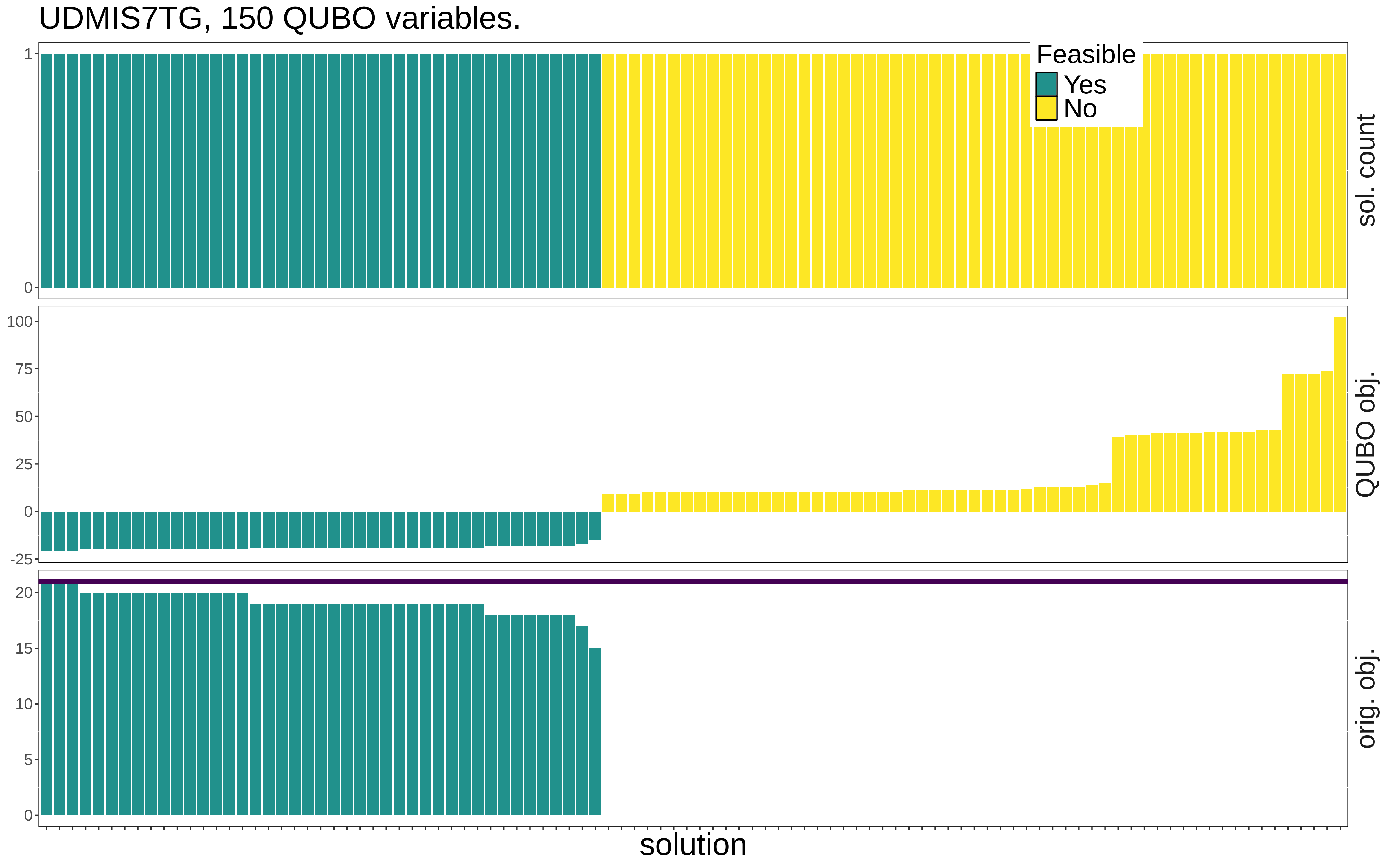}
	\caption{\label{fig:QuEra samples} Output examples for the \gls{QuEra-OPT} approach. Final sample for three selected instances (see \cref{fig:UDMIS1-4-7}) on \gls{Aquila}. Dark color marks feasible solutions.}
\end{figure}

\begin{figure}[ht]
	\centering
	\includegraphics[width=0.95\textwidth]{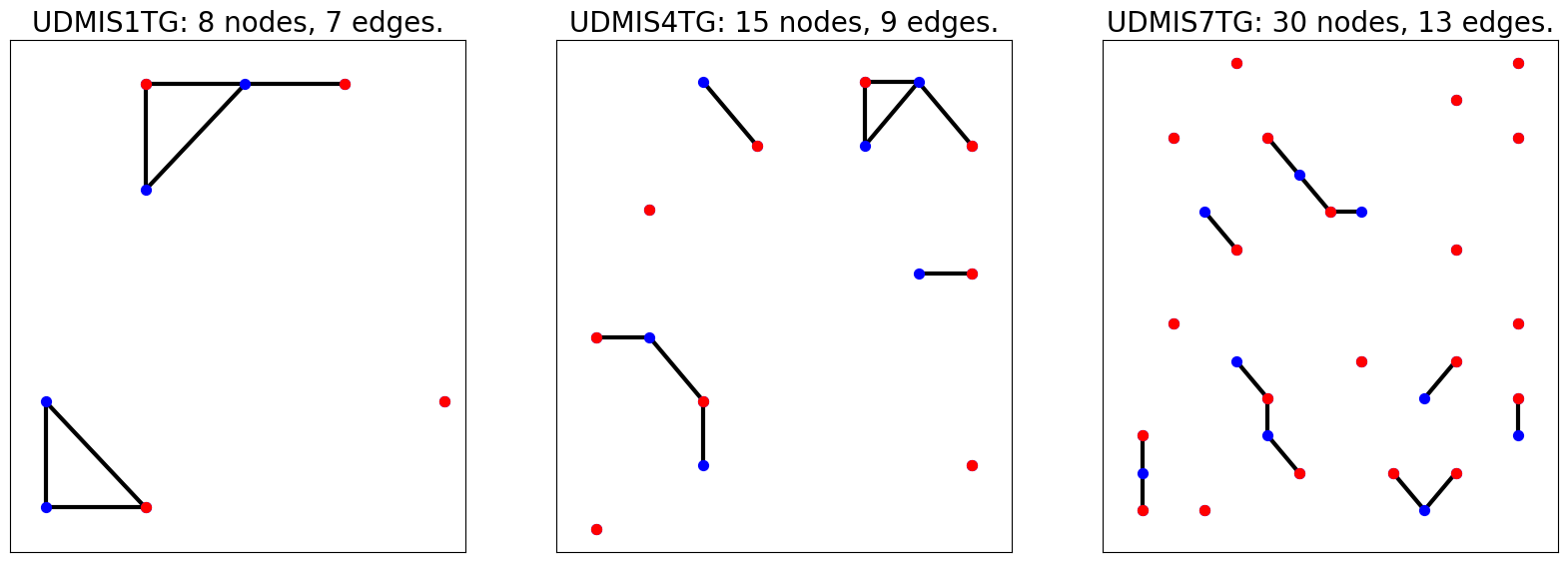}
	\caption{\label{fig:UDMIS1-4-7} Selected instances for \cref{fig:QuEra samples}.
		Original \Gls{MIS} graphs, colored points represent an optimal solution.}
\end{figure}

\paragraph{\texorpdfstring{\gls{IBM-OPT}}{QAOA-OPT} and \texorpdfstring{\gls{IBM-SIM-OPT}}{SIM-OPT}.}
Outputs for the same three instances from \cref{fig:UDMIS1-4-7}
on the \gls{IBM} device \gls{IBMNazca} and a noise-free simulator are presented in
\cref{fig:IBM samples}. For these instance sizes, the simulator was able to find
more feasible solutions, but overall the picture is the same. The smallest
instance, \texttt{UDMIS1TG}, yields a reasonable solution frequency profile.
Approximately doubling the number of variables results in most of the solutions
being sampled once or twice, and doubling the number of variables again yields a
completely flat frequency profile (\num{1000} different solutions sampled, out of \num{1000}
attempts), with noticeable number of infeasible solutions (more so for the
quantum device in comparison with the simulator).
In this direct comparison (based on our naive implementations), we find that the \gls{QuEra-OPT} approach, while offering less flexibility than the \gls{IBM-OPT} approach, leads to better solutions for the considered \gls{UD-MIS}
instances.

\begin{figure}[ht]
	\begin{minipage}{0.45\linewidth}
		\includegraphics[width=\textwidth]{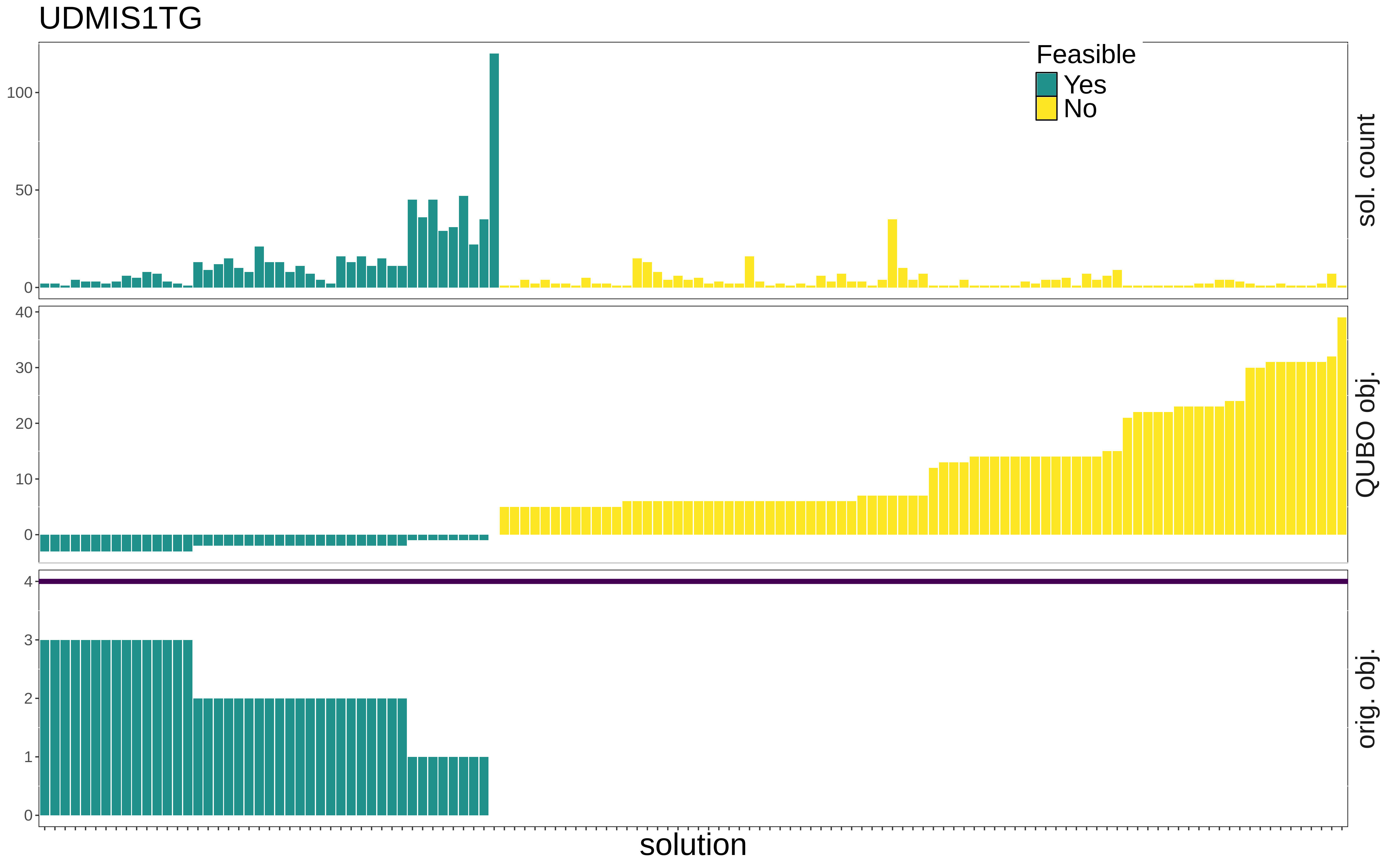}\vspace{\baselineskip}
		\includegraphics[width=\textwidth]{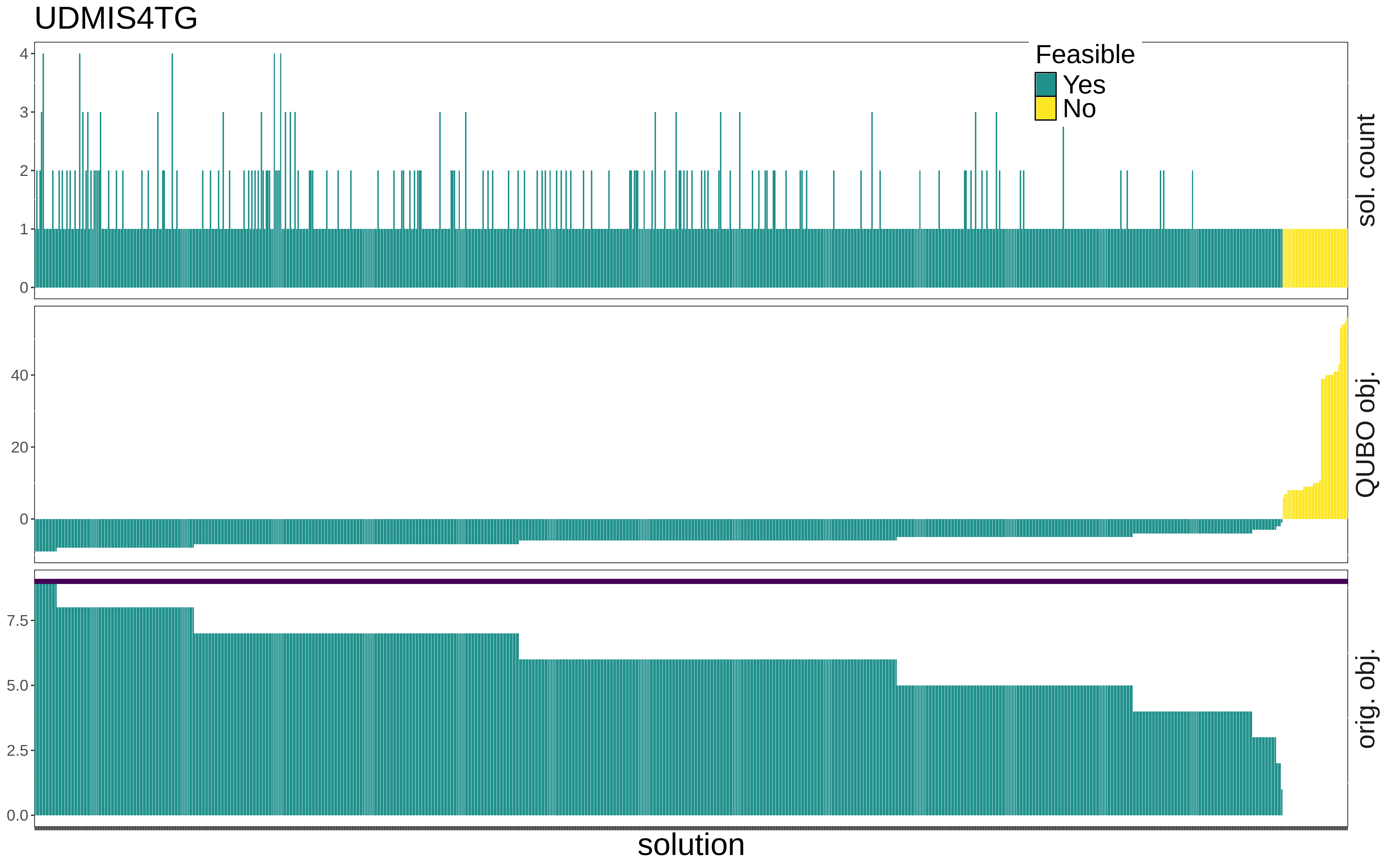}\vspace{\baselineskip}
		\includegraphics[width=\textwidth]{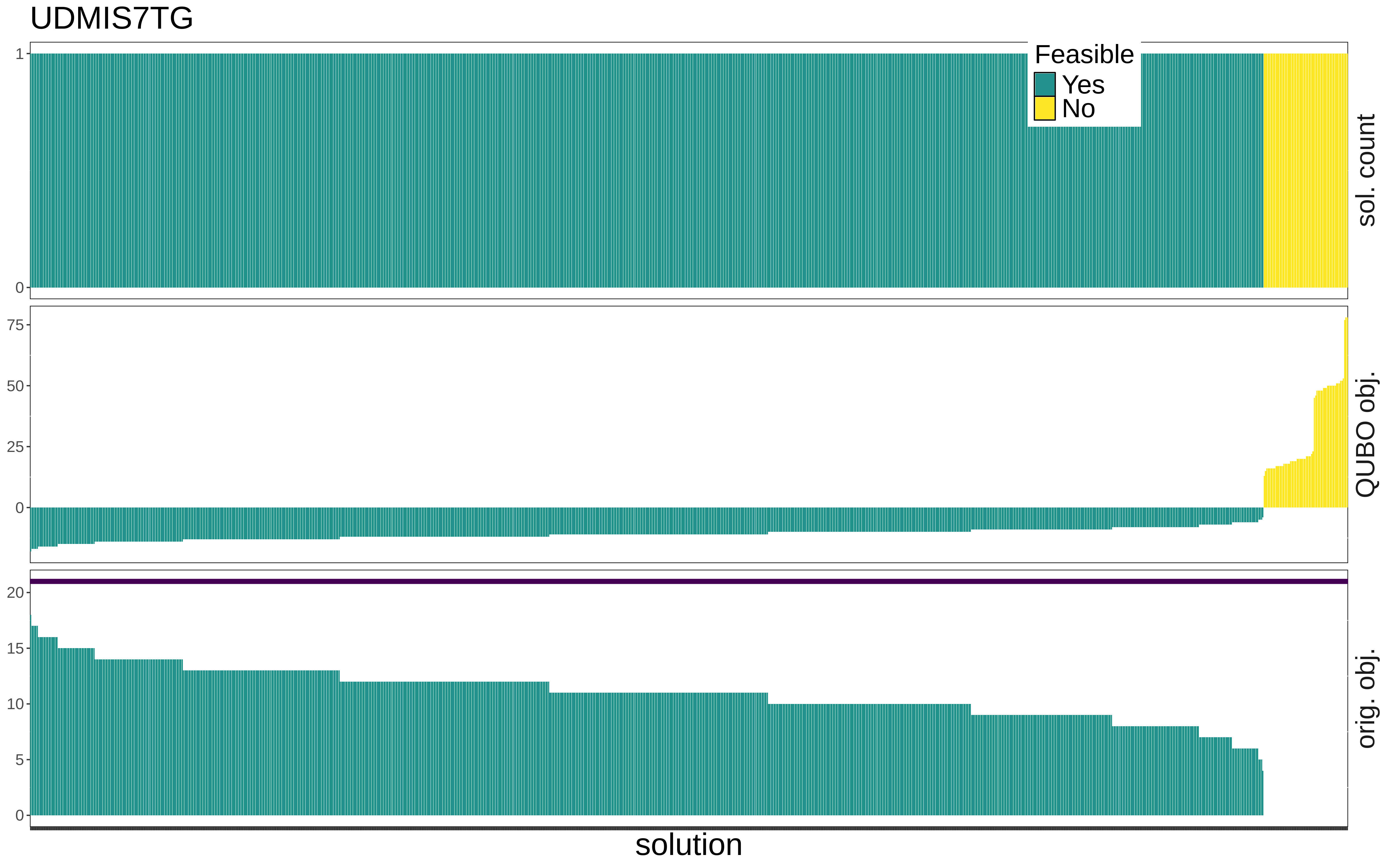}
	\end{minipage}\hfill%
	%%%%%%%%%%%%%%%%%%%%%%%%%%%%%%%%%%%%%%%%%%%%%%%%%%%%%%%%%%%%%%%%%%%%%%
	\begin{minipage}{0.45\linewidth}
		\includegraphics[width=\textwidth]{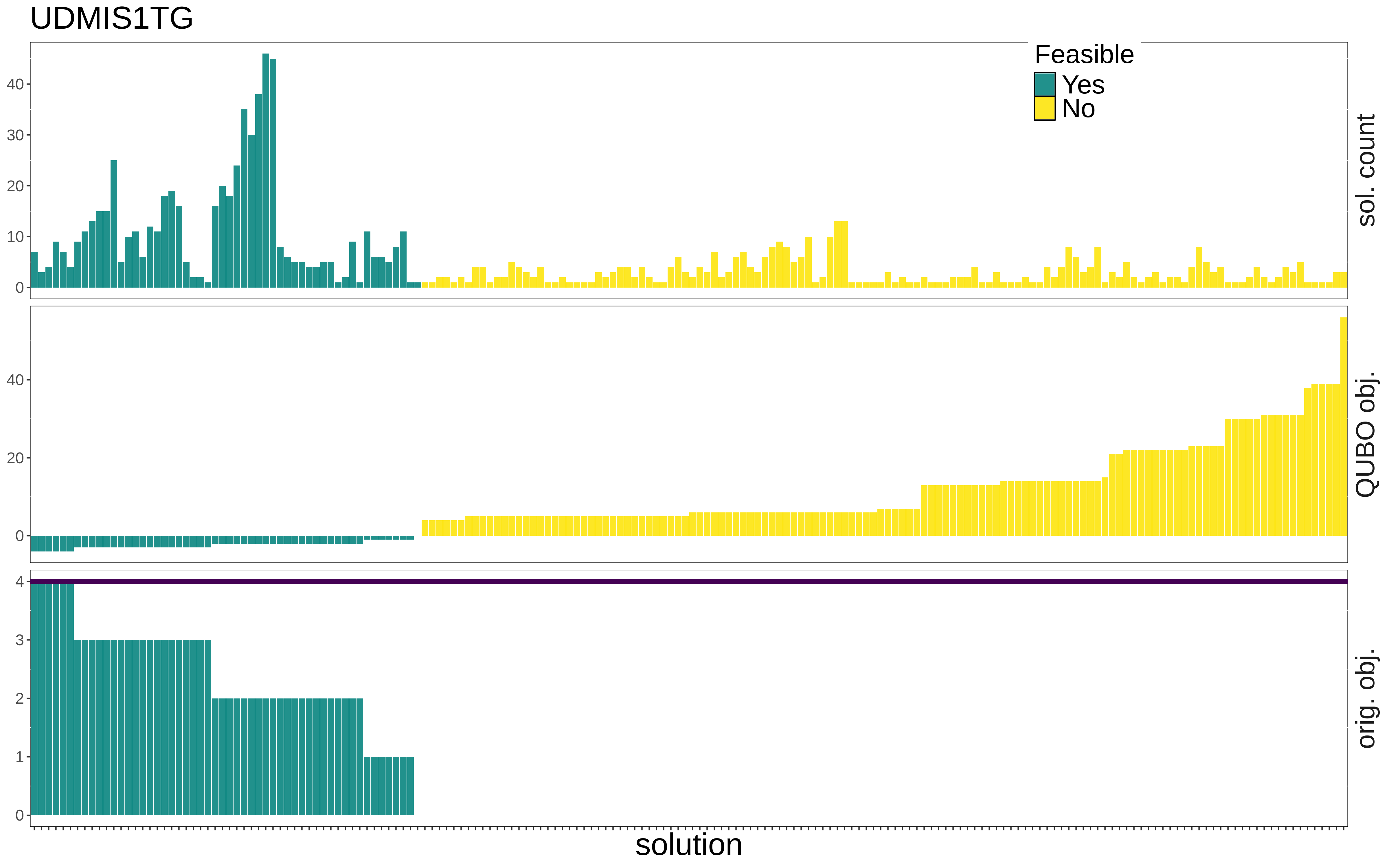}\vspace{\baselineskip}
		\includegraphics[width=\textwidth]{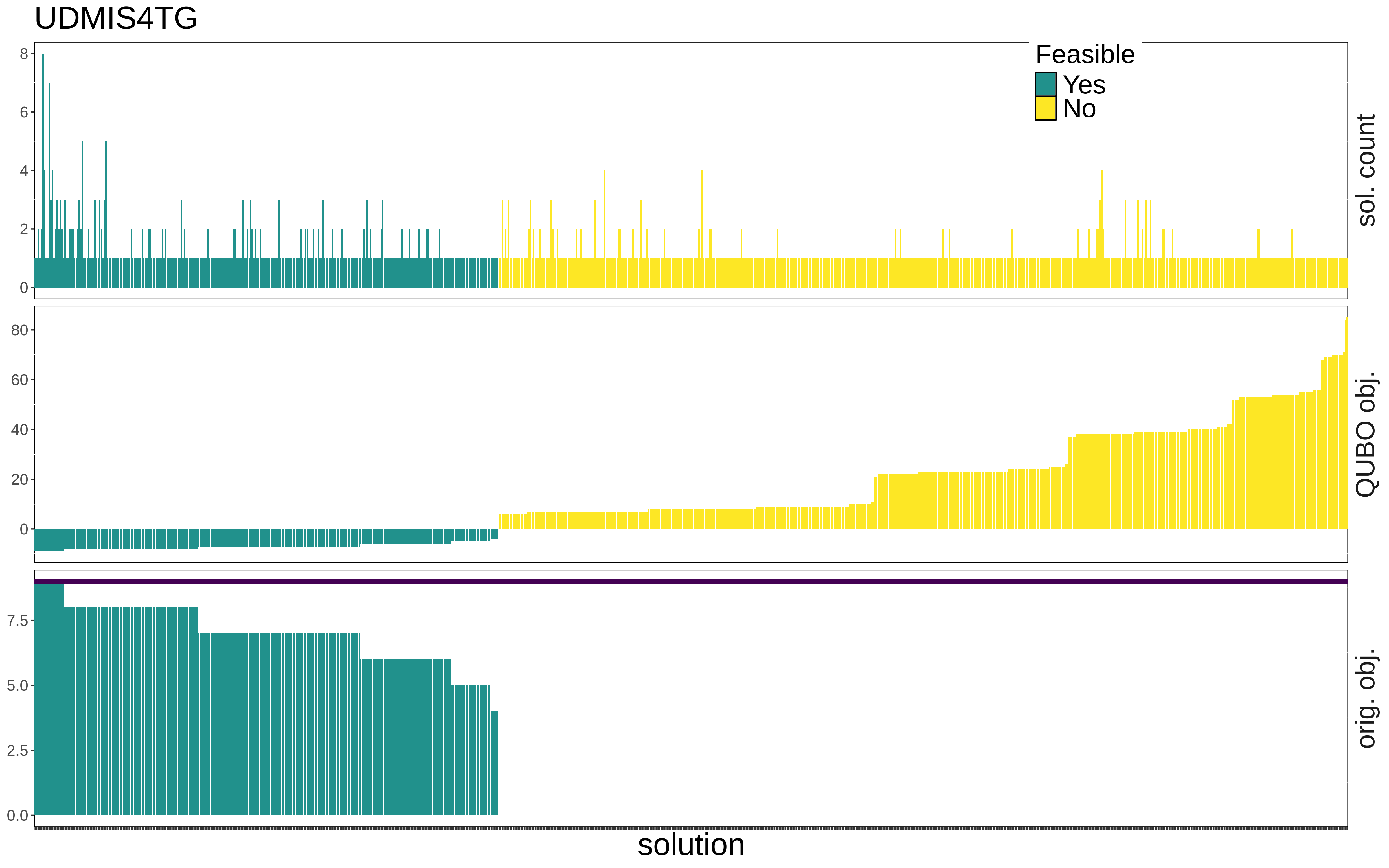}\vspace{\baselineskip}
		\includegraphics[width=\textwidth]{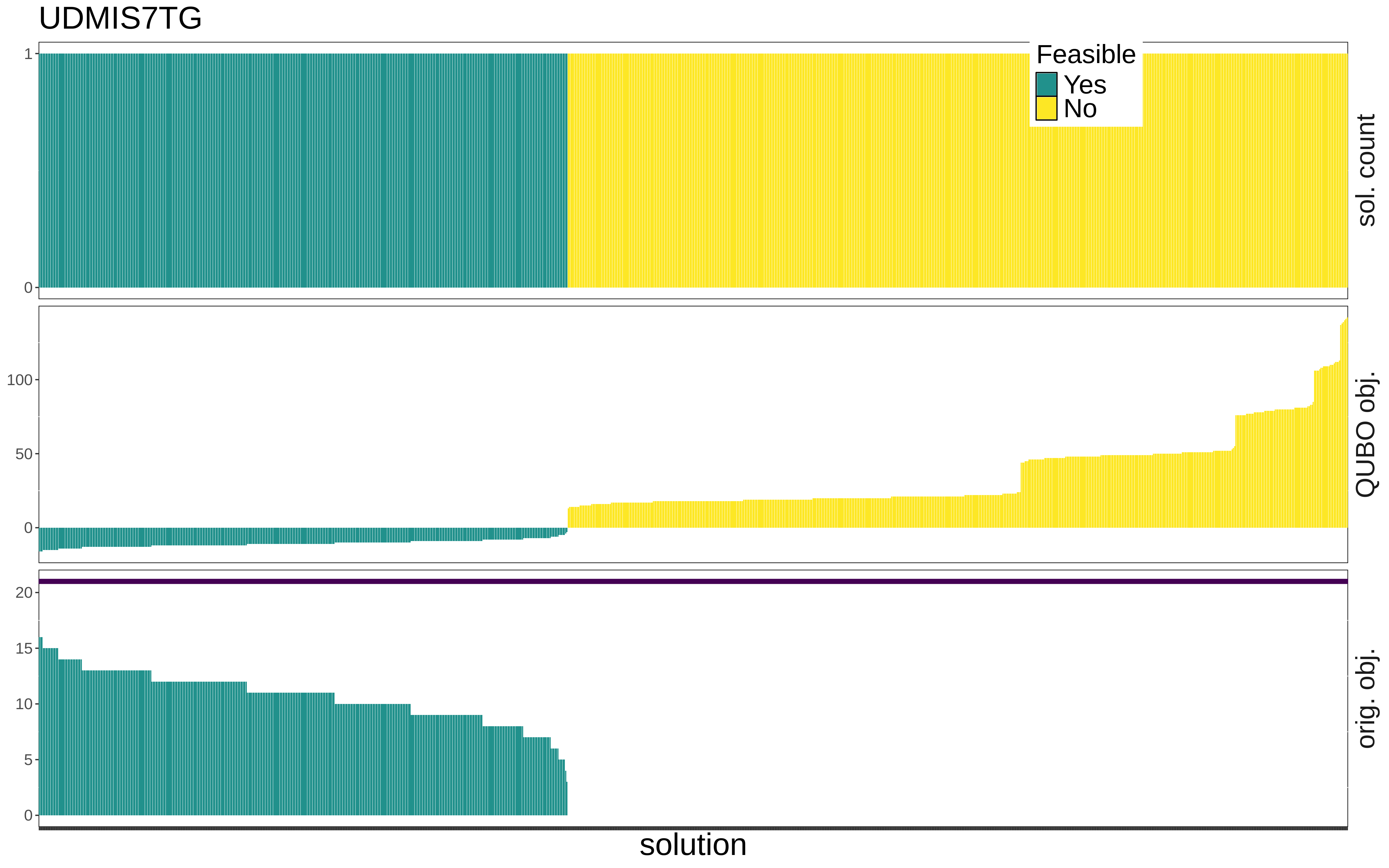}
	\end{minipage}
	\caption{\label{fig:IBM samples} Output examples for the \gls{IBM-SIM-OPT} approach (left) and the \gls{IBM-OPT} approach (right).
		Final sample for three selected instances
		on a noise-free simulator and \gls{IBMNazca}, respectively. Dark color marks feasible solutions.}
\end{figure}
% - IBM: convergence figures

\section{D-Wave embeddings: Chimera and Pegasus topology graphs}\zlabel{app:Pegasus}\label{app:Pegasus}

The topology graphs of \gls{D-Wave} quantum annealers are intersection graphs of axis-parallel segments.
An example, the so-called \emph{Chimera} topology is shown in \cref{fig:chimera-embeddings}: 
the top left panel shows the intersecting segments (thick lines), where
intersections of two neighboring horizontal or vertical qubits are represented
by thin lines between them. The resulting topology graph~$G_T$ is shown in the top right panel.

One safe way to obtain an embedding of the \gls{QUBO} graph~$G_Q$ into $G_T$ is to
search for an embedding of a clique~$K_n$ into $G_T$, which allows realizing all possible
interactions. The problem of finding the largest clique minor in a
broken topology graph is fixed-parameter tractable in the number of broken
qubits \citep{lobe2021}, \ie, it can be solved exactly with running time polynomial in
the size of the topology graph (but exponential in the number of broken qubits).
In this paper, for our upper bounds on the number of physical qubits required,
we consider clique embeddings into the non-broken \emph{Pegasus} graph, described below. In practice,
broken qubits may turn these embeddings invalid and lead to a larger number of
needed qubits, but on the other hand, the graphs~$G_Q$ resulting from our application
problems are usually rather sparse, which allows for a significant reduction of
qubits required as compared to our bound, which is confirmed by our experimental
results, see~\zcref{sec:results}.

To explain the clique embedding, we first consider the Chimera topology,
shown in \cref{fig:chimera-embeddings}, which was used by a previous generation of \gls{D-Wave} devices. It
consists of eight-qubit cells, tiled vertically and horizontally, and connected
to each other as presented in the top right panel
of \cref{fig:chimera-embeddings}. Each cell constitutes a complete bipartite
graph with four nodes on each side, denoted $K_{4,4}$. The nodes from one of
the bipartition classes (``horizontal qubits'', corresponding to horizontal segments)
are connected to their counterparts in the horizontally
neighboring cells, while the nodes from the other bipartition class (``vertical qubits'') are
connected to the vertically neighboring cells. Every node has degree at most
$6$, which makes it in particular necessary to use multiple nodes (\ie, physical qubits) to
represent any variable interacting with more than six other variables.
\Citet{choi2011} described a variant of clique embedding where always four
chains are grouped together and each $K_{4,4}$ realizes the interactions within
one such group or between two such groups \citep[see also][]{klymko2013,boothby2016,Date2019}.
For example, in 
\cref{fig:chimera-embeddings}, a quadruple of binary variables, such as $x_1,\dotsc,x_{4}$, is involved in five cells: one modeling the interactions within the quadruple and four for the interactions with each of the other variable quadruples.
Distinct quadruples are represented by different colors in the figure; the high-level
logic is shown in the bottom panel.

\begin{figure}
	\includegraphics[width=\textwidth]{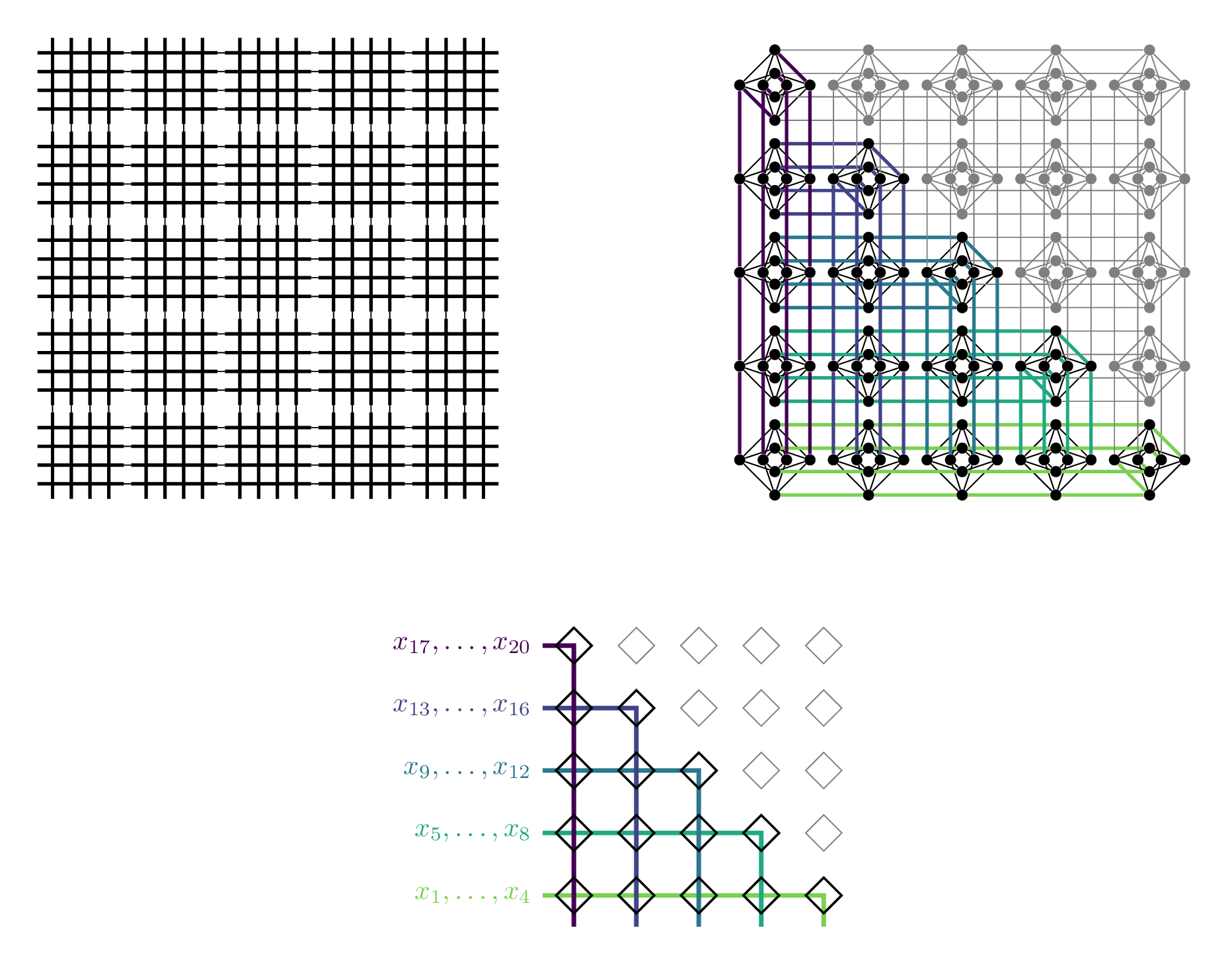}
	\caption{\Gls{D-Wave}'s Chimera topology: segments corresponding to
		superconducting loops (top left) and their intersection graph, the Chimera
		graph (top right); groups used in embedding of $K_{20}$ represented by
		colors; high-level logic of clique embedding (bottom) and explicit embedding
		(top right).\label{fig:chimera-embeddings}}
\end{figure}

The topology graph of the \gls{Advantage} device, the so-called \emph{Pegasus} graph
\citep{boothby2020,dattani2019a,dwavetopologies}, consists of three Chimera
graphs with additional connections, see \cref{fig:pegasus-embeddings}. %
\begin{figure}
	\includegraphics[width=\textwidth]{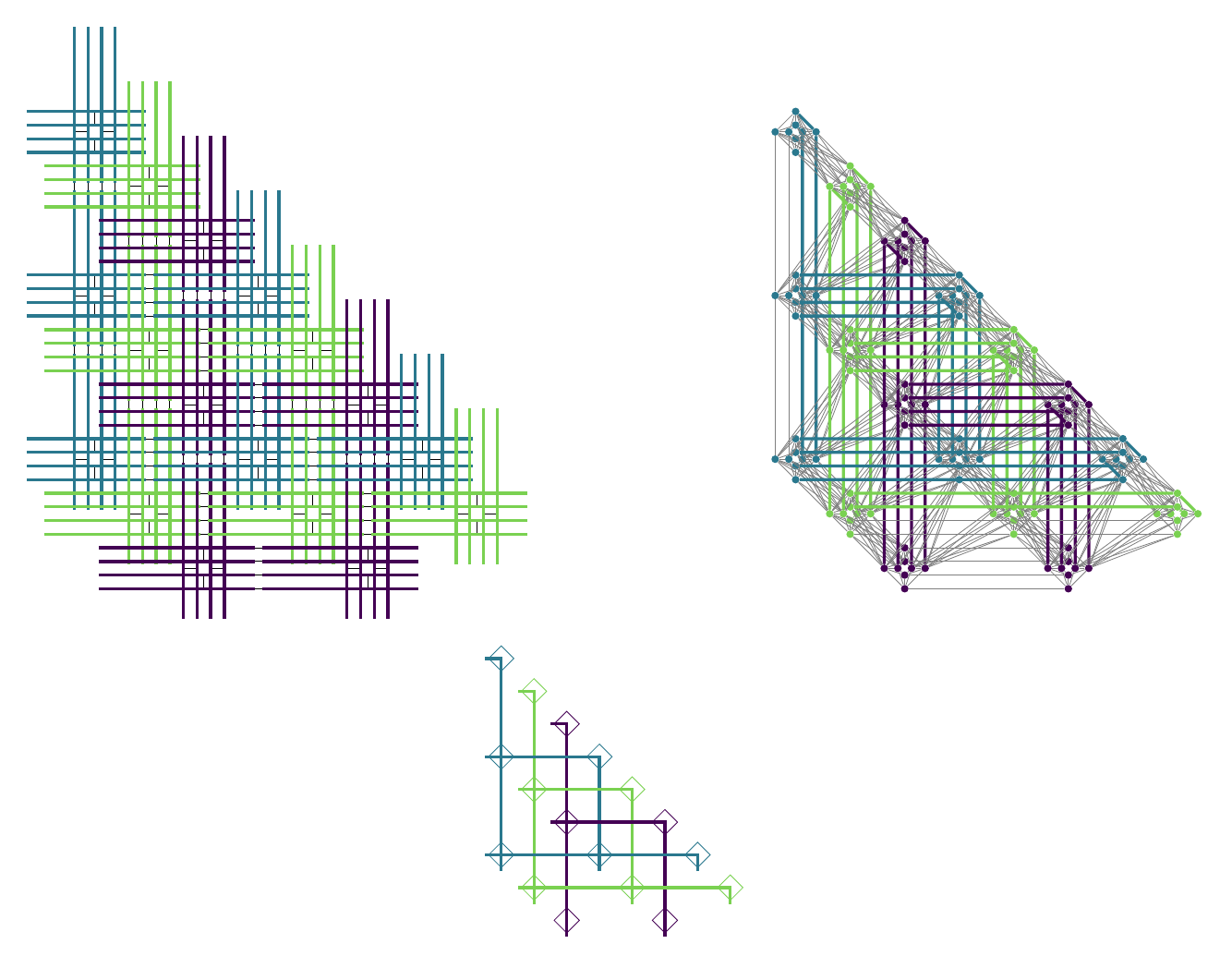}
	\caption{%
		Section of \gls{D-Wave}'s Pegasus topology: segments corresponding to
		superconducting loops (top left) and their intersection graph (top right),
		colors representing the three Chimera subgraphs. In the clique embedding,
		only intersections of chain groups of the same color are realized by
		$K_{4,4}$ cells; the others are realized by additional edges (bottom).
		Embedding of $K_{28}$ (top right, thick colored
		lines).}\label{fig:pegasus-embeddings}
\end{figure}%
Each inner node has, in addition to the six incident edges from its Chimera graph, one
extra edge to a neighboring node in its cell (breaking the
bipartiteness) and eight edges to nodes from the other two Chimera graphs, giving a
total node degree of~$15$. The connections between different
Chimera graphs within the Pegasus graph are designed in a way that combining the
Chimera embeddings of three cliques~$K_n$ yields an embedding of $K_{3n}$ (with
some caution at the boundaries). This means that in order to embed a clique of
$12n$ logical qubits, we can subdivide this into three cliques of size $4n$ and
embed each into a Chimera consisting of $n^2$ $K_{4,4}$-cells, resulting in
$3 n^2$ cells \revII{in total}, while an embedding into a single Chimera graph would require
$(3n)^2$ cells. Hence, (ignoring boundary effects), the number of required
physical qubits to embed a clique into the new graph is divided by three.
For either of the two topology graphs, an arbitrary \gls{QUBO} with
$n$ variables can thus be encoded using $O(n^{2})$ physical qubits, although with
different hidden constants. \Citet{pelofske2023} provides a detailed comparison
of \gls{D-Wave}'s device topologies. The exact number for Pegasus is given in
the following lemma.

\begin{lemma}[\citealp{boothby2020}]
	For every $M \in \mathbb N$ the clique $K_{12M - 10}$ is a minor of a Pegasus
	graph consisting of $24M(M-1)$ nodes corresponding to physical qubits.
\end{lemma}

This result immediately gives us that if we want to embed $n$ logical qubits, then we can apply the lemma to $M = \lceil\frac{n+10}{12}\rceil \le \frac{n+21}{12}$, resulting in at most $\frac{(n+9)(n+21)}{6}$ physical qubits for \revII{a} connected \gls{QUBO} instance, as presented in~\zcref{lm:pegasus-qubits}.

\revIII{As outlined in \zcref{sec:QA}, in our experiments the
	embeddings were calculated heuristically using the standard procedure offered
	by \gls{D-Wave}. For example, the actual embedding 
	for instance \texttt{MWC3}, a small \gls{MaxCut} instance, is presented in
	\cref{fig:MWC3-emb}. Specifically, the complete graph with five nodes was
	embedded into a subgraph of Pegasus topology comprising six, and not
	$(5+9)(5+21)/6 \approx 60$ nodes: The variable indexed by $2$ in the figure
	is represented by a chain of two nodes in the device graph. Across our
	collection of instances, we needed slightly less than quadratically
	many physical qubits, which depended on the problem structure (see~\zcref{sec:results}).
	Finding embeddings was very
	resource-intensive for larger instances, and constructing fast and reasonably
	effective heuristic algorithms might constitute an interesting direction for
	further research. While discussing the embeddings in more detail is beyond the
	scope of this article, we refer the reader to the works mentioned in \zcref{sec:QA} devoted to the topic for more details and remark that the
	embeddings for our considered problem instances are available for further analysis in the materials
	accompanying the paper (in the folder \texttt{run\_logs/dwave/embeddings}).}

\begin{figure}
	\centering
	\begin{minipage}{0.3\textwidth}
		\includegraphics{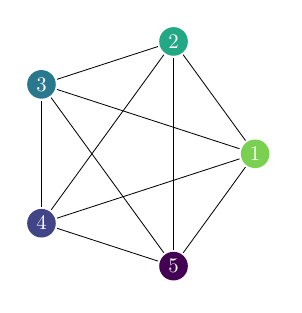}
	\end{minipage}%
	$\xrightarrow{\text{ embedding }}$ % Math arrow with text above
	\begin{minipage}{0.5\textwidth}
		\includegraphics{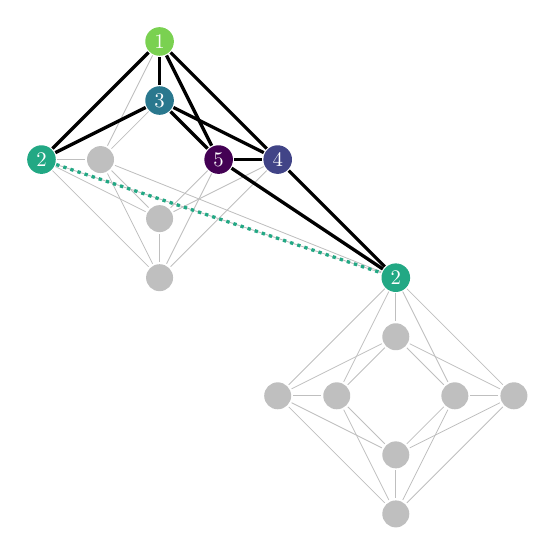}
	\end{minipage}
	\caption{Illustration for a small instance: embedding of a complete graph with
		five nodes (left) into Pegasus graph (right). A chain is represented by a dotted line between nodes labelled \textcircled{2}. Couplers used in the embedding besides the chain are depicted with solid thick lines. Most unused couplers between the two $4\times 4$ gadgets are omitted in the picture for readability.\label{fig:MWC3-emb}}
\end{figure}

Another representation of the experimental results that shows the annealing time
and embedding time separately is given in \cref{fig:embedding-time-shares}. The
left panel presents annealing time in seconds (without the embedding), while the
embedding time shares across all considered instances are summarized in
the right panel.

\begin{figure}[ht]
	\begin{minipage}[t]{0.475\linewidth}
		\centering
		\includegraphics[width=\linewidth]{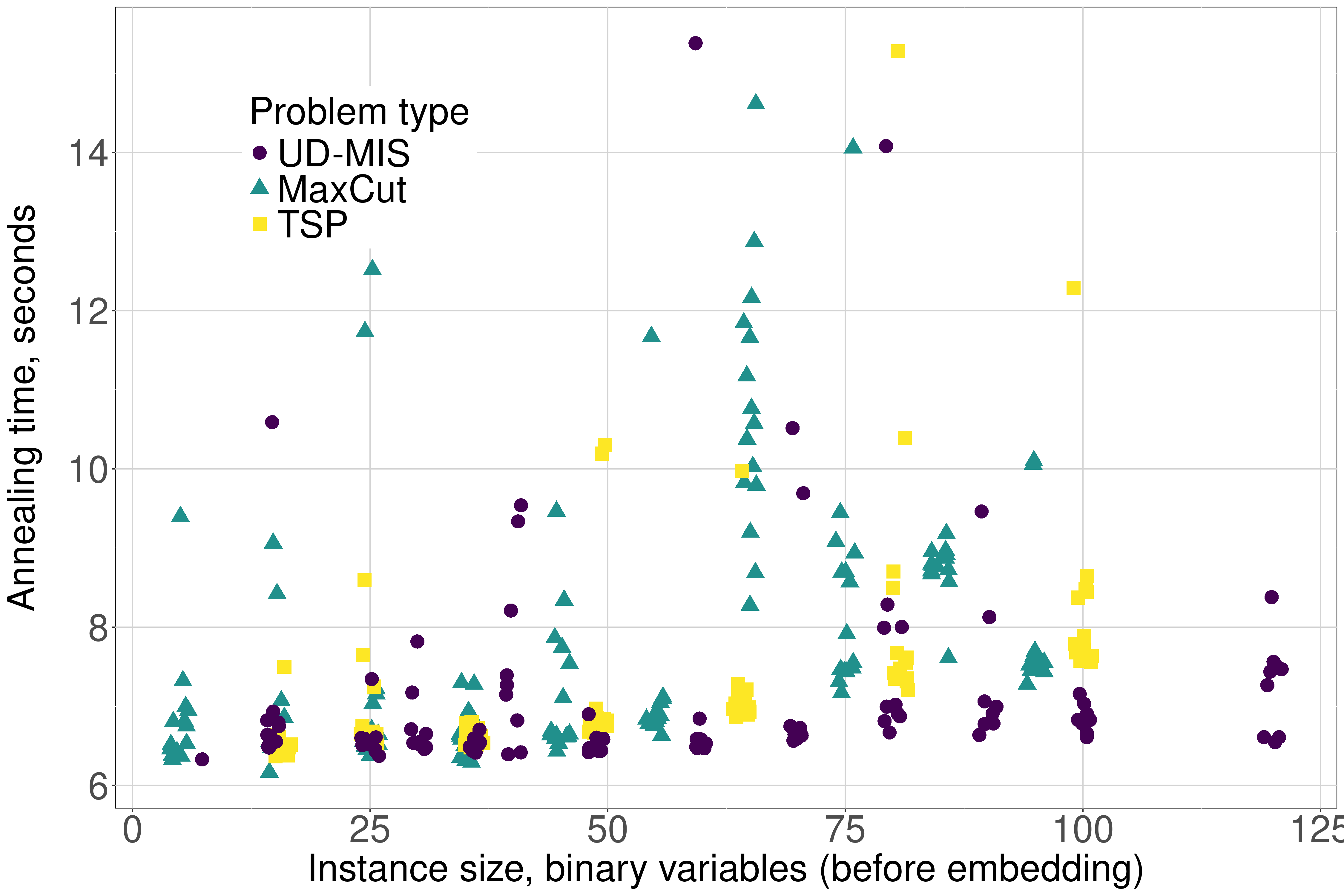}
	\end{minipage}\hfill%
	\begin{minipage}[t]{0.475\linewidth}
		\centering
		\includegraphics[width=\linewidth]{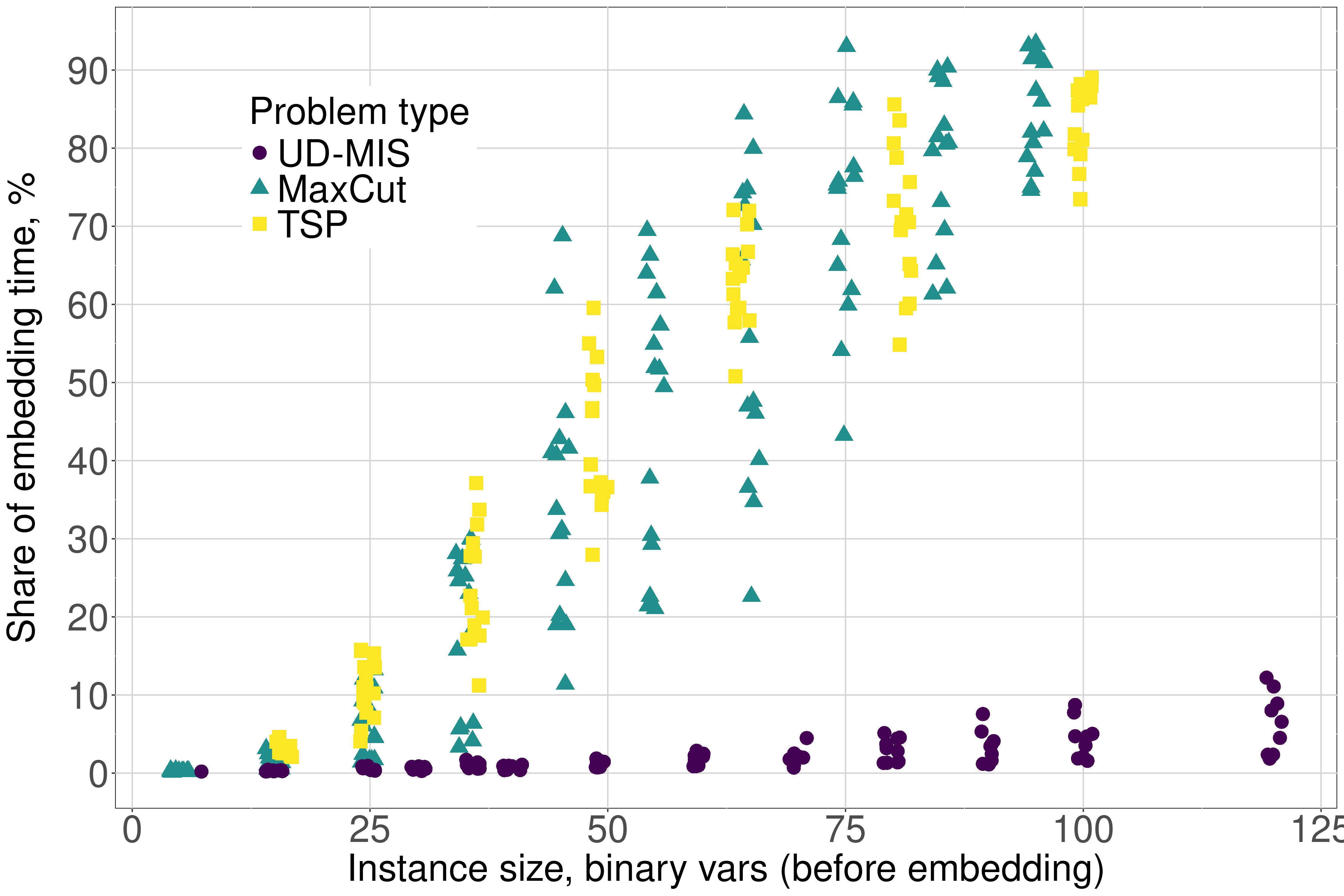}
	\end{minipage}
	\caption{\label{fig:embedding-time-shares} Annealing time (left) and the share of
		embedding time in the total runtime (right).}
\end{figure}
\clearpage

\section{On baseline model selection}\zlabel{app:baselines}

\revIIcomment{This section is new.}

Since numerical benchmarking is not the primary focus of this paper, we aimed to select baseline models that are as simple as possible while reflecting a natural first choice from the perspective of an Operations Research scientist. In this section, we elaborate further on our selection.

For {\protect\gls{TSP}}, many ways to model the problem are known, and
many algorithms exist, including exact, approximate, and heuristic approaches
{\protect\citep{lawler1985, junger1997}}. We consider three well-known
alternative formulations: {\protect\gls{DFJ}}, {\protect\gls{MTZ}}, and
{\protect\gls{QAP}}. Since the {\protect\gls{DFJ}} formulation has an exponential
number of constraints, we choose the next simplest linear model,
{\protect\gls{MTZ}}. Note that we have chosen another formulation (\gls{QAP}) as
a basis for the {\protect\gls{QUBO}} to be solved on the quantum annealer, as it
has lighter requirements on the number of qubits (see \cref{tab:formulations}). This way, we could have larger instances solved on a quantum
computer. Our experiments indicate that for the considered problem sizes and
structures, \gls{Gurobi} was essentially able to outperform our quantum annealing
pipeline, even as we restrict it to finding exact solutions. Therefore, we think
that our ``natural choice'' of {\protect\gls{MTZ}} model as a baseline was
enough to support the discussion in the paper.

Exactly the same logic applies for the {\protect\gls{UD-MIS}} problem set, where we
just used the natural binary program formulation as a baseline. However, for
{\protect\gls{MaxCut}} instances the formulation that we would consider the
first natural choice, the linear binary optimization problem (LBOP, denoted by
formulation~\ref*{eq:MaxCut-ILP}) was not fast enough to allow comfortable work
with larger instances in our collection of instances. Therefore, we improved the baseline
model: the {\protect\gls{QUBO}} formulation, denoted by (\ref*{eq:MaxCut-QUBO}),
somewhat counter-intuitively demonstrated better results with \gls{Gurobi} for
almost all the instances. These preliminary experiments are summarized in
\cref{fig:prelim-maxcut}. Each point corresponds to a single
{\protect\gls{MaxCut}} instance. For the instances where the exact optimum was
found for both formulations (left panel), we depict wall-clock runtimes for both
formulations. Where the exact optimum was not found for both due to the time
limit, we depict the respective values of the optimality gaps (right panel).
The tilted line indicates a situation when the {\protect\gls{QUBO}} and
the {\protect\gls{LBOP}} yielded the same runtime, and a point above (below) the
line indicates that for the respective instance the {\protect\gls{QUBO}}
(respectively, {\protect\gls{LBOP}}) formulation yielded faster convergence. We
see that both in terms of the runtimes and optimality gaps, the
{\protect\gls{QUBO}} formulation was beneficial. Therefore, for our
experiments, we considered the quadratic formulation also for our
classical baseline.

\begin{figure}[ht]
	\centering
	\begin{minipage}[t]{0.45\textwidth}
		\includegraphics[width=\textwidth]{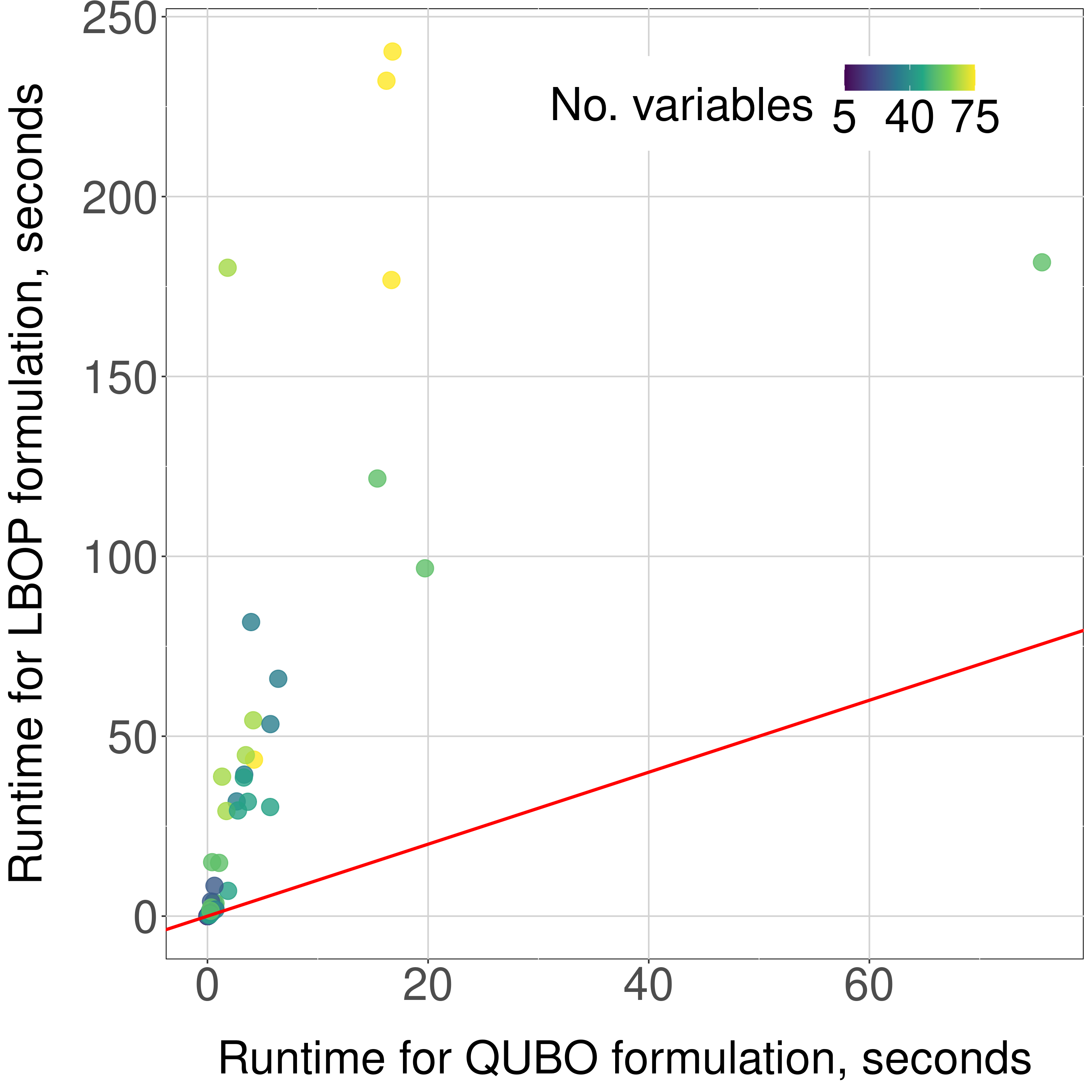}
	\end{minipage}\hfill%
	\begin{minipage}[t]{0.45\textwidth}
		\includegraphics[width=\textwidth]{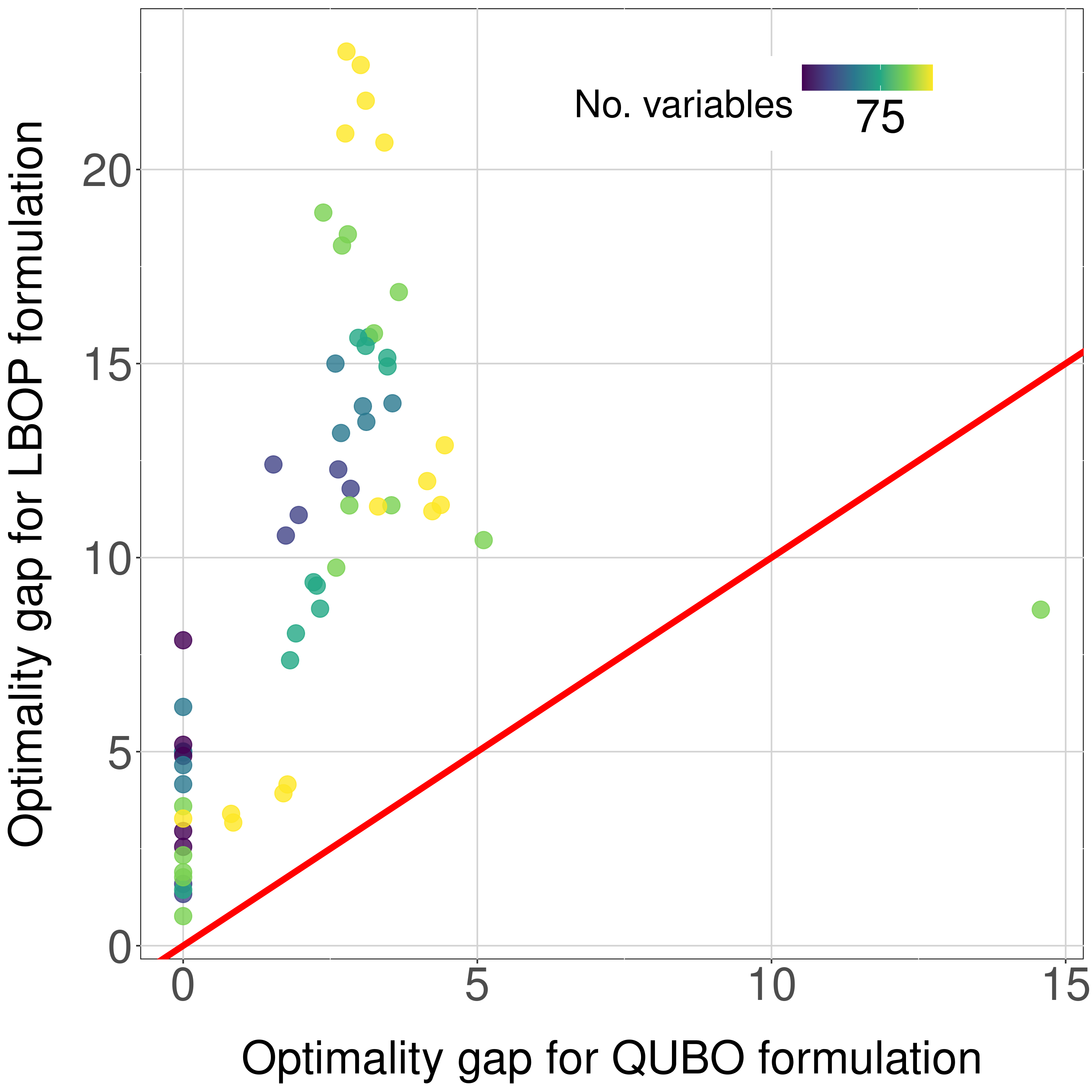}
	\end{minipage}
	\caption{Runtime (left) and optimality gap (right) comparison for
		{\protect\gls{MaxCut}} instances: {\protect\gls{LBOP}} against
		{\protect\gls{QUBO}} formulations.\label{fig:prelim-maxcut}}
\end{figure}

In fact, the classical solver was often able to guess a good solution early on.
We illustrate the branch-and-cut convergence data obtained from \gls{Gurobi} 
for three selected \gls{MaxCut} instances in \cref{fig:Gurobi-conv}.
% \sveninline{To me, this sentence confusing here, after we have discussed for half a page that we sometimes use MIP and sometimes quadratic formulations. After the discussion I thought that for TSP and UD-MIS we use integer linear minimization problems and for MaxCut we used a quadratic maximization problem. It would be ideal if after the previous paragraph, it is crystal-clear what exactly we solved classically, so that we don't need a potentially confusing `Note that'-addition here.}
% \abinline{Yeah, I think you are right. Immediately from the last sentence of the previous paragraph
	% it seems to follow that MaxCut → QUBO.}
A usual situation is
depicted in the left two panels: A good solution is obtained heuristically from
the very beginning, and then, most of the time is spent for improving the upper
bound to prove optimality, or for getting a relatively minor improvement of the
solution (as it is the case on the left panel). While this situation holds for
the majority of the instances, certainly, it is not true in general. For
example, for one of the instances (the right-most panel in the figure) the
solver was able neither to improve the bound, nor to update the solution in the
given timeframe.

\begin{figure}[ht]
	\centering
	\begin{minipage}[t]{0.3\textwidth}
		\includegraphics[width=\textwidth]{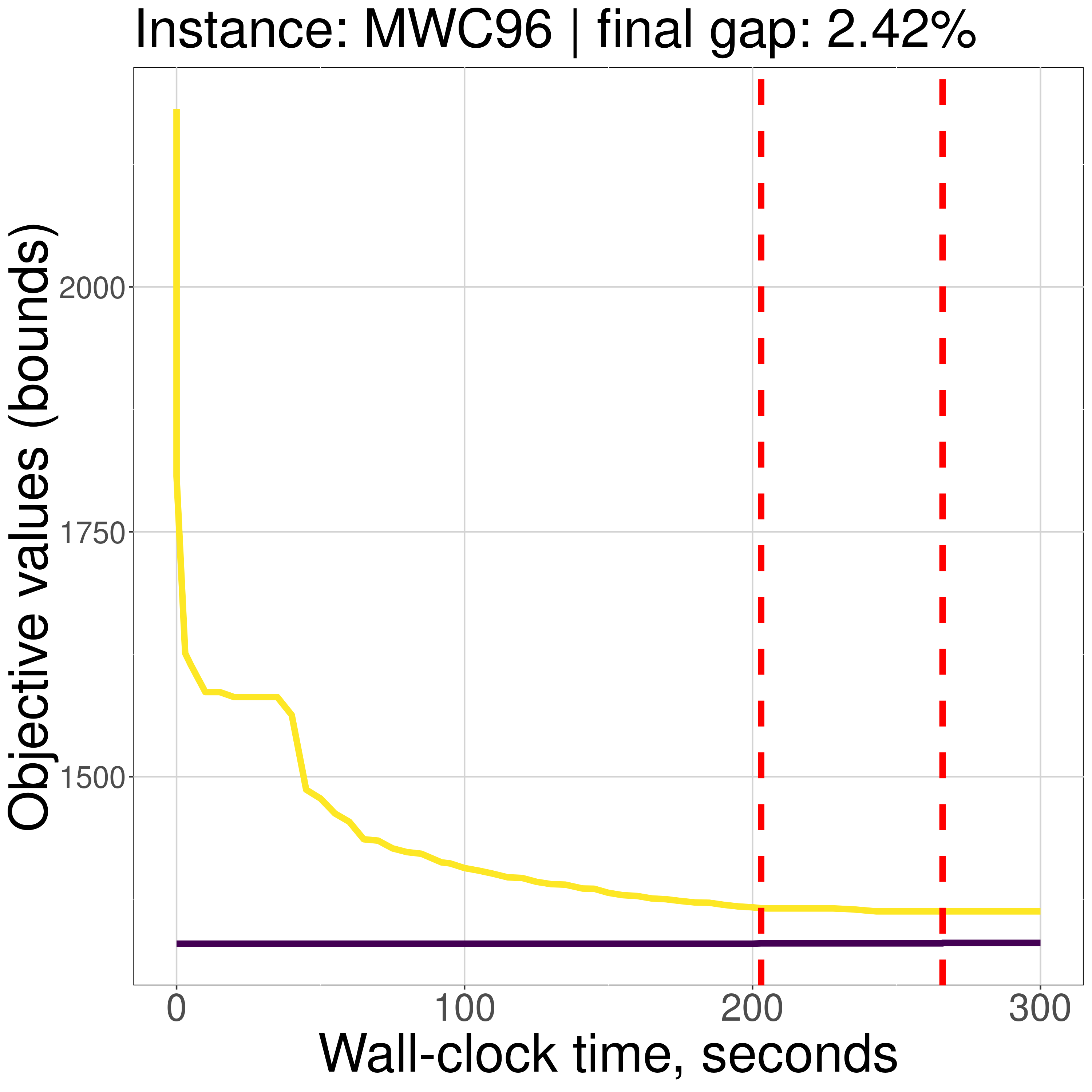}
	\end{minipage}\hfill%
	\begin{minipage}[t]{0.3\textwidth}
		\includegraphics[width=\textwidth]{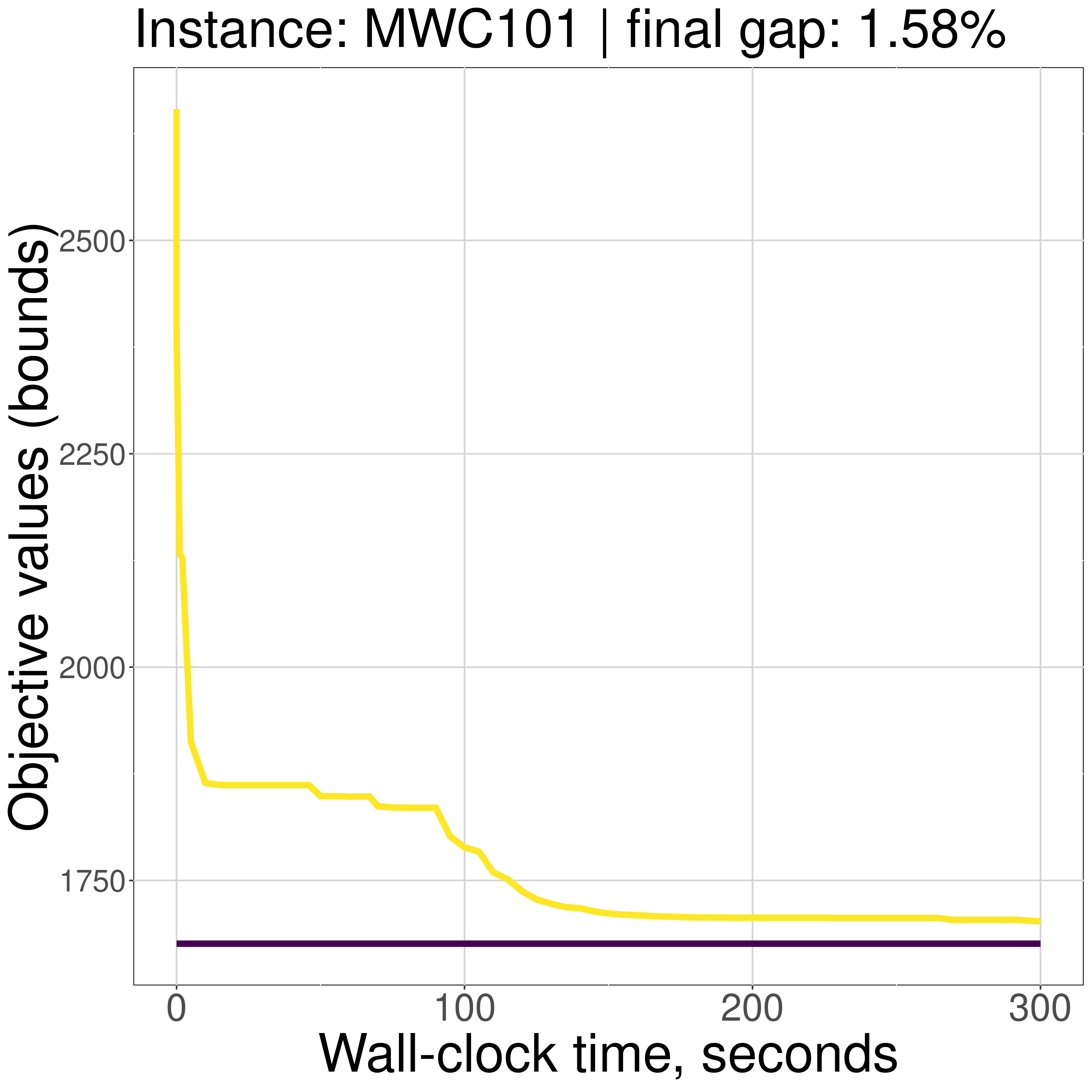}
	\end{minipage}\hfill%
	\begin{minipage}[t]{0.3\textwidth}
		\includegraphics[width=\textwidth]{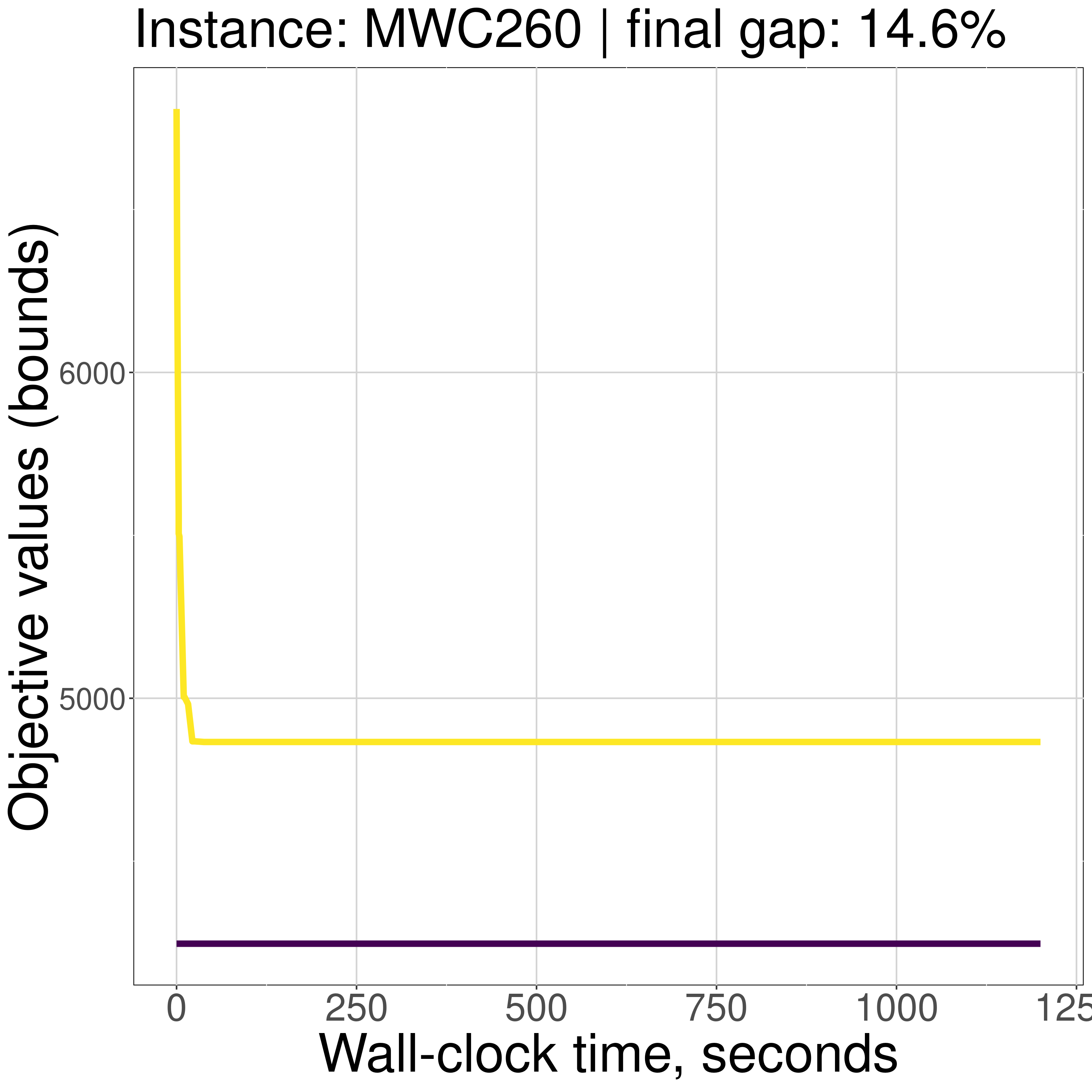}
	\end{minipage}\hfill%
	\caption{Classical solution convergence for \gls{QUBO} (maximization)
		formulations of three selected instances, left-to-right: 55, 75, and 85
		binary variables. Upper line corresponds to the best bound, lower line
		reflects the current incumbent solution in the branch-and-cut tree. Vertical
		dashed lines denote updates on the current best (incumbent)
		solution.\label{fig:Gurobi-conv}}
\end{figure}

Such a convergence picture immediately suggests a benchmarking methodology.
We first solve an instance using the quantum annealer and record the runtime,
including the embedding time and the actual annealing time. Further, we run 
\gls{Gurobi} using the recorded time as a timeout, and compare the objective
values. The convergence patterns depicted in \cref{fig:Gurobi-conv} suggest that
such heuristic would be comparable in objective quality to the baseline we used
in the main text of the paper. The results of such an experiment for
\gls{MaxCut} instances are presented in \cref{fig:MaxCut-gtimeouts}. The
relative objective deviations, when they are far from zero, are essentially
similar to the ones depicted in \cref{fig:runtime-vs-classic-a}, but the
relative runtime deviation is forced to zero.

Therefore, we would like to emphasize it again that we cannot claim that our
quantum-powered heuristic \emph{outperforms} the classical methods. However, we
identified a group of instances, which are non-trivial enough (judging from the
time it takes to find and prove optimality using our usual methods), where the
annealer is capable of producing solutions of reasonable quality. Certainly,
careful benchmarking would require implementing both better classical and better
quantum algorithms, and would constitute an interesting further research
direction.

\begin{figure}[ht]
	\centering
	\includegraphics[width=0.7\textwidth]{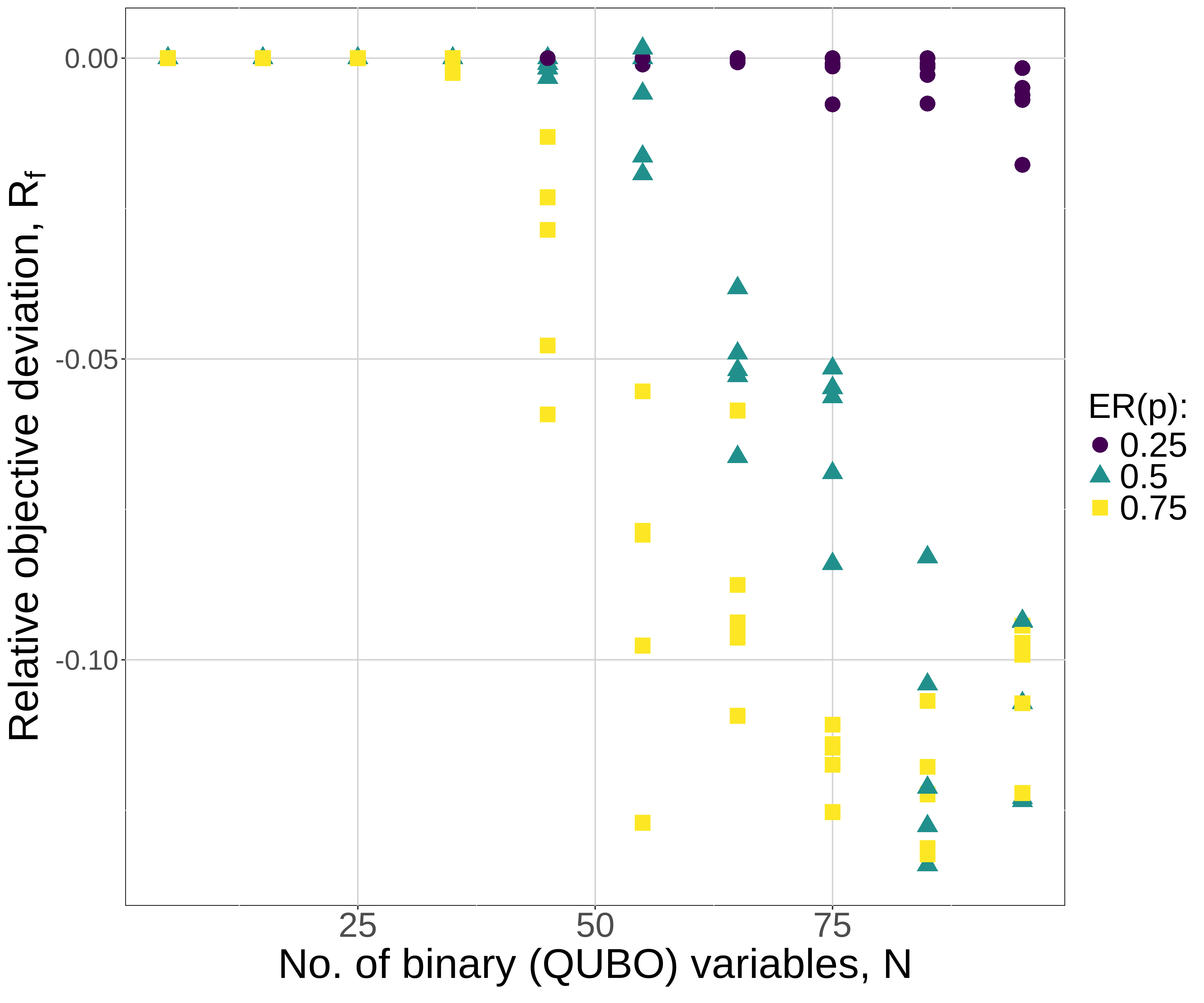}
	\caption{\Gls{Gurobi}-based heuristic with the timeout set after the
		\gls{D-Wave-OPT} runtime for the \gls{MaxCut} instances. Colors and shapes
		denote the respective instance generation parameter values $p$ of the
		Erd\H{o}s--R\'enyi model. On the average, and the number of nodes being
		equal, larger values of $p$ correspond to graphs with more
		edges. Objective deviations are calculated according to \zcref{eqn:rel-obj}, \ie, negative deviations indicate that the classical baseline yielded higher objective value. In this case, the figure implies that the classical heuristic outperforms the quantum-based approach essentially for all instances.\label{fig:MaxCut-gtimeouts}}
\end{figure}
\clearpage
\bibliographystyle{elsarticle-harv-sven}
\bibliography{references}